\documentclass{article}

\usepackage{amssymb,amsfonts,amsmath,latexsym,color,amscd} 
\usepackage[all,cmtip]{xy}
\usepackage{url}
\usepackage{mathrsfs}
\usepackage[dvips]{psfrag,graphicx} 
\usepackage{cases}

\newtheorem{theorem}{Theorem}[section]
\newtheorem{lemma}[theorem]{Lemma}
\newtheorem{corollary}[theorem]{Corollary}
\newtheorem{proposition}[theorem]{Proposition}

\newtheorem{definition}[theorem]{Definition}

\newenvironment{proof}{{\par\addvspace{0.1cm}\noindent \bf Proof. }}{\hfill$\Box$\par\medskip} 

\newtheorem{conjecture}[theorem]{Conjecture}
\newtheorem{remark}[theorem]{Remark}

\numberwithin{equation}{section}

\def\a{\alpha}
\def\e{\varepsilon}
\def\R{\Re\mathfrak{e} \,}
\def\I{\Im\mathfrak{m} \,}
\def\vect#1{\mbox{\boldmath $#1$}}  
\def\RR{\mathbb{R}}
\def\CC{\mathbb{C}}

\def\NN{\mathbb N}

\def\de{\delta}
\def\ka{\kappa}
\def\la{\lambda}
\def\si{\sigma}
\def\Ga{\Gamma}

\def\uon#1{\nu_{#1}}

\def\Rsub#1{\mbox{\rm R}_{#1}}
\def\Rd#1#2{\mbox{\rm R}_{#2}\,(#1)}
\def\Res{\mbox{\rm R}}
\def\dom{\Omega}
\def\pO{\partial\Omega}
\def\O{\Omega}
\def\ip#1#2{\langle #1, #2 \rangle}
\def\pd#1#2{\frac{\partial #1}{\partial #2}}

\def\resat#1{\underset{#1}{\mbox{\rm Res}}\,\,}
\def\eao{\stackrel{\rm{at}\> 0}=}

\def\bt{\mbox{\large $\cdot$}}

\def\Sc{{\rm Sc}} 
\def\Rm{\mbox{\rm Rm}} 

\def\dx{dv(x)}
\def\dy{dv(y)}

\def\Pf{\mbox{\rm Pf.}}
\def\M{M\"obius }

\def\jb#1{\textcolor[named]{Black}{{#1}}}

\begin{document}

\title{Residues of manifolds}

\author{Jun O'Hara\footnote{Supported by JSPS KAKENHI Grant Number 19K03462.}}
\maketitle

\begin{abstract}
The Riesz $z$-energy of a manifold $X$ is the integration of the distance between two points to the power $z$ over the product space $X\times X$. 
Considered as a function of a complex variable $z$, it can be generalized to a meromorphic function by analytic continuation, which we will call the meromorphic energy function of $X$. 
It has only simple poles at some negative integers. 
The residues of a manifold $X$ are the residues of the meromorphic energy function. 
For example, the volume and the Willmore energy for surfaces in $\RR^3$ can be obtained as residues. 
 
In this paper we first show the \M invariance of the residue at $z=-2\dim X$ of a closed submanifold or a compact body in a Euclidean space. 
We introduce the relative residues for compact bodies and weighted residues, and show that the scalar curvature and the mean curvature as well as the Euler characteristic of compact bodies of dimension less than $4$ can be expressed in terms of residues and local residues. 
We study the order of differentiaion of a local defining function of $X$ that is necessary to obtain the (global) residues when $X$ is a closed submanifold of a Euclidean space. 
We also show the inclusion-exclusion principle. 

Residues appear to be similar to quantities obtained by asymptotic expansion such as intrinsic volumes (Lipschitz-Killing curvatures), spectra of Laplacian, and the Graham-Witten energy. 
We show that residues are independent from them. 
Finally we introduce a \M invariant principal curvature energy for $4$-dimensional hypersurfaces in $\RR^5$, and express the Graham-Witten energy in terms of the residues, Weyl tensor, and this \M invariant principal curvature energy. 
\end{abstract}



\setcounter{tocdepth}{3}
\tableofcontents

\section{Introduction}
%
Let $X$ be a Riemannian manifold, $d(x,y)$ be the distance between two points $x$ and $y$ in $X$, and $dv$ be the volume element of $X$. If we consider 
\begin{equation}\label{IXz}
\iint_{X\times X}{d(x,y)}^z\,dv(x)dv(y),
\end{equation}
as a function of a complex variable $z$, it becomes holomorphic if $\R z$ is sufficiently large ($\R z>-\dim X$). 
If we extend the domain to the whole complex plane by analytic continuation, we obtain a meromorphic function with only simple poles at some negative integers. 
We call it the {\em meromorphic energy function} or {\em Brylinski beta function} or {\em regularized Riesz $z$-energy} of $X$ and denoet it by $B_X(z)$. We call the residues of it at a pole $z_0$ the {\em residues} of the manifold $X$ and denote it by $\Res_X(z_0)$. 
The subject of this paper is the information of manifolds obtained from the meromorphic energy functions, in particular their residues.  
In this paper we mainly deal with closed submanifolds of Euclidean space or {\em compact bodies} which are compact submanifolds of Euclidean space with the same dimension as the ambient space. 
In this case, we use the distance $|x-y|$ in the ambient space as $d(x,y)$. 
On the other hand, if $X$ is a compact connected Riemannian manifold we use the geodesic distance.

The source of this study is the energy of a knot $K$ given by 
\[
\iint_{(K\times K)\setminus\Delta_\e}\frac1{{|x-y|}^2}\,dxdy =
\frac{2L(K)}\e+E(K)+O(\e)
\]
(\cite{O1}), where $\Delta_\e=\{(x,y)\in\RR^3\times\RR^3\,:\,|x-y|\le\e\}$ $(\e>0)$ and $L(K)$ is the length of the knot. 
This is the Hadamard regularization of the divergent integral $\iint_{K\times K}|x-y|^{-2}dxdy$. 
That is, first we perform the integration on the complement of an $\e$-neighbourhood of the diagonal set where the integration diverges, next extend the result in a series in $\e$, and finally take its constant term called Hadamard's finite part of the integral.

The energy of a knot is motivated by Fukuhara and Sakuma's problem of finding a functional called energy on the space of knots that can produce a canonical shape of knot that minimizes the energy in each knot type.
Freedman, He, Wang showed that the energy $E(K)$ is invariant under \M transformations of $\RR^3\cup\{\infty\}$ as far as the image of the knot is compact\footnote{This is why $E(K)+4$ is sometimes called the {\em \M energy}.}, and using the \M invariance they showed that there is an energy minimizer in each {\sl prime} knot type (\cite{FHW}). 
This is how the \M invariance got involved in the study of knot energies. 
These were the starting points of geometric knot theory, where various energy functionals on the space of knots, and its higher dimensional analogues have been studied. 

On the other hand, Brylinski proposed another regularization of the divergent integral $\iint_{K\times K}|x-y|^{-2}dxdy$ using the analytic continuation (\cite{B}). 
Let $z$ be a complex variable, and consider 
\[
\iint_{K\times K}{|x-y|}^z\,dxdy. 
\]
It is well-defined on $\R z>-1$, where it becomes a holomorphic function of $z$. 
A meromorphic function can be obtained by extending the domain to the whole complex plane by means of analytic continuation. It has only simple poles at negative odd integers. 
Brylinski called it the {\em beta function of knot} since it is represented by Euler's beta function when the knot is a round circle. 
Let it be denoted by $B_K(z)$. Then the energy $E(K)$ is equal to $B_K(-2)$.
Fuller and Vemuri extended this to closed submanifolds $M$ of Euclidean space. In this case, $B_M(z)$ has possible simple poles at $z=-m-2j$ where $m=\dim M$ and $j=0,1,2, \dots$ (\cite{FV}).

As was noticed in \cite{OS2}, Hadamard regularization and regularization via analytic continuation yield essentially the same thing. 
Specifically, $B_X(z)$ has a pole at $z_0$ if and only if the $\log$ term appears in Hadamard regularization of $\iint_{X\times X}|x-y|^{z_0}\,dv(x)dv(y)$, and the coefficient of which coincides with the residue of $B_X$ at $z_0$. 

Brylinski's beta function was extended to compact bodies $\dom$ in \cite{OS2}. 
In this case $B_\dom(z)$ has possible simple poles at $z=-n,-n-1-2j$ $(n=\dim \dom, j=0,1,2,\dots)$. 
Applying Stokes's Theorem, $B_\dom(z)$ can be expressed as a double integral on the boundary with the integrand being the product of $|x-y|^{z+2}$ and the inner product $\ip{\nu_x}{\nu_y}$ of the outward unit normal vectors. 
Thus $B_\dom$ can be considered as a meromorphic energy function of $M=\pO$ with {\em weight} $\ip{\nu_x}{\nu_y}$.

The {\em $z$-energy} of $X$, denoted by $E_X(z)$, is defined to be Hadamard's finite part of $\iint_{X\times X}|x-y|^z\,dv(x)dv(y)$. 
It is equal to $B_X(z)$ if $z$ is not a pole of $B_X$ and to 
\[
\lim_{w\to z}\biggl(B_X(w)-\frac{{\rm Res}(B_X;z)}{w-z}\biggr)
\]
if $z$ is a pole. 
Both the energy and the residue can be expressed by the double integral on $M$ or $\pO $. 
The integrand is determined by the global data of $M$ or $\pO$ in the case of energies, and by local data such as curvatures, the second fundamental form and the derivatives in the case of residues. 

The \M invariance of the energy of knots $E(K)=B_K(-2)$ was generalized in \cite{OS2} and \cite{OS2'} to the \M invariance of the $(-2\dim X)$-energy of an odd dimensional closed submanifold of $\RR^n$ or an even dimensional compact body. 
The condition of the dimension was put so that $B_X(z)$ does not have a pole at $z=-2\dim X$.  

\medskip
In this paper we will mainly study the residues. 
These are what were missing from previous studies of knots and surface energies. 
We study the properties of residues and the relation with geometric quantities obtained by asymptotic expansion which have been studied in integral geometry and differential geometry. 

First we show that the \M invariance also holds for the residues at $z=-2\dim X$ by modifying the proofs in \cite{OS2,OS2'}. 
As a consequence, 
for an $m$ dimensional closed submanifold $M$, the Hadamard regularization of $\iint_{M\times M}|x-y|^{-2m}dv(x)dv(y)$ is given by 
\begin{equation}\label{eq_symmetry_intro}
\begin{array}{rcl}
\displaystyle \iint_{(M\times M)\setminus\Delta_\e}\frac1{{|x-y|}^{2m}}\,\dx \dy 
&=&\displaystyle \sum_{l=0}^{\lceil m/2 \rceil-1}a_{l}(M)\e^{-m+2l}  \\[5mm]
&&\displaystyle 
+\left\{
\begin{array}{ll}
{E}^{o}_M(-2m)+O(\e) & \hspace{0.4cm}\mbox{$m$ : odd}, \\[2mm]
\Res_M(-2m)\log\e+E_M^{e}(-2m)+O(\e) & \hspace{0.4cm}\mbox{$m$ : even},
\end{array}
\right.  
\end{array}
\end{equation}
and the \M invariance holds for the constant term ${E}^{o}_M$ if $m$ is odd, i.e., when the $\log$ term does not appear, although when the $\log$ term does appear, which happens if $m$ is even, the \M invariance does not hold for the constant term $E_M^{e}$ but for the coefficient of the $\log$ term $\Res_M$ (Theorem \ref{M-inv-res-closed}). 
The same thing happens for compact bodies, reversing even-odd (Theorem \ref{thm_M_inv_body}). 
Thus, the same phenomenon as in the studies of quantum field theory and theory of renormalized volumes is observed: the anomalies have the symmetry they break (cf. p7 of \cite{A}). 

We introduce two kinds of new residues. 
First, as was stated before, $B_\dom$ can be considered as a meromorphic energy function of $M=\pO$ with weight $\ip{\nu_x}{\nu_y}$. 
We can generalize it to a closed submanifold $M$ with arbitrary codimension by 
\jb{generalizing $\ip{\nu_x}{\nu_y}$ to $\ip{T_xM}{T_yM}$ using the Grassmann algebra (Definition \ref{nu-weighted}). }
We denote it by $B_{M,\nu}(z)$ and call it the {\em $\nu$-weighted meromorphic energy function.} 
$B_{M,\nu}(z)$ has only simple poles. 
The residues are called the {\em $\nu$-weighted residue} of $M$ and are denoted by $\Res_{M,\nu}(z_0)$.

The second is the {\em relative} version for compact bodies (Subsection \ref{subsect_relative}). The {\em meromorphic relative energy function} of $\Omega$, denoted by $B_{\dom,\pO}(z)$, can be defined by regularizing 
\[
\int_\dom\int_{\pO}|x-y|^z\,\dx\dy
\]
via analytic continuation in the same way. 
It has only simple poles at some negative integers. 
Let the residues of it be called the {\em relative residues}, denoted by $\Res_{\dom,\pO}(z_0)$. 
We show that the relative residues can be obtained by half the derivatives of the residues of the $\e$-parellel body of $\dom$ by $\e$ at $\e=0$, where the $\e$-parellel body is a body obtained by inflating $\dom$ outward by $\e$ (Subsection \ref{subsubsection_parallel_bodies}). 

The residues are obtained by integrating the local residues. 
The local residues can be expressed using the derivatives of a function defined on the tangent space such that $M$ or $\pO$ are locally expressed as a graph of it. 
A priori, necessary order of differentiation to express local residue at $z_0+2$ is greater than that at $z_0$ by $2$, although we conjecture that after integrating on a closed submanifold $M$ or $\pO$ to obtain the ``global residue'' the difference is equal to $1$. 
We prove it for some cases (Subsection \ref{subsub_order_diff}). 

As the first few examples of residues we obtain the volumes of $M$, $\dom$, and $\pO$, the total mean curvature, the Willmore energy 
if $M$ is a closed surface in $\RR^3$, 
the integral of the linear combination of the squared scalar curvature and the square norms of the Riemann and the Ricci curvature tensors if $M$ is a Riemannian manifold, and so on. 
They are same or similar to the first few examples of some known geometric quantities such as Lipschitz-Killing curvatures (or intrinsic volumes for convex bodies), the coefficients of Weyl's tube formula, spectra of the Laplacian, and Graham-Witten energy (Section \ref{sect_similar_q}). We will study the relationships and independence of these quantities. 

First, we show that $B_\dom(z)$ and $B_{\pO}(z)$ are independent as functionals from the space of compact bodies to the space of meromorphic functions (Subsection \ref{subsubsection_indep_B_dom_B_pO}). 

As for the curvatures, we show that the scalar curvature and the squared norm of the mean curvature vector can be expressed by a linear combination of the local residue and the $\nu$-weighted local residue (Subsection \ref{subsect_scalar_curv_residues}). 
Furthermore, if $M$ is of codimension $1$ then the mean curvature and the Laplacian of it can be expressed by the local relative residues (Subsection \ref{subsect_relative_mean_curv}). 

Let $\dom$ be a a convex body. 
The {\em intrinsic volumes} of $\dom$ are coefficients that appear in the series expansion in $\e$ of the volume of the $\e$-parallel body $\dom_\e$. 
They are essentially equal to the integral of the elementary symmetric polynomials of the principal curvatures of the boundary, which are called the {\em Lipschitz-Killing curvatures} in differential geometry. 
The latter formulation does not require that $\dom$ be convex. 
By Gauss-Bonnet-Chern theorem, the Lipschitz-Killing curvatures include the Euler characteristic. 
Hadwiger's theorem (Theorem \ref{Hadwiger_thm}) states that the intrinsic volumes form the basis of the space of valuations that satisfy some conditions containing the continuity with respect to the Hausdorff metric, where a valuation is a function on the space of convex bodies that satisfy $\mu(\emptyset)=0$ and the inclusion-exclusion principle: $\mu(K\cup L)=\mu(K)+\mu(L)-\mu(K\cap L)$ if $K,L$ and $K\cup L$ are convex. 
Thus the intrinsic volumes are important quantities in the theory of valuation that has been rapidly developing in recent years (cf. \cite{AF}). 

We show that the residues and the relative residues of compact bodies satisfy the inclusion-exclusion principle as far as $\dom_1, \dom_2, \dom_1\cap\dom_2$ and $\dom_1\cup\dom_2$ all satisfy some regularity condition (Subsection \ref{subsect_inclusion_exclusion}). We do not need the assumption of convexity here. 
The linear combination of the residues of $\dom$ and of $\pO$, and the relative residues of $\dom$ can express all the Lipschitz-Killing curvatures of $\dom$ when the dimension is $2$ or $3$. However, we conjecture that this is not possible as the dimension increases (Subsection \ref{subsubsection_Lipschitz-Killing_curv_intrinsic_volum}).

Graham-Witten energy, which we denote by $\mathcal{E}$ in this paper, was introduced in 1999 via the volume renormalization in the study the AdS/CFT correspondence. 
Here the renormalization was carried out by the Hadamard regularization. 
As in the case of energy and residue, 
if the $\log$ term does not appear in the series expansion then the constant term is conformally invariant, whereas if the $\log$ term does appear then the coefficient of it instead of the constant term is conformally invariant. 
The Graham-Witten energy $\mathcal{E}$ is a coefficient of this $\log$ term.
Therefore it is a conformal invariant for an even-dimensional closed submanifold. 
In addition, if the dimension is $2$, it coincides with the Willmore energy. 
These two properties are what some residues also have. 

We first assume that $M$ is a four dimensional closed submanifold of Euclidean space, and show that the Graham-Witten energy $\mathcal{E}(M)$, the residue $\Res_M(-8)$, and $\nu$-weighted residue $\Res_{M,\nu}(-8)$ are independent. 
Next, we give an ``optimal'' approximation of $\mathcal{E}(M)$ by a linear combination of $\Res_M(-8)$ and $\Res_{M,\nu}(-8)$, from a viewpoint of the order of differentiation of a function such that $M$ can be locally expressed as a graph of it (Subsection \ref{subsec_Graham-Witten}). 

Next we assume further that $M$ is a hypersurface. 
We consider the integral of fourth-order symmetric polynomials of the principal curvature $\ka_1,\dots,\ka_4$. By Gauss-Bonnet-Chern theorem, the integral of $\ka_1\ka_2\ka_3\ka_4$ is equal to the Euler characteristic of $M$ up to multiplication by a constant, which is a topological invariant. It is known that the integral of the square norm of the Weyl tensor is also conformal invariant. 
We show that there is only one more symmetric fourth-order polynomial of the principal curvatures so that the integral, which we denote by $Z(M)$, is \M invariant (Subsection \ref{appendix}). We express the Graham-Witten energy as a linear combination of $\Res_M(-8)$, $\Res_{M,\nu}(-8)$, the integral of the square norm of Weyl tensor, and $Z(M)$ (Subsection \ref{subsec_Graham-Witten}).

In comparison of some geometric quantities arizing from asymptotic expansion with the residues, we have so far observed the same phenomenon; the first few agree with the residues, but if we go further, they are both linear combination of the same terms, but with different coefficients.

\medskip
We use the Einstein convention, namely the summation is taken for all the indices that appear in pairs in each term, when there is no fear of confusion. 
We assume that manifolds are smooth in this paper for the sake of simplicity. 

\medskip
{\bf Acknowledgement}: 
The author thanks Professor Hiroshi Iriyeh for the question of the conformal invariance of the residues, Professor Kazuki Hiroe for helpful suggestions on complex analysis, and Professor Yuya Takeuchi for the explanation of Graham-Witten invariant. 
He also thanks Professors Kazuo Akutagawa and Takashi Kurose for helpful comments. 

\section{Preliminaries}

\subsection{Regularization of divergent integrals}\label{subsection_preliminaries_analysis}
We introduce some of the basics from previous studies mainly from \cite{B,FV,OS2}. 

First we introduce two kinds of regularization of the integral 
\[
I_\varphi(z)=\int_0^\infty t^z\varphi(t)\,dt \quad (z\in\CC),
\]
where $\varphi\colon\RR\to\RR$ is an integrable function with compact support which is smooth on $[0,\de]$ for some $\de>0$. 
Let $k$ be a natural number. Then, if $\R z>-1$ we have
\begin{equation}\label{Iphiz}
I_\varphi(z)=\int_0^{\de} t^{z}\left(\varphi(t)
-\sum_{j=0}^{k-1}\frac{{\varphi}^{(j)}(0)}{j!} t^{j}\right)dt
+\sum_{j=0}^{k-1}\frac{{\varphi}^{(j)}(0)\,\de^{\,z+j+1}}{j!\,(z+j+1)} 
+\int_\de^\infty t^z\varphi(t)\,dt.
\end{equation}
Note that the last term is an entire 
function of $z$. 
Since the integrand of the first term of the right hand side of \eqref{Iphiz} can be estimated by $t^{z+k}$, it is integrable on $[0,\delta]$ when $\R z>-k-1$. 
Therefore $I_\varphi(z)$ can be extended by analytic continuation to a meromorphic function defined on $\R z>-k-1$ with possible simple poles at $z=-1, \dots , -k$. 
As $k$ is arbitrary, it follows that $I_\varphi(z)$ can be extended to a meromorphic function on the whole complex plane $\CC$ with possible simple poles at negative integers with residue 
\begin{equation}\label{residue_I_phi}
\Res(I_\varphi;-\ell)=\frac{{\varphi}^{(\ell-1)}(0)}{(\ell-1)!}  \quad (\ell\in\NN).
\end{equation}

We remark that the residues appear in Hadamard regularization as coefficients of the series expansion: 
\[
\int_\e^\infty t^{-\ell}\varphi(t)\,dt
=\sum_{j=1}^{\ell-1}\frac1{\ell-j}\cdot\frac{{\varphi}^{(j-1)}(0)}{(j-1)!}\,\e^{\,-\ell+j}
-\frac{{\varphi}^{(\ell-1)}(0)}{(\ell-1)!}\log\e +C +O(\e)
\qquad (\ell\in\NN).  
\]
$I_\varphi(z)$ has a pole at $z=-k$ if and only if the extension of $\int_\e^\infty t^{-k}\varphi(t)\,dt$ in a series in $\e$ has a $\log$ term, and the coefficient of the $\log$ term is equal to minus the residue of $I_\varphi(z)$ at $z=-k$. 
The constant $C$ is called Hadamard's finite part, denoted by $\Pf \int_0^\infty t^{-k}\varphi(t)\,dt$, which satisfies
\[
\Pf \int_0^\infty t^{-k}\varphi(t)\,dt=\lim_{z\to-k}\left(I_\varphi(z)-\frac{\mbox{\rm Res}(I_\varphi;-k)}{z+k}\right).
\]
Thus we obtain the same quantities from two kinds of regularization above mentioned. 
The reader is referred to, for example, \cite{GS} Ch.1, 3.2 for the details. 
In this paper we use the regularization via analytic continuation instead of Hadamard regularization, which has been used in the study of energy of knots and surfaces (cf. \cite{O1},\cite{AS}). 

In the following two subsections we introduce two methods to carry out the regularization. 
We remark that when the dimension is equal to $1$, i.e., in the case of knots, one can use expression in terms of the arc-length as Brylinski did in \cite{B}. 

\subsection{Regularization using small extrinsic or geodesic balls}\label{subsubsect_extrinsic_balls}
%
Let $M$ be an $m$ dimensional closed manifold, either a submanifold of $\RR^n$ or a connected Riemannian manifold. 
We use the extrinsic distance $|x-y|$ in the former case, where $|\,\cdot\,|$ is the standard norm of $\RR^n$, and the geodesic distance in the latter case. 

Let $\lambda$ be a smooth real-valued function on $M\times M$, and $x$ be a point in $M$. 
We are interested in 
\begin{eqnarray}
I_{M,\lambda,x}(z)&=&\displaystyle \int_M{d(x,y)}^z\,\lambda(x,y)\,\dy , \label{IMrhox} \\ 
I_{M,\lambda}(z)&=&\displaystyle \iint_{M\times M}{d(x,y)}^z\,\lambda(x,y)\,\dx \dy 
=\int_M I_{M,\lambda,x}(z)\,\dx . \label{IMrho}
\end{eqnarray}
The integrals converge if $z$ satisfies $\R z>-m$, in which case $I_{M,\lambda,x}(z)$ and $I_{M,\lambda}(z)$ are holomorphic functions of $z$. 

Put 
\[
\begin{array}{rcl}
B_t(x)&=&\{y\in M\,:\,d(x,y)\le t\}, \\[2mm]
\Delta_t(M\times M)&=&\{(x,y)\in M\times M\,:\,d(x,y)\le t\}
\end{array}
\]
for $t\ge0$ and put 
\begin{eqnarray}
\psi_{\lambda,x}(t)&=&\displaystyle \int_{B_t(x)} \lambda(x,y)\,\dy  \qquad (x\in M), \label{psiMrhox} \\[1mm]
\psi_{\lambda}(t)&=&\displaystyle \iint_{\Delta_t(M\times M)} \lambda(x,y)\,\dx \dy . \notag 
\end{eqnarray}
The coarea formula implies that $I_{M,\lambda,x}(z)$ and $I_{M,\lambda}(z)$ 
can be expressed 
by 
\begin{equation}
\begin{array}{rcl}
I_{M,\lambda,x}(z)&=&\displaystyle \int_0^{\infty} t^z \, \psi_{\lambda,x}'(t) \,dt,\\[4mm]
I_{M,\lambda}(z)&=&\displaystyle \int_0^{\infty} t^z \, \psi_{\lambda}'(t) \,dt.
\end{array}
\notag
\end{equation}
We remark that $\psi_{\lambda}(t)$ has derivative almost everywhere. 
This was shown in Proposition 3.3 of \cite{OS2} when $M$ is an Euclidean submanifold.  The Riemannian case can be shown similarly.

\begin{lemma}\label{even_odd}
{\rm (\cite{G,OS2,Onew})} As above, 
\begin{enumerate}
\item 
There exists $d_0>0$ such that, for each $x\in M$ the function $\psi_{\lambda,x}(t)$ extends to a smooth function defined on $(-d_0,d_0)$, which we denote by the same symbol $\psi_{\lambda,x}$, such that $\psi_{\lambda,x}(-t)=(-1)^m\psi_{\lambda,x}(t)$. 
\item 
$\psi_{\lambda,x}'(t)$ can be expressed as 
\begin{equation}
\psi_{\lambda,x}'(t)
=t^{m-1} \, \overline{\varphi}_{\lambda,x}(t) 
\notag
\end{equation}
for some smooth function $\overline{\varphi}_{\lambda,x}$ with 
\begin{equation}
\overline{\varphi}_{\lambda,x}^{{}\>\,(2j+1)}(0)=0
\>\>\>(j=0,1,2,\dots). \nonumber
\end{equation}

\item 
When $\lambda(x,y)\equiv1$ we have 
\[
\overline{\varphi}_{1,x}(0)=o_{m-1}, 
\]
where $o_k$ is the volume of the unit $k$-sphere in $\RR^{k+1}$. 

\item 
Suppose $(\lambda_i)_{i=1}^\infty$ is a sequence of smooth functions with derivatives of all orders converging uniformly to $0$. 
Then $\psi_{\lambda_i,x}$ and its derivatives converge uniformly to $0$. 
\end{enumerate}
\end{lemma}

When $M$ is a closed submanifold of $\RR^n$ the statements were shown in Proposition 3.1 and Corollary 3.2 of \cite{OS2}, and Proposition 2.3 of \cite{Onew}. 
We remark that even if we substitute $\la_\nu(x,y)$  
in the proof of Proposition 2.3 of \cite{Onew} by a general function $\lambda$, 
the proof works as well.  
When $M$ is a closed connected Riemannian manifold (2) and (3) were shown in Theorem 3.1 and its proof of \cite{G}. (1) and (4) can be shown by the same argument as in the case of a submanifold of $\RR^n$.

\medskip
Thus $I_{M,\lambda,x}(z)$ and $I_{M,\lambda}(z)$ are reduced to the integral introduced in the previous subsection, and therefore they can be extended to meromorphic functions on $\CC$ with only simple poles. 

\begin{definition}\label{def_B_M_etc} \rm 
We denote them by $B_{M,\lambda,x}^{{\rm loc}}(z)$ and $B_{M,\lambda}(z)$ and call them the {\em $\lambda$-weighted meromorphic potential function at $x$} and the {\em $\lambda$-weighted meromorphic energy function}. 
They have residues at $z=-m-2j$ $(j=0,1,2,\dots)$, which we denote by $\Rsub{M,\lambda,x}^{\rm loc}(-m-2j)$ and $\Rsub{M,\lambda}(-m-2j)$ and call them the {\em $\lambda$-weighted (local) residue of $M$ (at $x$) at $z=-m-2j$}. 

In particular, when $\lambda\equiv1$ we drop the letter $\lambda$ off. 
In this case, $B_{M}(z)$ is called {\em Brylinski's beta function of $M$}.\footnote{ We use the same symbol $B$ for balls and meromorphic functions. The author hopes it does not cause any confusion.} 
\end{definition}
$B_{M}(z)$ was first introduced for knots by Brylinski \cite{B} and for closed submanifolds of $\RR^n$ by Fuller and Vemuri \cite{FV}. 

\smallskip
By \eqref{IMrho}, we have 
\begin{eqnarray}
B_{M,\lambda}(z)&=&\displaystyle \int_M B_{M,\lambda,x}^{{\rm loc}}(z)\,\dx  \notag \label{BMintBMloc} 
\end{eqnarray}
if $z$ is not a pole of $B_{M,\lambda,x}^{{\rm loc}}$, whereas by \eqref{residue_I_phi} the residues are given by  
\begin{equation}\label{RsubMrhoddtpsix}
\begin{array}{rcl}
\Rsub{M,\lambda,x}^{\rm loc}(-m-2j)&=&\displaystyle \frac1{(m+2j-1)!}\left.\frac{d^{\,m+2j}}{dt^{\,m+2j}} \,\psi_{\lambda,x}(t)\right|_{t=0}\,, \\[4mm]
\Rsub{M,\lambda}(-m-2j)
&=&\displaystyle \int_M \Rsub{M,\lambda,x}^{\rm loc}(-m-2j)\,\dx,  
\end{array}
\end{equation}
which were shown in Proposition 3.4 of \cite{OS2} when $M$ is a closed submanifold of $\RR^n$. 

\medskip
The regularization of potential and energy is given by 
\begin{eqnarray}
V_{M,\lambda,x}(z)&=&\displaystyle \left\{\begin{array}{ll}
\displaystyle B_{M,\lambda,x}^{\rm loc}(z) & \hspace{1cm} \mbox{if $z$ is not a pole of $B_{M,\lambda,x}^{\rm loc}$},\\[2mm]
\displaystyle \lim_{w\to z}\left(B_{M,\lambda,x}^{\rm loc}(w)-\frac{\Rsub{M,\lambda,x}^{\rm loc}(z)}{w-z}\right) & \hspace{1cm} \mbox{if $z$ is a pole of $B_{M,\lambda,x}^{\rm loc}$}, 
\end{array}\right. \notag \label{regularization_of_potential} \\[2mm]
E_{M,\lambda}(z)&=&\displaystyle \left\{\begin{array}{ll}
\displaystyle B_{M,\lambda}(z) & \hspace{1cm} \mbox{if $z$ is not a pole of $B_{M,\lambda}$},\\[2mm]
\displaystyle \lim_{w\to z}\left(B_{M,\lambda}(w)-\frac{\Rsub{M,\lambda}(z)}{w-z}\right) & \hspace{1cm} \mbox{if $z$ is a pole of $B_{M,\lambda}$}
\end{array}
\right. \label{def_regularized_energy}
\end{eqnarray}
cf. Proposition 3.6 of \cite{OS2} in the case when $M$ is a closed submanifold of $\RR^n$. 
We remark that if $M$ is not connected, the residues are the sums of the residues of the components, but this does not hold for the energies. 

\begin{definition} \rm 
We call $V_{M,\lambda,x}$ and $E_{M,\lambda}$ the {\em $\lambda$-weighted regularized potential} and {\em energy}  
respectively. 
\end{definition}

\smallskip
The ball $B_t(x)$ is called the {\em extrinsic ball} when $M$ is an Euclidean submanifold and the {\em geodesic ball} when $M$ is a Riemannian manifold. 
The series expansion of the volume $\psi_x(t)$ of an extrinsic ball was given in \cite{KP} and that of a geodesic ball was given in \cite{G}, both of which will appear later in \eqref{KP-expansion} and \eqref{Gray_volume_expansion_intrinsic}.

\subsection{Regularization using graphs for Euclidean submanifolds}\label{method_graph}
%
%
When $M$ is a submanifold of $\RR^n$ one can use normal coordinates to carry out reguralization, namely we express $M$ locally as a graph of a function on each tangent space and use partial derivatives of the defining function to give geometric quantities. 
This method was used by Auckly and Sadun \cite{AS} to define a \M invariant energy for surfaces using Hadamard regularization, and by Fuller and Vemuri \cite{FV}. 
This method is more suitable for calculation of residues in various examples. 

Suppose $M$ is locally expressed as a graph of a smooth function $f$ defined on an $m$-ball with center $0$ and radius $\rho$ in the tangent space $T_xM$ of $M$ at $x$ for some positive number $\rho$; 
 $f\colon T_xM\supset B_\rho^m(0)\to(T_xM)^\perp\cong\RR^{n-m}$, which vanishes to second order at $0$. 
Throughout the paper we agree that we use an orthonormal basis $\{\partial/\partial u_i\}_{1\le i\le m}$ for $T_xM$. 
Furthermore, when $M$ is a hypersurface we assume that each $\partial/\partial u_i$ belongs to a principal direction of $M$ at $x$. 

Put $V=f(B_\rho^m(0))$. Then 
\begin{equation}\label{I_M_lambda_x_z_graph}
I_{M,\lambda,x}(z)=\int_V{|y-x|}^z\lambda(x,y)\,\dy +\int_{M\setminus V}{|y-x|}^z\lambda(x,y)\,\dy . \notag
\end{equation}
The second term above is an entire function of $z$, hence it does not play a role in the process of regularization. 
The first term is given by 
\begin{equation}\label{BMx_graph_1}
\int_V{|y-x|}^z\lambda(x,y)\,\dy 
=
\int_{B_\rho^m(0)}{\left({|u|}^2+{|f(u)|}^2\right)}^{z/2}\lambda\bigl(0,(u,f(u))\bigr)A_f(u)\,du, 
\end{equation}
where $A_f(u)$ 
is the ``area density'' of the graph of $f$; 
\begin{equation}\label{def_A_f}
A_f(u)=\sqrt{\det\left(\delta_{ij}+\left\langle \pd{f}{u_i}, \pd{f}{u_j}\right\rangle\right)}\,. 
\end{equation}
In particular, when $M$ is a hypersurface, $A_f(u)=\sqrt{1+|\nabla f(u)|^2}$. 
Putting $r=|u|$ and $u=rw$, $w\in S^{m-1}$, \eqref{BMx_graph_1} can be expressed by 
\begin{eqnarray}
&&\int_0^\rho dr\int_{S^{m-1}}r^z {\biggl(1+\frac{{|f(rw)|}^2}{r^2}\biggr)}^{z/2} \lambda\bigl(0,(rw,f(rw))\bigr)A_f(rw)\,r^{m-1}\,dv(w) \notag \\[1mm] 
&&=\int_0^\rho r^{z+m-1}\bigg[\int_{S^{m-1}} {\biggl(1+\frac{{|f(rw)|}^2}{r^2}\biggr)}^{z/2} \lambda\bigl(0,(rw,f(rw))\bigr)A_f(rw)\,dv(w)\bigg]dr. \label{BMlambdax_1st_term}
\end{eqnarray}
Put 
\begin{equation}\label{def_xi}
\xi_\la(r,w,z)={\biggl(1+\frac{{|f(rw)|}^2}{r^2}\biggr)}^{z/2} \lambda\bigl(0,(rw,f(rw))\bigr)A_f(rw), \quad
\Xi(r,z)=\int_{S^{m-1}} \xi_\la(r,w,z)\,dv(w).
\end{equation}
Lemma 3.2 of \cite{FV}, which follows from Malgrange's preparation theorem, implies that 
\[
(0,\rho]\times S^{m-1}\times \CC \ni (r,w,z) \mapsto {\biggl(1+\frac{{|f(rw)|}^2}{r^2}\biggr)}^{z/2} \in \CC 
\]
can be extended to a smooth function on $\RR\times S^{m-1}\times \CC$. 
By \cite{blowup} 
\[
[-\rho,\rho]\times S^{m-1}\ni (r,w)\mapsto \lambda\bigl(0,(rw,f(rw))\bigr)A_f(rw) \in \RR
\]
is smooth. 
Therefore $\xi_\la(r,w,z)$ is smooth on $[-\rho,\rho]\times S^{m-1}\times \CC$. 
Since $\xi_\la(-r,-w,z)=\xi_\la(r,w,z)$, $\Xi(\,\bt\,,z)$ is an even function. 
Moreover, since $\xi_\la(r,w,z)$ is an entire function of $z$ for each fixed $(r,w)$, so is $\Xi_\la(r,z)$  
for each $r$ by Lemma 3.1 of \cite{FV}. 

Now \eqref{BMlambdax_1st_term} can be regularized by a similar\footnote{$\Xi$ depends also on $z$, which is different from $\varphi$ in \eqref{Iphiz}.} method to that of \eqref{Iphiz}:
\begin{eqnarray}
\eqref{BMlambdax_1st_term}&=& \displaystyle 
\int_0^{\rho} r^{z+m-1}\left(\Xi(r,z)
-\sum_{j=0}^{k-1}\frac{1}{j!} \, \frac{\partial^j\Xi}{\partial r^j}\,(0,z) \,r^{j}\right)dr  \label{IFVeta} 
%
+\sum_{j=0}^{k-1}\frac{{\rho}^{\,z+m+j}}{j!\,(z+m+j)}\, \frac{\partial^j\Xi}{\partial r^j}\,(0,z),
\notag
\end{eqnarray}
where $k$ is an arbitrary natural number. 
Since the integrand of the first term can be estimated by $r^{z+m+k-1}$, it is integrable on $[0,\rho]$ when $\R z >-m-k$. 
Therefore \eqref{BMlambdax_1st_term} can be extended by analytic continuation to a meromorphic function defined on $\R z>-m-k$ with possible simple poles at $z=-m, \dots , -m-k+1$. 
Since $k$ is arbitrary and $\Xi_\la(\,\bt\,,z)$ is an even function, $I_{M,\la,x}(z)$ can be extended to a meromorphic function on $\CC$ with possible simple poles at $z=-m,-m-2,-m-4,\dots$ with residue 
\[
\resat{z=-m-2j}
I_{M,\la,x}(z)=\frac1{(2j)!}\frac{\partial^{2j}\Xi}{\partial r^{2j}}\,(0,-m-2j)  \quad (j\in\NN\cup\{0\}).
\]
Smoothness of $\xi_\la$ implies that if we put $a_{\la,2j}(-m-2j;w)$ to be the coefficient of $r^{2j}$ of the Maclaurin expansion of $\xi_\la(r,w,z)$ in $r$, i.e. 
\begin{equation}\label{def_a_la_k}
a_{\la,2j}(-m-2j;w)=\frac1{(2j)!}\frac{\partial^{2j}\xi_\la}{\partial r^{2j}}\,(0,w,-m-2j), 
\end{equation}
the local residue is given by 
\begin{equation}\label{local_residue_graph_FV}
\Res_{M,\lambda,x}^{\rm loc}(-m-2j)=\int_{S^{m-1}}a_{\la,2j}(-m-2j;w)\,dv(w). \quad
\end{equation}

\noindent
The above argument for $\la\equiv1$ can be found in Theorem 3.3 of \cite{FV}.

\subsection{Curvatures of a graph in terms of the defining function}\label{preliminary_calculation_graph}
We summarize the formulas for curvatures that are needed when studying the residues by the methods introduced in Subsection \ref{method_graph}. 
Some of the content of this subsection can be found on the web or in some literatures. 

In what follows we use the Einstein convention, namely the summation is taken for all the indices that appear in pairs (which also includes $\partial_j^{\,2}$) in each term. 

Let $M=(M,g)$ be an $m$-dimensional closed connected Riemannian manifold. 
Put $G=(g_{ij})$ and $G^{-1}=(g^{ij})$. 
Let $R_{ijk}^{\phantom{ijj}l}$ be the Riemann curvature tensor, $R_{ijkl}=g_{\a l}R_{ijk}^{\phantom{ijj}\a}$, $R_{jk}=R_{jkl}^{\phantom{jkl}l}$ the Ricci curvature tensor, and $\Sc=g^{jk}R_{jk}$ the scalar curvature. 
The norms of the Riemann and Ricci curvature tensors are given by 
\begin{eqnarray}
{|\Rm|}^2&=&\displaystyle 
g^{\a\a'}g^{\beta\beta'}g^{\gamma\gamma'}g^{\delta\delta'}\,R_{\a\beta\gamma\delta}R_{\a'\beta'\gamma'\delta'}, \label{Rm_square}\\
{|\mbox{\rm Ric}|}^2&=&\displaystyle 
g^{\a\a'}g^{\beta\beta'}R_{\a\beta}R_{\a'\beta'}. \label{Ricci_sq} \notag
\end{eqnarray}

Assume that $M$ is an $m$ dimensional 
submanifold of $\RR^n$. 
Let $\nabla$ be the Levi-Civita connection, $h$ the second fundamental form $h(X,Y)={(\nabla_X Y)}^\perp$, 
$\Vert h\Vert$ the Hilbert-Schmidt norm : ${\Vert h\Vert}^2=\sum_{i,j}{\vert h(\vect e_i,\vect e_j)\vert}^2$, and $H$ the mean curvature vector $H=\mbox{\rm trace}\, h=\sum_ih(\vect e_i, \vect e_i)$, where $\{\vect e_i\}$ is an orthonormal basis of $T_xM$. When $M$ is a hypersurface $H=\sum_i\kappa_i$, where $\kappa_i$ is a principal curvature. 

Let $x$ be a point on $M$. 
We assume $x=0$, $T_xM=\RR^m$, and $M$ is locally expressed near $x=0$ as a graph of a function $f\colon\RR^m\to\RR^{n-m}$. 
A point $y\in M$ near $x=0$ can be expressed by $y(u)=\bigl(u, f(u)\bigr)$. 
Put $\partial_i={\partial}/{\partial u_i}$, $f_i=\partial_i f$, and $f_{ij}=\partial_i\partial_j f$ etc. 
Then 
\[
g_{ij}
=\left\langle\pd{y}{u_i}, \pd{y}{u_j}\right\rangle
=\delta_{ij}+\ip{f_i}{f_j}, 
\quad (1\le i,j\le m),
\]
where $\delta_{ij}$ is the Kronecker delta. 
Put $g=\det G$. Then $A_f=\sqrt{g}$. 

Let $\eao $ denote the equality at the origin in what follows. We have $f_j\eao 0, \, g_{ij}\eao \delta_{ij}, \, g^{ij}\eao \delta_{ij}.$ 
The Christoffel symbols and the Ricci curvature tensor are given by 
\begin{eqnarray}
\Gamma_{ij}^k&=&\displaystyle \frac12 \, g^{k\ell}\left(\partial_ig_{j\ell}+\partial_jg_{i\ell}-\partial_\ell g_{ij}\right)
\notag \\
&=&g^{k\ell}\ip{f_\ell}{f_{ij}} \notag \\
&\eao&  0, \notag \\[2mm] 
R_{ij}&=& \partial_\ell\Gamma_{ij}^\ell-\partial_j\Gamma_{i\ell}^\ell
+\Gamma_{ij}^k\Gamma_{\ell k}^\ell+\Gamma_{i\ell}^k\Gamma_{jk}^\ell  \notag \\[1mm]
&=& \big(\partial_\ell g^{\ell k}\big)\ip{f_{k}}{f_{ij}}-\big(\partial_j g^{\ell k}\big)\ip{f_{k}}{f_{i\ell}} 
+g^{pk}\ip{f_{k}}{f_{ij}}g^{\ell h}\ip{f_{h}}{f_{\ell p}}
- g^{pk}\ip{f_{k}}{f_{i\ell}}g^{\ell h}\ip{f_{h}}{f_{jp}}  
\label{R_ij_gen_1} \\
&&+g^{\ell k}\ip{f_{k\ell}}{f_{ij}}-g^{\ell k}\ip{f_{kj}}{f_{i\ell}}, \label{R_ij_gen_2} 
%
%
\end{eqnarray}
where the indices $i,j,k,\ell$ are taken from $1,\dots, m$, which implies that the scalar curvature is given by 
\begin{equation}\label{Sc_at_0}
\Sc\eao
\ip{f_{jj}}{f_{kk}}-\ip{f_{jk}}{f_{jk}}
=\sum_{j\ne k}\left(\ip{f_{jj}}{f_{kk}}-\ip{f_{jk}}{f_{jk}}\right).
\end{equation}

\begin{lemma}\label{lem_la_A_f_Delta_Sc_H}
Let $M$ be an $m$ dimensional submanifold of $\RR^n$ and $x$ be a point on $M$. 
Suppose $x=0$, $T_xM=\RR^m$, and $M$ is expressed near $x=0$ as a graph of a function $f\colon\RR^m\to\RR^{n-m}$. 
Then at the origin we have 
\begin{eqnarray}
\Delta\Sc(0)&=&\displaystyle -4\ip{f_{\ell\ell}}{f_{ij}}\,\ip{f_{ik}}{f_{jk}}
-2\ip{f_{\ell\ell}}{f_{ij}}\,\ip{f_{kk}}{f_{ij}} \notag \label{Delta_Sc_at_0-1} \\ 
&&\displaystyle +2\big(\ip{f_{ij}}{f_{kl}}\big)^2 +2\big(\ip{f_{ik}}{f_{jk}}\big)^2
+2\ip{f_{ik}}{f_{i\ell}}\,\ip{f_{jk}}{f_{j\ell}} \label{Delta_Sc_at_0-2} \notag \\[1mm] 
%
%
&&\displaystyle +4\sum_{j\ne k}\ip{f_{jjk}}{f_{kkk}}
-4\sum_{j\ne k}\ip{f_{jkk}}{f_{jkk}}
+4\sum_{k<\ell, j\ne k,\ell}
\ip{f_{jkk}}{f_{j\ell\ell}}
-12\sum_{j<k<\ell}
\ip{f_{jk\ell}}{f_{jk\ell}}   \label{Delta_Sc_at_0-3} \\[2mm]
&&\displaystyle +\sum_{j\ne k\ne l\ne j}
\big(
2\ip{f_{jj}}{f_{jjkk}}
+2\ip{f_{jj}}{f_{kkkk}}
+4\ip{f_{jj}}{f_{kkll}}
-4\ip{f_{jk}}{f_{jjjk}}
-4\ip{f_{jk}}{f_{jkll}}
\big), \qquad{} \label{f****-terms_Delta_Sc} \\[2mm]
\Delta {|H|}^2(0)&=&\displaystyle -2\ip{f_{ii}}{f_{k\ell}}\,\ip{f_{jj}}{f_{k\ell}}
-4\ip{f_{i\ell}}{f_{j\ell}}\,\ip{f_{ij}}{f_{kk}} \notag \\[1mm]
&&\displaystyle +2\sum_i\ip{f_{iii}}{f_{iii}}
+4\sum_{j\ne k}\ip{f_{jjk}}{f_{kkk}}
+2\sum_{j\ne k}\ip{f_{jkk}}{f_{jkk}}
+4\sum_{k<\ell, j\ne k,\ell}
\ip{f_{jkk}}{f_{j\ell\ell}}   \label{Delta_H_at_0_third} \\[2mm]
&&\displaystyle +2\sum_{j\ne k\ne l\ne j}\big(
\ip{f_{jj}}{f_{jjjj}}
+2\ip{f_{jj}}{f_{jjkk}}
+\ip{f_{jj}}{f_{kkkk}}
+2\ip{f_{jj}}{f_{kkll}}
\big).   \label{f****-terms_Delta_H}
%
\end{eqnarray}
\end{lemma}

\begin{proof}
(i) First we show 
\begin{equation}\label{Delta_Sc_at_0_first}
\begin{array}{rcl}
\Delta\Sc&\eao&\displaystyle -4\ip{f_{\ell\ell}}{f_{ij}}\,\ip{f_{ik}}{f_{jk}}
-2\ip{f_{\ell\ell}}{f_{ij}}\,\ip{f_{kk}}{f_{ij}}  \\[1mm]
&&\displaystyle +2\big(\ip{f_{ij}}{f_{kl}}\big)^2 +2\big(\ip{f_{ik}}{f_{jk}}\big)^2
+2\ip{f_{ik}}{f_{i\ell}}\,\ip{f_{jk}}{f_{j\ell}}  \\[1mm] 
&&\displaystyle +\partial_k^{\,2}\ip{f_{jj}}{f_{\ell\ell}}-\partial_k^{\,2}\ip{f_{j\ell}}{f_{j\ell}}. 
\end{array}
\end{equation}
Applying \eqref{R_ij_gen_1}, \eqref{R_ij_gen_2}, and 
\begin{equation}\label{Laplacian_at_0}
\begin{array}{rcl}
\Delta \eta&=&\displaystyle \frac1{\sqrt g}\,\partial_j\big(g^{jk}\sqrt{g}\,\partial_k\eta \big) 
\eao \displaystyle  
\partial_j^{\,2}\eta 
\end{array}
\notag
\end{equation}
to $\Sc=g^{ij}R_{ij}$, we obtain 
\begin{equation}\label{Delta_Sc_at_0}
\Delta\Sc(0)=
\big(\partial_k^{\,2}g^{ij}\big)R_{ij}
+2
\big(\partial_k g^{ij}\big)\big(\partial_k R_{ij}\big)
+
g^{ij}\big(\partial_k^{\,2}R_{ij}\big).
\end{equation}
Taking partial derivatives of $G^{-1}G=I$ we have 
\begin{equation}\label{partial_g^ij}
\partial_k g^{ij} \eao 0, \>\> 
\partial_k\partial_\ell\, g^{ij} \eao -\big(\ip{f_{ik}}{f_{j\ell}}+\ip{f_{i\ell}}{f_{jk}}\big). 
\end{equation}
Substituting \eqref{partial_g^ij} to \eqref{Delta_Sc_at_0} we have 
\begin{eqnarray}\label{}
\displaystyle 
\big(\partial_k^{\,2}g^{ij}\big)R_{ij} &\eao& \displaystyle 
-2\ip{f_{\ell\ell}}{f_{ij}}\,\ip{f_{ik}}{f_{jk}}
+2\big(\ip{f_{ik}}{f_{jk}}\big)^2, \label{computation_Delta_Sc-1} \notag \\[3mm]
\displaystyle 
\big(\partial_k g^{ij}\big)\big(\partial_k R_{ij}\big) &\eao& 0, \label{computation_Delta_Sc-2} \notag \\[3mm]
\displaystyle 
g^{ij}\big(\partial_k^{\,2}R_{ij}\big) &\eao& \displaystyle -2\ip{f_{\ell\ell}}{f_{ij}}\,\ip{f_{ik}}{f_{jk}}
-2\ip{f_{\ell\ell}}{f_{ij}}\,\ip{f_{kk}}{f_{ij}} 
%
+2\big(\ip{f_{ij}}{f_{kl}}\big)^2 
+2\ip{f_{ik}}{f_{i\ell}}\,\ip{f_{jk}}{f_{j\ell}} \hspace{1.0cm}{} \label{computation_Delta_Sc-3} \notag \\[2mm]
&&\displaystyle 
+\partial_k^{\,2}\ip{f_{ii}}{f_{\ell\ell}}-\partial_k^{\,2}\ip{f_{i\ell}}{f_{i\ell}}, \label{computation_Delta_Sc-4}
\end{eqnarray}
which implies \eqref{Delta_Sc_at_0_first}. 

Next we compute the last two terms of \eqref{Delta_Sc_at_0_first}
\begin{equation}\label{Delta_Sc0_c_d_terms}
\sum_k\sum_{j,\ell }\big(\partial_k^{\,2}\ip{f_{jj}}{f_{\ell \ell }}-\partial_k^{\,2}\ip{f_{j\ell }}{f_{j\ell }}\big)
=
2\sum_k\sum_{j\ne \ell }\bigl(f_{jj}f_{\ell \ell kk}+(f_{jjk}f_{\ell \ell k}-f_{j\ell k}f_{j\ell k})\bigr). 
\end{equation}
By taking the sum of the terms in the cases of $k=j\ne l$, $j\ne l=k$, and $k\ne j\ne l\ne k$ of the right hand side of \eqref{Delta_Sc0_c_d_terms} 
we obtain \eqref{Delta_Sc_at_0-3} and \eqref{f****-terms_Delta_Sc}.

\smallskip
(ii) First we show 
\begin{equation}\label{Delta_H_at_0}
\Delta {|H|}^2\eao\displaystyle -2\ip{f_{ii}}{f_{k\ell}}\,\ip{f_{jj}}{f_{k\ell}}
-4\ip{f_{i\ell}}{f_{j\ell}}\,\ip{f_{ij}}{f_{kk}} 
+2\ip{f_{iik}}{f_{jjk}}
+2\ip{f_{kk}}{f_{ii\ell\ell}}.  
\end{equation}
To compute the mean curvature vector $H$, we agree that indices $i,j,k,\ell,p,\dots\in\{1,\dots,m\}$ are for $M$ and $\sigma, \tau, \upsilon,\dots\in\{1,\dots,n\}$ are for $\RR^n$. 
Recall that a point of $M$ near $x=0$ is expressed by $y(u)=(u,f(u))$ $(u\in\RR^m)$. Put 
\[
B_i^\si=\pd{y^\si}{u^i}=\left\{
\begin{array}{ll}
\delta^\si_i & \quad \si\le m, \\
\partial_i f^\si & \quad \si>m,
\end{array}
\right.
\]
where $f(u)=(f^{m+1}(u),\dots,f^{n}(u))$. 
Then the Euler-Schouten tensor is given by 
\[
H_{ij}^\si=\partial_iB_j^\si-\Ga_{ij}^kB_k^\si+B_i^\tau\Ga_{\tau\upsilon}^\si B_j^\upsilon,
\]
where $\Ga_{\tau\upsilon}^\si$ are the Christoffel symbols of $\RR^n$, and hence $\Ga_{\tau\upsilon}^\si=0$. 
Therefore we have 
\[
H_{ij}^\si=\left\{
\begin{array}{ll}
-g^{kp}\ip{f_{p}}{f_{ij}}\delta_k^\si=-g^{\si p}\ip{f_{p}}{f_{ij}} & \quad \si\le m, \\[2mm]
f_{ij}^\si-g^{kp}\ip{f_{p}}{f_{ij}}f_k^\si & \quad \si>m.
\end{array}
\right.
\]
The mean curvature vector $H$ is given by 
\begin{equation}\label{mean_curv_vect_H}
\begin{array}{rcl}
H^\si&=&\displaystyle g^{ij}H_{ij}^\si 
=\displaystyle \left\{
\begin{array}{ll}
-g^{ij}g^{\si p}\ip{f_{p}}{f_{ij}} & \quad \si\le m, \\[2mm]
g^{ij}f_{ij}^\si-g^{ij}g^{kp}\ip{f_{p}}{f_{ij}}f_k^\si & \quad \si>m,
\end{array}
\right.
\end{array}
\end{equation}
which implies 
\begin{equation}\label{Delta_H_sq_gen}
\Delta {|H|}^2(0)=\sum_{\ell\le m}\partial_\ell^{\,2}\biggl(
\sum_{\si\le m} \left(g^{ij}g^{\si p}\ip{f_{p}}{f_{ij}}
\right)^2
+\sum_{\si> m} \left(g^{ij}f_{ij}^\si-g^{ij}g^{kp}\ip{f_{p}}{f_{ij}}f_k^\si\right)^2
\biggr).
\end{equation}
Now \eqref{Delta_H_at_0} follows from 
\[
\begin{array}{rcl}
\displaystyle \sum_{\ell\le m}\partial_\ell^{\,2} \biggl(\sum_{\si\le m} \left(g^{ij}g^{\si p}\ip{f_{p}}{f_{ij}}\right)^2\biggr)
&\eao &\displaystyle 
2\ip{f_{ii}}{f_{k\ell}}\,\ip{f_{jj}}{f_{k\ell}}, \\[4mm]
\displaystyle \sum_{\ell\le m}\partial_\ell^{\,2} \biggl(\sum_{\si> m} \left(g^{ij}f_{ij}^\si\right)^2\biggr)
&\eao &\displaystyle 2\big\langle f_{kk}, -2\ip{f_{i\ell}}{f_{j\ell}}f_{ij}
+{f_{ii\ell\ell}}
\big\rangle
+2\ip{f_{iik}}{f_{jjk}}, \\[4mm] 
\displaystyle -2\sum_{\ell\le m}\partial_\ell^{\,2} \biggl(\sum_{\si> m} 
\left(g^{ij}f_{ij}^\si\right)
\left(g^{i'j'}g^{kp}\ip{f_{p}}{f_{i'j'}}f_k^\si\right)\biggr)
&\eao &\displaystyle -4\ip{f_{ii}}{f_{k\ell}}\,\ip{f_{jj}}{f_{k\ell}}  \\[4mm]
\displaystyle \sum_{\ell\le m}\partial_\ell^{\,2} \biggl(\sum_{\si> m} 
\left(g^{ij}g^{kp}\ip{f_{p}}{f_{ij}}f_k^\si\right)^2\biggr)
&\eao &\displaystyle 0.
\end{array}
\]

\eqref{Delta_H_at_0_third} and \eqref{f****-terms_Delta_H} can be obtained from the last two terms of \eqref{Delta_H_at_0}. 
\end{proof}

\smallskip
Assume $M$ is an $m$-dimensional hypersurface in what follows in this subsection. 
Assume that $\{\partial/\partial u_i\}$ is an orthonormal basis of $T_0M$, each $\partial/\partial u_i$ belonging to a principal direction of $M$ at $0$. 
Since 
$R_{jk\ell m}\eao\kappa_j\kappa_k(\delta_{jm}\delta_{k\ell}-\delta_{j\ell}\delta_{km})$ 
we have
\begin{eqnarray}
\Sc^2&=&\displaystyle 4\bigg(\sum_{i\ne j}\kappa_i\kappa_j\bigg)^2,  \label{Sc_square_principal_curv}  \\
{|\Rm|}^2&=&\displaystyle 2\Bigg(\bigg(\sum_i\kappa_i^{\,2}\bigg)^2-\sum_j\kappa_j^{\,4}\Bigg), \label{Rm_square_principal_curv} \\
{|\mbox{\rm Ric}|}^2&=&\displaystyle \sum_i\kappa_i^{\,4}
-2\bigg(\sum_i\kappa_i^{\,3}\bigg)\sum_j\kappa_j
+\bigg(\sum_i\kappa_i^{\,2}\bigg)\bigg(\sum_j\kappa_j\bigg)^2.
\label{Ricci_sq_principal_curv} 
\end{eqnarray}

\begin{lemma}
Under the same assumption as above, 
we have at the origin 
\begin{eqnarray}
\Delta H&\eao & -2\sum_i\kappa_i^3-H\sum_i\kappa_i^2+\sum_if_{iiii}+2\sum_{i<j}f_{iijj}\,,  \label{Delta_H_m=4}\\
{|\nabla H|}^2&\eao & 2\sum_i{\Big(\sum_jf_{ijj}\Big)}^2, \label{grad_H_sq_m=4}\\
\Delta \Sc &\eao & -8\sum_{i\ne j}\kappa_i\kappa_j^3-4\sum_{i<j,k\ne i,j}
\kappa_i\kappa_j\kappa_k^2  \label{Delta_Sc_1} \\
&&\displaystyle +4\sum_{i\ne j}f_{iii}f_{ijj}+4\sum_{i<j,k\ne i,j}
f_{iik}f_{jjk}
-4\sum_{i\ne j}f_{ijj}^2-12\sum_{i<j<k}f_{ijk}^2  \label{Delta_Sc_center} \\
&&\displaystyle +2\sum_i(H-\kappa_i)f_{iiii}+2\sum_{i<j}(2H-\kappa_i-\kappa_j)f_{iijj}\,.  \label{Delta_Sc_3}
\end{eqnarray}
\end{lemma}

\begin{proof}
When the codimension is $1$, the inner product $\ip{f_{i}}{f_{j}}$ is substituted by the product $f_if_j$, so we have 
\begin{equation}\label{preliminaries_gij_hij_H_m=4}
\begin{array}{c}
\displaystyle g_{ij}=\delta_{ij}+f_if_j, 
\>\> g
=1+{|\nabla f|}^2,
\>\> g^{ij}=\delta^{ij}-\frac{f_if_j}{g}, \>\>
h_{ij}=\frac{f_{ij}}{\sqrt{g}}, 
\>\> H=g^{ij}h_{ij}. 
\end{array}
\end{equation}
Since $g_{ij}\eao\delta_{ij}$, $g^{ij}\eao\delta^{ij}$, $g\eao1$, and $f_{ij}\eao\delta_{ij}\kappa_i$, we have 
\begin{eqnarray}\label{Laplacian_H_m=4_o}
\Delta H&\eao&\displaystyle \sum_{i,\a}f_{ii\,\a\a}-\sum_{i,\a}f_{ii}{f_{\a\a}}^2-2\sum_i{f_{ii}}^3,  
%
\notag
\end{eqnarray}
which implies \eqref{Delta_H_m=4}. 
\eqref{grad_H_sq_m=4} follows from 
\[
\begin{array}{rcl}
{|\nabla H|}^2&=&\displaystyle g^{\a\beta}\pd{H}{u_\a}\pd{H}{u_\beta} 
\eao\displaystyle \sum_\a\biggl(\sum_if_{ii\a}\biggr)\biggl(\sum_jf_{jj\a}\biggr).
\end{array}
\]

Finally, \eqref{Delta_Sc_1} - \eqref{Delta_Sc_3} are obtained from the first equality of Lemma \ref{lem_la_A_f_Delta_Sc_H}. 
\end{proof}
%

\section{Residues of closed submanifolds of Euclidean spaces}

\subsection{Known formulae of residues}
%
Suppose $M$ is an $m$ dimensional closed submanifold of $\RR^n$. 
Assume $\lambda\equiv1$ in this Subsection. 
In this case we drop the letter $\lambda$ off. 
Recall that $\psi_{x}(t)=\textrm{Vol}(M\cap B_t(x))$, where $B_t(x)$ is the $n$-ball with radius $t$ and center $x$. 
Karp and Pinsky gave 
\begin{equation}\label{KP-expansion}
\textrm{Vol}(M\cap B_t(x))=\frac{o_{m-1}}m\,t^m
+\frac{o_{m-1}}{8m(m+2)}\left(2{\Vert h\Vert}^2-{\vert H\vert}^2\right)t^{m+2}+O(t^{m+3})
\end{equation}
(\cite{KP}), where $o_{m-1}$ is the volume of the unit $(m-1)$-sphere as before (see Subsection \ref{preliminary_calculation_graph} for $h$ and $H$).

When $M$ is a hypersurface in $\RR^{m+1}$ with principal curvatures $\ka_i$, ${\Vert h\Vert}^2=\sum_i{\ka_i}^2$ and 
$H=\sum_i\kappa_i$, i.e. $H$ is $m$ times the mean curvature. 

\begin{proposition}\label{prop_residues_closed_12}{\rm (\cite{B,FV,OS2})}
As above, 
\begin{enumerate}
\item The first residue is given by $\displaystyle \Rd{-m}{M}=\displaystyle o_{m-1}\,\mbox{\rm Vol}\,(M).$
\item When $m=1$ and $K$ is a knot in $\RR^3$ the second residue is given by 
\begin{equation}\label{residue_knot_2}
\Rd{-3}{K}= \frac14\int_K\kappa(x)^2 dx ,   \notag
\end{equation}
where $\ka$ is the curvature. 
\item When $m=2$ and $M$ is a closed surface in $\RR^3$ the second residue is given by 
\begin{equation}\label{residue_surface_2}
\Rd{-4}{M}= \frac\pi 8\int_M(\kappa_1-\kappa_2)^2\dx.
\end{equation}
It is equal to the Willmore energy modulo the Euler characteristic up to multiplication by a constant, and is invariant under \M transformations. 
\item In general the second residue is given by 
\begin{equation}\label{residue_closed_-m-2}
\Rsub{M}(-m-2)=\frac{o_{m-1}}{8m}\int_M\left(2{\Vert h\Vert}^2-{\vert H\vert}^2\right)\dx .
\end{equation}
\end{enumerate}
\end{proposition}

\eqref{residue_closed_-m-2} follows from \eqref{RsubMrhoddtpsix} and \eqref{KP-expansion}, although it was not explicitly written in \cite{OS2}. 

\subsection{\M invariance}
Let $M$ be an $m$ dimensional closed submanifold of $\RR^n$. We show that if $m$ is even then the residue at $z=-2m$ is \M invariant. 
First consider the behavior of the energies and residues under homotheties $x\mapsto cx$ $(c>0)$. 
When $z$ is not a pole of $B_M$, $B_M(z)$ is a homogeneous function of $M$ with degree $z+2m$, i.e. $B_{cM}(z)=c^{z+2m}B_M(z)$. 
When $z$ is a pole of $B_M$, the residue $\Res_M$ inherits the homogeneity property (Lemma 3.9 \cite{OS2}), whereas the regularized energy $E_M(z)$ does not unless the residue vanishes identically 
(Proposition 3.10 \cite{OS2}). 

We show that the same phenomenon occurs for \M invariance. 
The degree of homogeneity implies that the \M invariance can occur only when $z=-2m$. We have shown that if $m$ is odd then $z=-2m$ is not the pole of $B_M$ and $B_M(-2m)$ is \M invariant (Proposition 3.12 \cite{OS2}).

\begin{theorem}\label{M-inv-res-closed}
Suppose $m$ is even. 
The residue at $z=-2m$ is \M invariant, i.e. 
$\Rsub{T(M)}(-2m)=\Rsub{M}(-2m)$ for any $m$-dimensional closed submanifold $M$ of $\RR^n$ and 
for any \M transformation $T$ as far as $T(M)$ remains compact. 
\end{theorem}

The proof is almost same as that of Proposition 3.12 of \cite{OS2} 
except for \eqref{z+2mBTM-BM} and the final step (iii) below. 
The difference is that in \cite{OS2} we studied the behavior of $B_M(z)$ but this time we consider the residue, which can be expressed by $\lim_{z\to-2m}(z+2m)B_M(z)$, so we have a better convergence situation when $z$ goes to $-2m$. 
We describe the entire proof here for the convenience of the reader. 

\begin{proof}
Remark that we have only to discuss the case when $0\not \in M$ and $T$ is an inversion in the unit sphere with center the origin. 

Let $\widetilde M, \tilde x, \tilde y$ denote the images by $T$ of $M,x,y$ respectively. 
Since 
\[
|\tilde x-\tilde y|=\frac{|x-y|}{|x|\,|y|},\quad dv({\tilde x})=\frac{\dx }{|x|^{2m}}, \quad dv({\tilde y})=\frac{\dy }{|y|^{2m}},
\]
we have for $\R z>-m$ 
\begin{align}
 B_{\widetilde M}(z)- B_{ M}(z)
&=\int_{M\times M}|x-y|^z \frac1{|x|^{z+2m}|y|^{z+2m}}\, \dx \dy 
-\int_{M\times M}|x-y|^z\, \dx \dy  \notag \\[2mm]
&= \int_{M\times M}|x-y|^z \left[\left(\frac1{|x|\,|y|}\right)^{z+2m}-1\right] \dx \dy . \notag 
%
\end{align}
Let $\delta >0$ be such that the spheres $\partial B_t(x)$ are transverse to $M$ for all $x$ in $M$ and all $t\in (0,\delta ]$. 
Put 
\[\begin{array}{rcl}
\la_z(x,y)&=&\displaystyle \left(\frac1{|x|\,|y|}\right)^{z+2m}-1, \\[4mm]
\Psi_z(t)&=&\displaystyle \psi_{\la_z}(t)=\iint_{(M\times M)\cap \Delta_t}\left[\left(\frac1{|x|\,|y|}\right)^{z+2m}-1\right]\,\dx \dy ,
\end{array}
\]
where $\Delta_t=\{(x,y)\in\RR^n\times\RR^n\,:\,|x-y|\le t\}$. 
Then $\Psi_z(t)$ is smooth on $[0,\delta ]$. 

By putting $\varphi(t)=\Psi_z'(t)$ and $k=2m$ in \eqref{Iphiz}, 
\begin{equation}\label{BtildeM-BM}
\begin{array}{l}
\displaystyle B_{\widetilde M}(z)- B_{ M}(z)=\displaystyle \left(\int_0^\delta  t^z  \Psi_z'(t) dt
+\iint_{(M\times M)\setminus \Delta_\delta } |x-y|^z\, \la_z(x,y)\dx \dy  \right) 
\\[4mm]
=\displaystyle \int_0^\delta  t^z\left(\Psi_z'(t)-\sum_{j=0}^{2m-1}\frac{\Psi_z^{(j+1)}(0)}{j!}\,t^j \right)dt  
+\sum_{j=1}^{2m}\frac{\Psi_z^{(j)}(0)\, \delta ^{z+j}}{(j-1)!\,(z+j)}  
+\iint_{(M\times M)\setminus \Delta_\delta } |x-y|^z \,\la_z(x,y)\dx \dy . 
\end{array}
\end{equation}

\medskip  
Since 
\begin{equation}\label{residue=limit}
\Rd{-2m}{M}=\lim_{z\to-2m}(z+2m)B_M(z),
\notag
\end{equation}
we have only to show 
\begin{equation}\label{z+2mBTM-BM}
\lim_{z\to-2m}(z+2m)(B_{\widetilde M}(z)-B_M(z))=0. 
\end{equation}
Since all the derivatives of $\la_z$ converge uniformly to $0$ as $z$ approaches $-2m$, Lemma \ref{even_odd} implies 
\begin{equation}\label{eq_lim_sup}
\lim_{z\to-2m} \sup_{0\le t\le \delta }\left|\Psi_z^{(i)}(t)\right|=0 \qquad (\forall i). 
\end{equation}

(i) The first term in the right hand side of \eqref{BtildeM-BM} goes to $0$ as $z$ approaches $-2m$ since 
\[
 \left|\int_0^\delta  t^z\left(\Psi_z'(t)-\sum_{j=0}^{2m-1}\frac{\Psi_z^{(j+1)}}{j!}\,t^j
\right)\,dt\right| \displaystyle\leq \sup_{0\le t\le \delta }\left|\Psi_z^{(2m+1)}(t)\right|\frac{1}{(2m)!}\int_0^\delta \,t^{z+2m}\,dt.
\]

\smallskip
(ii) The third term in the right hand side of \eqref{BtildeM-BM} also goes to $0$ as $z$ approaches $-2m$ since $\la_z$ converges uniformly to $0$ on $M\times M$ as $z$ approaches $-2m$. 

\smallskip
(iii) The second term in the right hand side of \eqref{BtildeM-BM}, if multiplied by $z+2m$, goes to $0$ as $z$ approaches $-2m$ since 
\[
(z+2m)\sum_{j=1}^{2m}\frac{\Psi_z^{(j)}(0)\, \delta ^{z+j}}{(j-1)!\,(z+j)}
=(z+2m)\sum_{j=1}^{2m-1}\frac{\Psi_z^{(j)}(0)\, \delta ^{z+j}}{(j-1)!\,(z+j)}
+\frac{\Psi_z^{(2m)}(0)\, \delta ^{z+2m}}{(2m-1)!}
\]
and $\Psi_z^{(2m)}(0)$ goes to $0$ as $z$ to $-2m$ by \eqref{eq_lim_sup}. 

Now (i), (ii) and (iii) complete the proof of \eqref{z+2mBTM-BM}. 
\end{proof}

As we mentioned at the beginning of this section, if $\Res_M(-2m)$ does not vanish for all $M$, the regularized energy $E_M(-2m)$ is not even scale invariant. 
We conjecture that for each even number $m$ there is an $m$-dimensional closed submanifold $M$ such that $\Res_M(-2m)\ne0$, which holds if $m=2$ (\cite{OS2}).

\subsection{$\nu$-weighted energy and residue}
%
Assume $M$ is an oriented closed submanifold of $\RR^n$. 
When $M$ is a hypersurface bounding a compact body $\dom$, the meromorphic energy function can be expressed by a double integral on $M$ as \eqref{IOmegaboundary} as we will see in Section \ref{section_compact_bodies}. 
We generalize it to the case when $M$ has codimension greater than $1$, which we will deal with later in Subsections \ref{subsec_Weyl} and \ref{subsec_Graham-Witten}. 
For this purpose we generalize \[\la_\nu(x,y)=\langle \uon{x}, \uon{y}\rangle,\] where $\uon{x}$ is the unit outer normal vector to $\pO$,  
by the inner product of $m$-vectors as follows. 

\medskip
Let $\Pi$ and $\Pi'$ be $m$-dimensional oriented vector subspaces of $\RR^n$. 
Let $(v_1,\dots,v_m)$ and $(v_1',\dots,v_m')$ be ordered orthonormal bases of $\Pi$ and $\Pi'$ respectively that give the positive orientations of $\Pi$ and $\Pi'$. 
Define
\begin{equation}\label{product_planes}
\langle \Pi, \Pi'\rangle
=\ip{v_1\wedge\dots\wedge v_m}{v'_1\wedge\dots\wedge v'_m}
=\det (\langle v_i, v_j'\rangle). 
\end{equation}

We remark that it is independent of the choice of ordered orthonormal bases. 

\begin{definition}\label{nu-weighted} \rm 
Define the {\em $\nu$-weighted potential, energy and (local) residue} of $M$ (at $x$) by substituting $\lambda$ in Definition \ref{def_B_M_etc}  
by \[\lambda_\nu(x,y)=\langle T_xM, T_yM\rangle,\] where the orientation of the tangent spaces are induced from that of $M$, and denote them by $B_{M,\nu}(z)$ and $\Res_{M,\nu}(z)$. 
\end{definition}

We remark that when $M$ is the boundary of a compact body $\dom$ in $\RR^n$, 
\begin{equation}\label{BR_nu_BR_dom}
B_{M,\nu}(z)=-z(z+n-2)B_\dom(z-2), \quad \Res_{M,\nu}(z)=-z(z+n-2)\Res_\dom(z-2) \quad 
\quad  (z\ne0,-n+2).
\end{equation}

When $M$ is a knot $K$ in $\RR^3$ and $z=-2$, $B_{K,\nu}(-2)$ is equal to the ``energy'' studied in \cite{IN14} and \cite{OS1} up to multiplication by a constant. 
For general $z$, $B_{K,\nu}(z)$ was studied in \cite{Oinductance}.

\begin{lemma}\label{lemma_la_nu_times_A_f_=1}
Let $M$ be an $m$ dimensional oriented closed submanifold of $\RR^n$ and $x$ be a point on $M$. 
Suppose $x=0$, $T_xM=\RR^m$, and $M$ is expressed near $x=0$ as a graph of a function $f\colon\RR^m\to\RR^{n-m}$. 
The area density $A_f$ given by \eqref{def_A_f} satisfies 
\begin{equation}\label{la_nu_times_A_f_=1}
\la_\nu(x,y)A_f(y)=1
\end{equation}
for a point $y$ in a neighbourhood of $x$. 
\end{lemma}

\begin{proof}
Recall $A_f=\sqrt{g}$ (see Subsection \ref{preliminary_calculation_graph}). 
By definition $\la_\nu(0,y)=\langle T_0M, T_yM\rangle$, where $\langle \,\cdot\,, \,\cdot\,\rangle$ is given by \eqref{product_planes}. 
Since $\{{\partial y}/{\partial u_i}=(\vect e_i,f_i)\}_{i=1,\dots,m}$ 
is a (not necessarily orthonormal) basis of $T_yM$ and $\{v_i=(\vect e_i,0)\}_{i=1,\dots,m}$ 
is an orthonormal basis of $T_0M$, where $\vect e_i$ are the standard unit vectors of $\RR^m$, 
using a matrix of change of basis of $T_yM$ we have 
\[
\la_\nu(0,y)=\langle T_0M, T_yM\rangle=\frac1{A_f}.
\]
\end{proof}

\subsection{Scalar curvature and mean curvature in terms of residues}\label{subsect_scalar_curv_residues}
%
\begin{lemma}\label{lem_nu_Res_-m-2}
Let $M$ be an $m$ dimensional submanifold of $\RR^n$ and $x$ be a point on $M$. 
Assume that $M$ is locally expressed near $x$ as a graph of a smooth function $f$ from $T_xM\cong\RR^m$ to $\RR^{n-m}$. 
Then the local $\nu$-weighted residue and the local residue at $-m-2$ are given by 
\begin{eqnarray}
\Res_{M,\nu,x}^{\rm loc}(-m-2)&=&\displaystyle \frac{o_{m-1}}m\biggl(
-\frac38\sum_i \ip{f_{ii}}{f_{ii}}-\frac18\sum_{i\ne j} \ip{f_{ii}}{f_{jj}}-\frac14\sum_{i\ne j}\ip{f_{ij}}{f_{ij}} \biggr), \label{R_M_lambda_loc_-m-2}\\
\Res_{M,x}^{\rm loc}(-m-2)&=&\displaystyle \frac{o_{m-1}}m\biggl(
\phantom{-}\frac18\sum_i \ip{f_{ii}}{f_{ii}}-\frac18\sum_{i\ne j} \ip{f_{ii}}{f_{jj}}+\frac14\sum_{i\ne j} \ip{f_{ij}}{f_{ij}} \biggr), \label{R_M_loc_-m-2}
\end{eqnarray}
where $o_{m-1}$ is the volume of the unit $(m-1)$-sphere. 
\end{lemma}

\begin{proof}
Put $x=0$. 
Lemma \ref{lemma_la_nu_times_A_f_=1} implies that $a_{\nu,2}(-m-2;w)$ and $a_{2}(-m-2;w)$, where $w\in S^{m-1}$, defined in Subsection \ref{method_graph} are the coefficients of $r^2$ of the Maclaurin series expansion in $r$ of 
\begin{eqnarray}
\xi_\nu(r,w,-m-2)&=&\displaystyle {\biggl(1+\frac{|f(rw)|^2}{r^2}\biggr)}^{-(m+2)/2} ,\notag \\[1mm]
\xi(r,w,-m-2)&=&\displaystyle \xi_\nu(r,w,-m-2) A_f(rw) \label{xi_r_w_-m-2}
\end{eqnarray}
respectively. 
By \eqref{def_xi}, \eqref{def_a_la_k}, and \eqref{local_residue_graph_FV} the local $\nu$-weighted residue and the local residue are given by 
\[
\Res_{M,\nu,x}^{\rm loc}(-m-2)=\int_{S^{m-1}}
a_{\nu,2}(-m-2;w)\,dv(w), \quad
\Res_{M,x}^{\rm loc}(-m-2)=\int_{S^{m-1}}
a_{2}(-m-2;w)\,dv(w).
\]

If we put $A_k(w)$ to be the coefficient of $r^{k}$ of the Maclaurin series expansion of ${|f(rw)|}^2$, then  
\[
A_1(w)=A_2(w)=A_3(w)=0, \> A_4(w)=\frac14\sum_{i,j,k,l}\ip{f_{ij}}{f_{kl}} w_iw_jw_kw_l. 
\]
which implies that 
\[a_{\nu,2}(-m-2;w)=-\frac{m+2}2A_4(w). \]
Since the integration on $S^{m-1}$ of a term containing odd powers of $w_i$ vanishes, 
\begin{eqnarray}
\Res_{M,\nu,x}^{\rm loc}(-m-2)&=&\displaystyle -\frac{m+2}8\int_{S^{m-1}}\biggl(
\sum_i \ip{f_{ii}}{f_{ii}} \>w_i^4+\sum_{i\ne j}\ip{f_{ii}}{f_{jj}}\>w_i^2w_j^2
+2\sum_{i\ne j}\ip{f_{ij}}{f_{ij}}\>w_i^2w_j^2
\biggr)dv(w), \hspace{1.1cm}{} \label{Res_M_nu_-m-2_f**}
\end{eqnarray}
and by \eqref{def_A_f} and \eqref{xi_r_w_-m-2}, 
\begin{equation}
\Res_{M,x}^{\rm loc}(-m-2)=\Res_{M,\nu,x}^{\rm loc}(-m-2)+\frac12\int_{S^{m-1}}\sum_{i,j} \ip{f_{ij}}{f_{ij}}\>w_j^2 \,dv(w).
\notag
\end{equation}
The conclusion follows from substituting the following moment integrals 
\begin{equation}\label{int_S_n-2_w^2}
\int_{S^{m-1}}w_i^2\,dv=\frac{o_{m-1}}{m},
\end{equation}
\begin{equation}\label{moment_integral_4_on_S3}
\int_{S^{m-1}}w_i^{\,4}\,dv=\frac{3\,o_{m-1}}{m(m+2)}, \quad 
\int_{S^{m-1}}w_i^{\,2}w_j^{\,2}\,dv=\frac{o_{m-1}}{m(m+2)} \quad (i\ne j)
\end{equation}
(cf. Theorem 1.5 in Appendix A of \cite{G}).
\end{proof}

\eqref{Sc_at_0}, \eqref{mean_curv_vect_H} and Lemma \ref{lem_nu_Res_-m-2} imply 

\begin{theorem}\label{thm_Sc_residue}
Let $M$ be an $m$ dimensional closed submanifold of $\RR^n$. 
The scalar curvature and the squared norm of the mean curvature vector can be expressed by a linear combination of the local residue and the $\nu$-weighted local residue (cf. Definition \ref{nu-weighted}) as 
\begin{eqnarray}
\Sc(x)&=&\displaystyle -\frac{2m}{o_{m-1}}\Bigl(\Res_{M,\nu,x}^{\rm loc}(-m-2)+3\Res_{M,x}^{\rm loc}(-m-2)\Bigr), \label{scalar_curv_residue}\\
{|H|}^2(x)&=&\displaystyle -\frac{4m}{o_{m-1}}\Bigl(\Res_{M,\nu,x}^{\rm loc}(-m-2)+\Res_{M,x}^{\rm loc}(-m-2)\Bigr). \label{sq_norm_mean_curv_residue}
\end{eqnarray}
\end{theorem}

When $M$ is a hypersurface, the mean curvature can be expressed by relative residue as we will see in \eqref{local_relative_residue_mean_curv}. 
When $M$ is a surface in $\RR^3$, 
by substituting $m=2$ to \eqref{sq_norm_mean_curv_residue} 
and integrating it on $M$ we see that the classical Willmore energy of $M$ can be expressed as a linear combination of the residues as 
\begin{equation}\label{residue_Willmore_m=2}
\frac14\int_M H^2\dx =-\frac1\pi(\Res_{M,\nu}(-4)+\Res_{M}(-4)). \notag
\end{equation}
We will see a kind of generalization of the above  
in Theorem \ref{prop_GW-residues=W+PCE}.

\subsection{\M invariance}

Let us study the \M invariance of $B_{M,\nu}(-2m)$ for odd $m$ and $\Res_{M,\nu}(-2m)$ for even $m$. 

\begin{lemma}\label{inversion_codim>1}
Let $M$ be an $m$ dimensional submanifold of $\RR^n$ and $x$ be a point on $M$ $(x\ne0)$. 
Let $I$ be an inversion in the unit sphere with center the origin. 
Put $\widetilde M=I(M)$ and $\tilde x=I(x)$. 
Suppose $M$ (or $\widetilde M$) can be locally expressed as graph of $f\colon T_xM\to(T_xM)^\perp$ near $x$ (or respectively, $\tilde f\colon T_{\tilde x}\widetilde M\to\big(T_{\tilde x}\widetilde M\,\big)^\perp$ near $\tilde x$). 

Let $\{v_1,\dots, v_n\}$ be an orthonormal basis of $T_x\RR^n$ such that $\{v_1,\dots, v_m\}$ span $T_xM$. 
Put $\tilde v_i=I_\ast v_i/|I_\ast v_i|$. Then  $\{\tilde v_1,\dots, \tilde v_n\}$ is an orthonormal basis of $T_{\tilde x}\RR^n$ such that $\{\tilde v_1,\dots, \tilde v_m\}$ span $T_{\tilde x}\widetilde M$. 
Let $(u_1, \dots, u_n)$ (or $(\tilde u_1, \dots, \tilde u_n)$) be the coordinates of $T_x\RR^n$ (or resp. $T_{\tilde x}\RR^n$) with respect to the frame $(v_1,\dots, v_m)$ 
(or resp. $(\tilde v_1,\dots, \tilde v_m)$). 

Then at the origin we have
\begin{equation}\label{inversion_f_ij}
\frac{\partial^2\tilde f^\sigma}{\partial \tilde u_i\partial \tilde u_j}
={|x|}^2\frac{\partial^2f^\sigma}{\partial u_i\partial u_j}
+2\ip{v_\sigma}{x}\,\delta_{ij}
\notag
\end{equation}
for $1\le i,j\le m$ and $m+1\le \sigma\le n$. 
\end{lemma}

\begin{proof}
First note that 
\[\tilde v_i=v_i-2\frac{\ip{v_i}{x}}{{|x|}^2}x.\]
Define $\bar f\colon T_xM\to\RR^n$ by 
$\bar f(u)=x+u+f(u)$ $\>(u\in T_xM)$ and 
$\hat f\colon T_{\tilde x}\widetilde M\to\RR^n$ by $\hat f(\tilde u)=\tilde x+\tilde u+\tilde f(\tilde u)$ $\>(\tilde u\in T_{\tilde x}\widetilde M)$. 
Let $\phi={\hat f}^{-1}\circ I \circ \bar f$ be a bijection form a neighbourhood of $0$ of $T_xM$ to a neighbourhood of $0$ of $T_{\tilde x}\widetilde M$. 
Express both $\tilde u$ and $\tilde f(\tilde u)$ by functions of $u=\phi^{-1}(\tilde u)$ as 
\[
\begin{array}{rcl}
\tilde u_i&=&\displaystyle \frac{u_i}{{|\bar f(u)|}^2}+\ip{v_i}{x}\,\frac{{|u|}^2+{|f(u)|}^2}{{|x|}^2{|\bar f(u)|}^2} \quad(1\le i\le m),\\[4mm]
\tilde f^\sigma(\tilde u)&=&\displaystyle \frac{f^\sigma(u)}{{|\bar f(u)|}^2}+\ip{v_\sigma}{x}\,\frac{{|u|}^2+{|f(u)|}^2}{{|x|}^2{|\bar f(u)|}^2} \quad (m+1\le\sigma\le n),
\end{array}
\]
and use 
\[
\left(\pd{u_i}{\tilde u_j}\right)\left(\pd{\tilde u_i}{u_j}\right)=E, \>\>
\pd{u_i}{\tilde u_j}(0)={|x|}^2\delta_{ij}, \>\>
f(0)=\pd{f}{u_i}(0)=0
\]
to compute 
\[
\frac{\partial^2\tilde f^\sigma}{\partial \tilde u_i\partial \tilde u_j}
=\sum_{a,b}\pd{u_a}{\tilde u_i}\pd{}{u_a}\left(\pd{u_b}{\tilde u_j}\cdot\pd{\tilde f^\sigma}{u_b}\right),
\]
we obtain the conclusion. 
\end{proof}

\begin{proposition}\label{Mobius_inv_R_nu}
The $\nu$-weighted energy or residue, depending on the parity of $m$, at $z=-2m$ is \M invariant if one of the following conditions is satisfied.
\begin{enumerate}
\item $m=1$.
\item $m=2$.
\item $M$ is a boundary of a compact body $\Omega$. 
\end{enumerate}\end{proposition}

\begin{proof}
(1) was proved by \cite{OS1} when $n=3$ and by \cite{IN14} analytically in general case. 

\smallskip
(3) By \eqref{BR_nu_BR_dom} the conclusion follows form the \M invariance of the energy and residue of compact bodies, which is proved by Proposition 3.12 of \cite{OS2} if $m$ is odd and by Theorem \ref{thm_M_inv_body}, which we will state later, if $m$ is even. 

\smallskip
(2) Substituting $m=2$ to \eqref{R_M_lambda_loc_-m-2} we obtain 
\begin{equation}\label{Res_M_nu_x_-4}
\Res_{M,\nu,x}^{\rm loc}(-4)=-\frac\pi8
\Bigl[3\big(\ip{f_{11}}{f_{11}}+\ip{f_{22}}{f_{22}}\big)+2\ip{f_{11}}{f_{22}}
+4\ip{f_{12}}{f_{12}}
\Bigr].
\end{equation}
Putting $m=2$ in \eqref{Sc_at_0}, the scalar curvature at the origin is given by 
\[\Sc
\eao 2\big(\ip{f_{11}}{f_{22}}-\ip{f_{12}}{f_{12}}\big),\]
which implies 
\[
\Res_{M,\nu,x}^{\rm loc}(-4)=-\frac\pi8
\Bigl[3\ip{f_{11}-f_{22}}{f_{11}-f_{22}}+12\ip{f_{12}}{f_{12}}+4\Sc
\Bigr].
\]
Lemma \ref{inversion_codim>1} implies $\ip{f_{11}-f_{22}}{f_{11}-f_{22}}\,dv$ and 
$\ip{f_{12}}{f_{12}}\,dv$ are pointwise \M invariant. 
On the other hand, by the Gauss-Bonnet theorem 
$\int_M\Sc\,dv=4\pi\chi(M)$. 
Thus the proof is completed. 
\end{proof}

The author does not know whether the \M invariance of $\Res_{M,\nu}(-2m)$ holds in general.

\subsection{Necessary order of differentiation to obtain residues}\label{subsub_order_diff}
Suppose $M$ is locally expressed as a graph of a function $f$ defined on a tangent space as in Subsection \ref{method_graph}. 
A priori, \eqref{def_xi}, \eqref{def_a_la_k}, and \eqref{local_residue_graph_FV} imply that we need the derivatives of $f$ of order up to $2j$ to express the local residue at $z=-m-2j$ ($j=0,1,2,\dots$). 

\begin{conjecture}\label{conj_order_diff}\rm 
We only need the derivatives of order up to $j+1$ to express the ($\nu$-weighted) residue of a closed submanifold $M$ at $z=-m-2j$ ($j\ge 2$).
\end{conjecture}

\noindent
The case when $m=1$ was proved in \cite{Oinductance}. 
We will show that the conjecture holds when $j=2$. 

\begin{theorem}\label{prop_residue_H_Sc_graph}
Let $M$ be an $m$-dimensional closed submanifold of $\RR^n$ and $x$ be a point on $M$. 
Suppose 
$M$ can locally be expressed as a graph of a function $f\colon T_xM\to (T_xM)^\perp$. 
Then 
\begin{align}
\Res_{M,\nu,x}^{\rm loc}(-m-4)& +\frac{o_{m-1}}{32\,m(m+2)}\left(5\Delta {|H|}^2-4\Delta \Sc\right),    \label{residue_M^4_nu_-8_modified} \\[1mm]
\Res_{M,x}^{\rm loc}(-m-4)& -\frac{o_{m-1}}{32\,m(m+2)}\left(3\Delta {|H|}^2-4\Delta \Sc\right)   \label{residue_M^4_-8_modified}
\end{align}
can be expressed in terms of the derivatives of $f$ of order up to $3$. 
\end{theorem}

We will prove this theorem by taking the proof of Lemma \ref{lem_nu_Res_-m-2} one step further.

\begin{proof}
We first compute $\Res_{M,\nu,x}^{\rm loc}(-m-4)$. 
It is given by $\Res_{M,\nu,x}^{\rm loc}(-m-4)=\int_{S^{m-1}}a_{\nu,4}(w)\,dv$, 
where $a_{\nu,4}(w)$ is the coefficient of $r^{4}$ of the Maclaurin series expansion in $r$ of 
\[
\xi_{\nu}(r,w,-m-4)={\biggl(1+\frac{{|f(rw)|}^2}{r^2}\biggr)}^{{-({m+4})/{2}}}
\]
(cf. Lemma \ref{lemma_la_nu_times_A_f_=1}). 
Let $A_k(w)$ be the coefficient of $r^{k}$ of the Maclaurin series expansion of ${|f(rw)|}^2$, then 
\[
a_{\nu,4}(w)=\frac{(m+4)(m+6)}{8}{A_4(w)}^2-\frac{m+4}2A_6(w).
\]
The $f_{****}$ terms of $a_{\nu,4}(w)$ come from $-\frac{m+4}2A_6(w)$, which are given by 
\begin{equation}
-\frac{m+4}{48}\,
\ip{f_{ij}}{f_{abcd}}\,w_iw_jw_aw_bw_cw_d . \notag
\end{equation}
Since the integral on $S^{m-1}$ of a term containing odd powers of $w_i$ vanishes, we may assume that $\{i,j,a,b,c,d\}=\{i,i,a,a,c,c\}$. 
Computation using 
\begin{equation}\label{moment_integral_6_on_S3}
\begin{array}{c}
\displaystyle 
\int_{S^{m-1}}w_i^{\,6}\,dv=\frac{15\,o_{m-1}}{m(m+2)(m+4)}, \quad 
\int_{S^{m-1}}w_i^{\,4}w_j^{\,2}\,dv=\frac{3\,o_{m-1}}{m(m+2)(m+4)} \quad (i\ne j), \\[5mm]
\displaystyle 
\int_{S^{m-1}}w_i^{\,2}w_j^{\,2}w_k^{\,2}\,dv=\frac{o_{m-1}}{m(m+2)(m+4)} \quad (i\ne j\ne k\ne i)
\end{array}
\end{equation}
(cf. Theorem 1.5 in Appendix A of \cite{G}) shows that the terms in $\Res_{M,\nu,x}^{\rm loc}(-m-4)$ containing $f_{\ast\ast\ast\ast}$ are given by 
\begin{equation}\label{f****-terms_R_nu}
\begin{array}{l}
\displaystyle 
-\frac{o_{m-1}}{48m(m+2)}\left(
15\ip{f_{jj}}{f_{jjjj}}
+3\cdot{4\choose2}\ip{f_{jj}}{f_{jjkk}}
+3\ip{f_{jj}}{f_{kkkk}} \right.\\[4mm]
\displaystyle 
\left. 
\hspace{2.5cm}
+{4\choose2}\ip{f_{jj}}{f_{kkll}}
+3\cdot{2\choose1}{4\choose1}\ip{f_{jk}}{f_{jjjk}}
+4\cdot3\cdot2\,\ip{f_{jk}}{f_{jkll}}
\right)
\end{array}
\qquad (j\ne k\ne l\ne j),
\end{equation}
where we agree that we take sum for all possible combination of mutually distinct $j,k,l$. 
Now \eqref{f****-terms_R_nu}, \eqref{f****-terms_Delta_Sc} and \eqref{f****-terms_Delta_H} imply \eqref{residue_M^4_nu_-8_modified}. 

\medskip
Next we proceed to $\Res_{M,x}^{\rm loc}(-m-4)$. 
Let $a_4(w)$ be the coefficient of $r^{4}$ of the Maclaurin series expansion in $r$ of 
\[
\xi(r,w,-m-4)={\biggl(1+\frac{{|f(rw)|}^2}{r^2}\biggr)}^{-\frac{m+4}2}\sqrt{\det\big(\delta_{ij}+\left\langle f_i(rw), f_j(rw)\right\rangle\big)}\,.
\]
Let $B_k(w)$ be the coefficient of $r^{k}$ of the Maclaurin series expansions of $\det\left(g_{ij}(rw)\right)$, we have 
\[
a_4(w)=a_{\nu,4}(w)
+\frac12(-4)A_4(w)B_2(w)+\frac12B_4(w)-\frac18{B_2(w)}^2. 
\]
Then the $f_{****}$ terms of $a_4(w)-a_{\nu,4}(w)$ come from $\frac12B_4(w)$, which are given by
\[
\frac12\cdot\frac{8}{4!}\,\ip{f_{ia}}{f_{ibcd}}\,w_aw_bw_cw_d.
\]
The remaining computation can be carried out similarly using \eqref{moment_integral_4_on_S3}. 
Comparing the result with \eqref{f****-terms_Delta_Sc} and \eqref{f****-terms_Delta_H}, it follows that 
\[
\frac{32\,m(m+2)}{o_{m-1}}\big(\Res_{M,x}^{\rm loc}(-m-4)-\Res_{M,\nu,x}^{\rm loc}(-m-4)\big)
- 8\bigl(\Delta {|H|}^2-\Delta \Sc\bigr),
\]
can be expressed in terms of the derivatives of $f$ of order up to $3$, which implies \eqref{residue_M^4_-8_modified}. 
\end{proof}

Since the integrations of the second terms of \eqref{residue_M^4_-8_modified} and \eqref{residue_M^4_nu_-8_modified} vanish by Green's theorem as $\partial M=\emptyset$, 

\begin{corollary}
Conjecture \ref{conj_order_diff} holds for $j=2$. 
\end{corollary}

\subsection{4-dimensional case}
We give explicit expressions of the local ($\nu$-weighted) residues when $M$ is a $4$-dimensional hypersurface, which will be used later. 
Assume $m=4$ in this subsection. 
If we carry out all the calculations for the terms $a_{\nu,4}(w)$ and $a_4(w)-a_{\nu,4}(w)$ in the proof of Theorem \ref{prop_residue_H_Sc_graph}, we obtain 
\begin{eqnarray}
a_{\nu,4}(w)&=&\displaystyle \frac58
\ip{f_{ij}}{f_{kl}}\ip{f_{ab}}{f_{cd}}w_iw_jw_kw_lw_aw_bw_cw_d \label{a_nu_-8-1} \\[0mm]
&&\displaystyle -\frac19
\ip{f_{ijk}}{f_{abc}}w_iw_jw_kw_aw_bw_c
-\frac16
\ip{f_{ij}}{f_{abcd}}w_iw_jw_aw_bw_cw_d . 
\>\>{}\label{a_nu_-8}
\end{eqnarray}
\begin{eqnarray}
a_4(w)-a_{\nu,4}(w)&=&\displaystyle 
-\frac{12}{4\,!}
\ip{f_{ij}}{f_{ik}}\ip{f_{ab}}{f_{cd}}w_jw_kw_aw_bw_cw_d \label{a-a_nu-hot-1} \\[1mm]
&&\displaystyle +\frac14\sum_{i<j}\sum_{a,b,c,d}
\big(\ip{f_{ia}}{f_{ib}}\ip{f_{jc}}{f_{jd}}-2\ip{f_{ia}}{f_{jb}}\ip{f_{ic}}{f_{jd}}\big)w_aw_bw_cw_d \label{a-a_nu-hot-2} \\
&&\displaystyle -\frac18
\ip{f_{ij}}{f_{ik}}\ip{f_{ab}}{f_{ac}}w_jw_kw_bw_c \label{a-a_nu-hot-3} \\[1mm]
&&\displaystyle +\frac1{4\,!}
\big(3\,\ip{f_{iab}}{f_{icd}}+4\,\ip{f_{ia}}{f_{ibcd}}\big)w_aw_bw_cw_d. \label{a-a_nu-hot}
\end{eqnarray}

Application of \eqref{moment_integral_6_on_S3} to \eqref{a_nu_-8} implies that the $f_{\ast\ast\ast}$ terms in $\Res_{M,\nu,x}^{\rm loc}(-8)$ are given by 

\begin{equation}\label{f***-terms_R_M_nu_x_-8}
\frac{-\pi^2}{288}\bigg(
5\sum_i\ip{f_{iii}}{f_{iii}}
+6\sum_{j\ne k}\ip{f_{jjk}}{f_{kkk}}
+9\sum_{j\ne k}\ip{f_{jkk}}{f_{jkk}}
+6\sum_{k<\ell, j\ne k,\ell}
+12\sum_{j<k<\ell}
\ip{f_{jk\ell}}{f_{jk\ell}}
\bigg).
\end{equation}
Similarly, application of \eqref{moment_integral_4_on_S3} to \eqref{a-a_nu-hot} implies that the $f_{\ast\ast\ast}$ terms in $\Res_{M,x}^{\rm loc}(-8)$ are given by 
\begin{equation}\label{f***-terms_R_M_x_-8}
\frac{\pi^2}{72}\bigg(
\sum_i\ip{f_{iii}}{f_{iii}}
+3\sum_{j\ne k}\ip{f_{jkk}}{f_{jkk}}
+6\sum_{j<k<\ell}
\ip{f_{jk\ell}}{f_{jk\ell}}
\bigg). 
\end{equation}

Assume further that $M$ is a hypersurface. 
Application of \eqref{moment_integral_6_on_S3} and \eqref{moment_integral_4_on_S3} to 
\eqref{a_nu_-8-1} - \eqref{a-a_nu-hot}, together with \eqref{f***-terms_R_M_nu_x_-8} and \eqref{f***-terms_R_M_x_-8} implies 
\begin{lemma}\label{lemma_residue_H_Sc_graph}
Let $M$ be a $4$-dimensional hypersurface. 
Suppose $M\ni0$ and $T_0M$ is $\RR^4$. 
Suppose $M$ is locally expressed as a graph of a function $f$ on a neighbourhood of $0$ of $T_0M$. 
Assume that $\{\partial/\partial u_i\}$ is an orthonormal basis of $T_0M$, each $\partial/\partial u_i$ belonging to a principal direction of $M$ at $0$. 
Let $H$ be $\kappa_1+\kappa_2+\kappa_3+\kappa_4$. 
Then, at the origin we have 
\begin{eqnarray}
\displaystyle \frac{1536}{\pi^2}\,\Res_{M,x}^{\rm loc}(-8)&\eao & -63\sum_i\kappa_i^4 -26\sum_{i<j}\kappa_i^2\kappa_j^2
+12\sum_{i\ne j}\kappa_i\kappa_j^3+20\sum_{i<j,k\ne i,j}
\kappa_i\kappa_j\kappa_k^2 +24\kappa_1\kappa_2\kappa_3\kappa_4 \qquad {}\label{R_kappa} \\
&&\displaystyle +\frac{64}3\sum_if_{iii}^2+64\sum_{i\ne j}f_{iij}^2+128\sum_{i<j<k}f_{ijk}^2  \label{R_-8_c-part}  \\
&&\displaystyle +8\sum_i(4\kappa_i-H)f_{iiii}+16\sum_{i<j}(2\kappa_i+2\kappa_j-H)f_{iijj}\,, \label{R_-8_d-part} 
\end{eqnarray}
\begin{eqnarray}
\displaystyle \frac{1536}{\pi^2}\,\Res_{M,\nu,x}^{\rm loc}(-8)&\eao & 105\sum_i\kappa_i^4+54\sum_{i<j}\kappa_i^2\kappa_j^2+60\sum_{i\ne j}\kappa_i\kappa_j^3+36\sum_{i<j,k\ne i,j}
\kappa_i\kappa_j\kappa_k^2 
+24\kappa_1\kappa_2\kappa_3\kappa_4 \qquad{}\label{R_nu_kappa} \\
&&\displaystyle -\frac{80}3\sum_if_{iii}^2-48\sum_{i\ne j}f_{iij}^2-64 \sum_{i<j<k}f_{ijk}^2 \notag  \\
&&\displaystyle 
-32\sum_{i<j,k\ne i,j}
f_{iik}f_{jjk}-32\sum_{i\ne j}f_{iii}f_{ijj} \label{R_nu_-8_c-part}  \\
&&\displaystyle -8\sum_i(4\kappa_i+H)f_{iiii}-16\sum_{i<j}(2\kappa_i+2\kappa_j+H)f_{iijj}\,,  \label{R_nu_-8_d-part} 
\end{eqnarray}
\end{lemma}

\begin{proposition}
Let $M$ be a $4$-dimensional hypersurface. 
The local ($\nu$-weighted) residues $\Res_{M,x}^{\rm loc}(-8)$ and $\Res_{M,\nu,x}^{\rm loc}(-8)$ cannot be expressed as linear combinations of 
$H^4, H\Delta H, {|\nabla H|}^2, \Delta \Sc, {|\Rm|}^2, {|{\rm Ric}|}^2$, and $\Sc$. 
\end{proposition}

\begin{proof}
First note that \eqref{Sc_square_principal_curv} - 
\eqref{grad_H_sq_m=4}, \eqref{Delta_Sc_center}, \eqref{R_-8_c-part}, \eqref{R_-8_d-part}, \eqref{R_nu_-8_c-part} and \eqref{R_nu_-8_d-part} 
imply that $f_{\ast\ast\ast\ast}$-terms appear only in the residues, $H\Delta H$, and $\Delta \Sc$, and that $\sum_{j<k<l}f_{jkl}^{\,2}$ appear only in the residues and $\Delta \Sc$. 
Assume that the residues are expressed by the curvatures mentioned above. 
Since $\Delta H^2=2(H\Delta H+{|\nabla H|}^2)$, the coefficients of $\Delta \Sc$ are given by \eqref{residue_M^4_nu_-8_modified} and \eqref{residue_M^4_-8_modified}. 
Then the coefficients of $\sum_{j<k<l}f_{jkl}^{\,2}$ do not match, which is a contradiction. 
\end{proof}

\subsection{Residues of spheroids}
Calculating the residues is generally very difficult. 
We introduce hyper-spheroids as an example where some of the residues can be calculated. 

\begin{definition}\label{def_hyper-spheroid} \rm 
By an $m$-dimensional {\em $a$-hyper-spheroid} $S_a$ we mean a hyper-ellipsoid with axes of length $(1,\dots,1,a)$; 
\[
S_a\,:\,{x_1}^2+\dots+{x_m}^2+\frac{{x_{m+1}}^2}{a^2}=1.
\] 
\end{definition}

We use the parametrization 
\begin{equation}
p(a;\theta)=
\left(\begin{array}{l}
\sin\theta_1\dots\sin\theta_{m-1}\cos\theta_m\\
\sin\theta_1\dots\sin\theta_{m-1}\sin\theta_m\\
\sin\theta_1\dots\cos\theta_{m-1}\phantom{\sin\theta_m}\\
\vdots\\
\sin\theta_1\cos\theta_2\\
a\cos\theta_1
\end{array}\right) \qquad
\left(
\begin{array}{c}
\theta=(\theta_1,\dots,\theta_m),\\[1mm]
\begin{array}{rcl}
\theta_1,\dots,\theta_{m-1}&\in&\left[0,\pi\right],  \\[1mm]
\theta_m&\in&[0,2\pi)
\end{array}
\end{array}
\right). 
\notag
\end{equation}
Then 
\[
G=\left(\langle p_i, p_j\rangle\right)
={\rm diag}\biggl(a^2\sin^2\theta_1+\cos^2\theta_1, \sin^2\theta_1,\dots,\prod_{i=1}^{m-1}\sin^2\theta_i
\biggr),
\]
and the volume element $dv$ is given by 
\begin{equation}\label{vol_ele_spheroid}
dv=\sqrt{a^2\sin^2\theta_1+\cos^2\theta_1}\, \sin^{m-1}\theta_1\sin^{m-2}\theta_2\dots\sin\theta_{m-1} \,d\theta_1\dots d\theta_m.
\end{equation}
Since the outer unit normal $\nu$ to $S_a$ is given by 
\begin{equation}\label{n_a_theta}
\nu(a;\theta)=\frac{p(1/a;\theta)}{|p(1/a;\theta)|}, 
\notag
\end{equation}
we have
\[
-G^{-1}(\langle p_i,\nu_j\rangle)
=-\frac1{|p(1/a;\theta)|}\,{\rm diag}\biggl(\frac1{\langle p_1, p_1\rangle},1,\dots,1\biggr),
\]
which implies that the principal curvatures of $S_a$ at point $p(a;\theta_1,\dots,\theta_m)$ are given by 
\begin{equation}\label{principal_curvatures_S_a}
\ka_1=-\frac a{{\left(a^2\sin^2\theta_1+\cos^2\theta_1\right)}^{3/2}}, \>\>
\ka_2=\dots=\ka_m=-\frac a{\sqrt{a^2\sin^2\theta_1+\cos^2\theta_1}}. 
\end{equation}

We compute the residues of a $4$-dimensional hyper-spheroid at $-8$. 

\begin{lemma}\label{a-hyper-spheroid-residues}
Let $S_a$ be a $4$-dimensional $a$-hyper-spheroid. The $\nu$-weighted residue and the residue at $z=-8$ are given by 
\begin{equation}\label{formulae_a-spheroid}
\begin{array}{rcl}
\Rd{-8}{S_a,\nu}&=&\displaystyle \frac{\pi^4}{384}\cdot \frac{105a^6-50a^4-24a^2-16 
\,+105a^6\left(a^2-\displaystyle \frac{8}{7}\right)\omega(a)}{a^4(a^2-1)\phantom{\tilde l}\!\!} \\[5mm]
\Rd{-8}{S_a}&=&\displaystyle \frac{\pi^4}{40320}\cdot \frac{1241a^8-1346a^6-2232a^4+1648a^2-256 -1575a^8\left(a^2-\displaystyle \frac{8}{5}\right)\omega(a)}{a^6(a^2-1)\phantom{\tilde l}\!\!}
\end{array}
\end{equation}
when $a\ne1$, where $\omega(a)$ 
is given by 
\begin{equation}\label{def_omega}
\omega(a)=\left\{
\begin{array}{cl}
\displaystyle \frac{\displaystyle \log\biggl(\frac{1+\sqrt{1-a^2}}{a}\,\biggr)}{\sqrt{1-a^2}} & \hspace{0.7cm} 0<a<1, \\[5mm]
\displaystyle \frac{\arctan\sqrt{a^2-1}}{\sqrt{a^2-1}} &\hspace{0.7cm} a>1.
\end{array}
\right.
\end{equation}
\\ \indent
When $a=1$, that is, when $S_a$ is a round unit $4$-sphere $S^4$, 
\begin{equation}\label{unit_4-sphere}
\Rd{-8}{S_1,\nu}=\frac{2\pi^4}3 \>\>\mbox{ and }\>\> \Rd{-8}{S_1}=0. 
\end{equation}
\end{lemma}

Graphs of the above residues as functions of $a$ are illustrated in Figures \ref{ROmega_m8} and \ref{RM_m8}. 

\begin{proof} 
We use Lemma \ref{lemma_residue_H_Sc_graph}. 
We apply rotation and translation in $(x_4,x_5)$-plane to express a neighbourhood of $p(a;\theta_1,0,0,0)$ of $S_a$ as a graph of a function $f$. 
We have 
\begin{equation}\label{coeff_f_ijk}
f_{ii4}=\frac{a(a^2-1)\cos\theta_1\sin\theta_1}{{(\cos^2\theta_1+a^2\sin^2\theta_1)}^2} \>\> (i\ne4), \quad
f_{444}=3\frac{a(a^2-1)\cos\theta_1\sin\theta_1}{{(\cos^2\theta_1+a^2\sin^2\theta_1)}^3}, \quad
f_{ijk}=0 \>\> (\mbox{otherwise}).
\quad{}
\end{equation}
Substitute \eqref{principal_curvatures_S_a} and \eqref{coeff_f_ijk} to 
\eqref{Delta_H_m=4}, \eqref{Delta_Sc_1} - \eqref{Delta_Sc_3}, and 
\eqref{R_kappa} - 
\eqref{R_nu_-8_d-part}, and then use \eqref{residue_M^4_nu_-8_modified} and \eqref{residue_M^4_-8_modified} to delete $f_{****}$ terms of $\Res_{S_a,\nu}$ and $\Res_{S_a}$. 
The calculation can be reduced to the integration in a single variable $\theta_1$ as follows. 
\begin{equation}\label{f_Rnu_S_a}
\begin{array}{rcl}
\Rd{-8}{S_a,\nu}&=&\displaystyle  \frac{\pi^4}{768}\int_0^\pi \frac{a^2}{{\big(a^2\sin^2\theta_1+\cos^2\theta_1\big)}^6} 
\biggl[\big(a^2-1\big)^4\big(105a^2-120)\cos^8\theta_1
\\[6mm]
&&\displaystyle 
-4\big(a^2-1\big)^3\big(105a^4-33a^2+66\big)\cos^6\theta_1
+6\big(a^2-1\big)^2\big(105a^6+54a^4+85a^2-44\big)\cos^4\theta_1
\\[4mm]
&&\displaystyle 
-\big(420a^{10}+144a^8-336a^6-408a^4+300a^2-120\big)\cos^2\theta_1
\\[4mm]
&&\displaystyle 
+a^2\big(105a^8+228a^6-18a^4+84a^2-15\big)
\biggr] 
\sqrt{a^2\sin^2\theta_1+\cos^2\theta_1}\,\sin^3\theta_1\,d\theta_1.
\end{array}
\end{equation}
\begin{equation}\label{f_R_S_a}
\begin{array}{rcl}
\Rd{-8}{S_a}&=&\displaystyle \frac{\pi^4}{768}\int_0^\pi 
\frac{3\,a^2(a^2-1)^2\sin^2\theta_1}{{\big(a^2\sin^2\theta_1+\cos^2\theta_1\big)}^6} 
\biggl[\big(a^2-1\big)^2\big(5a^2-8\big)\cos^6\theta_1
 \\[6mm]
&&\displaystyle  \quad  
+\big(-15a^6+{22}a^4-55a^2+{48}\big)\cos^4\theta_1 
+\big(15a^6+10a^4+{31}a^2-8\big)\cos^2\theta_1  \hspace{2.5cm}{}
     \\[2mm]
&&\displaystyle \quad -5a^6-{14}a^4+3a^2  \biggr] 
\sqrt{a^2\sin^2\theta_1+\cos^2\theta_1}\,\sin^3\theta_1\,d\theta_1,
\end{array}
\end{equation}

First assume $a\ne1$. 
Computation using Maple implies that the antiderivatives of the integrands of \eqref{f_Rnu_S_a} and \eqref{f_R_S_a} can be explicitly expressed by elementary functions including $\omega(a)$ albeit lengthy, from which the conclusion follows.

The case when $a=1$ can be obtained by substituting $a=1$, $\int_0^\pi \sin^3\theta\,d\theta=4/3$ and $\int_0^\pi \cos^2\theta\sin^3\theta\,d\theta=4/15$ to \eqref{f_Rnu_S_a} and \eqref{f_R_S_a}. 
We remark that $\Rd{-8}{S_1}=0$ also follows from 
\[
B_{S^4}(z)=2^{z+4}o_4o_3 B\left(\frac{z}2+2,2\right), 
\]
which was given by Fuller-Vemuri \cite{FV}. 
\end{proof}

\section{Residues and relative residues of compact bodies}\label{section_compact_bodies}

\subsection{Preliminaries}
We introduce some of the basics from previous studies from \cite{OS2}. 

Let $\Omega$ be a smooth compact body, i.e. a compact $n$-dimensional submanifold of $\RR^n$ with a smooth boundary $\pO$. 
We study the case when the weight $\lambda$ is a constant function $1$, namely, 
\[
I_\Omega(z)=
\iint_{\dom\times\dom}{|x-y|}^z\,dxdy. 
\]
It is well-defined on $\R z>-n$ and is a holomorphic function of $z$ there. 
Since 
\begin{eqnarray}
\displaystyle (z+n)|x-y|^z &=& \displaystyle \textrm{div} _y\left[\,|x-y|^z(y-x)\,\right] \label{div_y}, \\[1mm]
(z+2)\langle |x-y|^{z}(y-x), \uon{y}\rangle
&=&\displaystyle \textrm{div} _x\left[-|x-y|^{z+2}\,\uon{y}\right], \nonumber
\end{eqnarray}
where $\uon{y}$ is an outer unit normal vector to $\pO$ at point $y$, 
using Gauss divergence theorem twice we obtain 
\begin{equation}\label{IOmegaboundary}
\begin{array}{rcl}
I_\Omega(z)
&=&\displaystyle \frac{-1}{(z+2)(z+n)}\iint_{\partial\Omega\times \partial\Omega}|x-y|^{z+2}\langle\uon{x},\uon{y}\rangle\,\dx \dy   
%
\end{array}
\end{equation}
when $\R z>-n$ and $z\ne-2$ (\cite{OS2}). 
Therefore if we put 
\begin{equation}\label{def_lambda_nu}
\la_\nu(x,y)=\ip{\nu_x}{\nu_y}
\end{equation}
then 
\begin{equation}\label{IOmegaboundary_nu}
I_\Omega(z)=\frac{-1}{(z+2)(z+n)}\,I_{\pO,\la_\nu}(z+2) \notag
\end{equation}
for $\R z>-n$ and $z\ne-2$. 
Thus, some part of the study of $I_\Omega(z)$ can be reduced to that of $I_{\pO,\la_\nu}(z+2)$. 

Since $I_{\pO,\la_\nu}(z+2) $ can be extended to a meromorphic function as we saw in the previous section, $I_\Omega(z)$ can also be extended to 
a meromorphic function, which we denote by $B_\Omega(z)$, with possible simple poles at $z=-n, -n-(2j+1)$ $(j=0,1,2,\dots)$. 
Let us call $B_\Omega(z)$ {\em meromorphic energy function} or {\em Brylinski's beta function} or {\em regularized Riesz $z$-energy} of $\Omega$. 

For a boundary point $x$ in $\pO$, put 
\[\begin{array}{rcl}
B_{\Omega,x}^{\partial \>{\rm loc}}(z)
=\displaystyle \frac{-1}{(z+2)(z+n)}\int_{\partial\Omega}|x-y|^{z+2}\langle\uon{x},\uon{y}\rangle\,\dy , 
\end{array}
\]
and call it the {\em meromorphic boundary potential function} of $\Omega$. 
Let us denote the residues of $B_\Omega(z)$ and $B_{\Omega,x}^{\partial \>{\rm loc}}(z)$ at $z=k$, $(k=-n,-n-2j-1$, $j=0,1,2,\dots)$ by $\Rsub{\dom}(k)$ and $\Rsub{\dom,x}^{\partial \> \rm loc}(k)$ respectively, and call them the {\em residue} and {\em boundary local residue} of $\Omega$. 
We have 
\[
\Rsub{\dom}(k)=\int_{\pO}\Rsub{\dom,x}^{\partial \> \rm loc}(k)\,\dx .
\]
We remark that the above also holds for $k=-n$ (\cite{OS2}). 

The {\em regularized energy} of $\dom$ is defined in the same way as in \eqref{def_regularized_energy}.

\subsection{Known formulae of residues}
Suppose $\Omega$ is a compact body in $\RR^n$. 

\begin{proposition}\label{residues_Omega_123}{\rm (\cite{OS2})}
The first three residues of a compact body $\dom$ are given by 
\begin{eqnarray}
\Rsub{\dom}(-n)&=&\displaystyle {o_{n-1}}\text{\rm Vol}\,(\dom)=\left\{ 
\begin{array}{ll}
\displaystyle \frac{1}{(n-2)}\iint_{ \pO\times \pO}|x-y|^{2-n}\langle\uon{x},\uon{y}\rangle\,\dx \dy 
&\quad (n>2) \\[4mm]
\displaystyle -\iint_{\pO\times \pO}\log|x-y|\,\langle\uon{x},\uon{y}\rangle\,\dx \dy  
& \quad (n=2)
\end{array}
\right. \label{residue_Omega_-n} \\
\Rsub{\dom}(-n-1)&=&-\frac{o_{n-2}}{n-1}\,\textrm{\rm Vol}(\pO), \label{residue_Omega_-n-1}\\
\Rd{-n-3}{\Omega}&=&\displaystyle \frac{o_{n-2}}{24(n^2-1)}\int_{\pO} \left(3H^2-2\Sc\right)\dx  \notag  \\[1mm] 
&=&\displaystyle \frac{o_{n-2}}{24(n^2-1)}\int_{\pO} \left(2{\Vert h\Vert}^2+{\vert H\vert}^2\right)\dx , \label{residueOmega-n-3_BH}
\end{eqnarray}
where $\Sc$ is the scalar curvatures, $\Sc=2\sum_{1\le i<j\le n-1}\kappa_i\kappa_j$, where $\kappa_i$ $(1\le i\le n-1)$ are the principal curvatures of $\pO$, $H$ is $n-1$ times the mean curvature $H=\sum_{i=1}^{n-1}\kappa_i$, 
and ${\Vert h\Vert}^2$ is the Hilbert-Schmidt norm of the second fundamental form. 
\end{proposition}

\subsection{\M invariance}
%
We show the \M invariance of the residue at $z=-2\dim\dom$. 

When $z$ is not a pole of $B_\dom$, $B_\dom(z)$ is a homogeneous function of $\dom$ with degree $z+2n$, i.e. $B_{c\dom}(z)=c^{z+2n}B_\dom(z)$. 
When $z$ is a pole of $B_\dom$, the residue $\Res_\dom$ inherits the homogeneity property, whereas the regularized energy $E_\dom(z)$ does not unless the residue vanishes identically. 

We show that the same phenomenon occurs for \M invariance. 
The degree of homogeneity implies that the \M invariance can occur only when $z=-2n$. 
We showed that if $n$ is even then $B_\dom(-2n)$ is \M invariant (\cite{OS2'}). 

\begin{theorem}\label{thm_M_inv_body}
If $n$ is odd then the residue at $z=-2n$ is \M invariant, i.e. $\Rsub{T(\dom)}(-2n)=\Rsub{\dom}(-2n)$ for any \M transformation $T$ as far as $T(\dom)$ remains compact. 
\end{theorem}

\begin{proof}
This theorem can be proved by replacing $M$ with $\Omega$ and $m$ with $n$ in the proof of Theorem \ref{M-inv-res-closed}. 

The non-trivial difficulty in this case is to show the smoothness of 
\[
\Psi_z(t)=\iint_{(\Omega\times\Omega)\cap\Delta_t}\la_z(x,y)\,dxdy
\]
on $[0,d]$ for some $d>0$, which has been done in \cite{OS2'}. 
\end{proof}

If $n$ is odd then the regularized $(-2n)$-energy $E_\dom(-2n)$ is not scale invariant (and hence not \M invariant) as was shown in the proof of Theorem 4.11 of \cite{OS2}.

\subsection{Independence of $B_\dom(z)$ and $B_{\pO}(z)$}\label{subsubsection_indep_B_dom_B_pO}

We show that $B_\dom(z)$ and $B_{\pO}(z)$ are independent as functionals from the space of compact bodies to the space of meromorphic functions. 

First recall that the meromorphic energy functions (Brylinski's beta functions) of a unit sphere and a unit ball are given by 
\begin{eqnarray}
B_{S^{n-1}}(z)&=& \displaystyle 2^{z+n-2}o_{n-1}o_{n-2}\,B\left(\frac{z+n-1}2,\frac{n-1}2\right), \label{beta_sphere} \notag \\
B_{B^n}(z)&=&\displaystyle \frac{2^{z+n}o_{n-1}o_{n-2}}{(n-1)(z+n)}
\,B\left(\frac{z+n+1}2,\frac{n+1}2\right)  \label{beta_ball} \\
&=&\displaystyle \frac{2^{z+n}o_{n-1}o_{n-2}}{(n-1)}
\frac{\displaystyle \Gamma\left(\frac{z+n+1}2\right)\Gamma\left(\frac{n+1}2\right)}{\displaystyle (z+n)\,\Gamma\left(\frac{z}2+n+1\right)},
\end{eqnarray}
where $B(\>,\>)$ denotes the beta function. The first equality is due to \cite{B} when $n=2$ and \cite{FV} in general. For the second, see for example, \cite{Mi,HR}. 

\begin{lemma}\label{lemma_indep_-2alpha0}
Assume $n\ge2$. 
\begin{enumerate}
\item Let $\a$ be a real number with $-2<\a\le0$. Then for any complex number $w$ there is a pair of compact bodies $\dom_1, \dom_2$ that satisfy 
$B_{\dom_1}(\a)=B_{\dom_2}(\a)$ and $B_{\partial\dom_1}(w)\ne B_{\partial\dom_2}(w)$. 
\item For any complex number $w$ there is a pair of compact bodies $\dom_1, \dom_2$ that satisfy 
$B_{\partial\dom_1}(0)= B_{\partial\dom_2}(0)$ and $B_{\dom_1}(w)\ne B_{\dom_2}(w)$. 
\end{enumerate}
\end{lemma}

\begin{proof}
(1-1) First assume $\a=0$. Put
\[
\begin{array}{rcl}
\dom_1&=&B^n \\
\dom_2=\dom_2(\de)&=&\displaystyle 2^{-1/n}B^n\sqcup\big(2^{-1/n}B^n+\de e_1\big) \qquad \big(\de>2\cdot 2^{-1/n}\big),
\end{array}
\]
where $X+v=\{x+v\,:\,x\in X\}$. 
Then $B_{\dom_1}(0)={({\rm Vol}(B^n))}^2=B_{\dom_2(\de)}(0)$ holds for any $\de$. 
Let $w=a+b\sqrt{-1}$ $(a,b\in\RR)$. It can be divided into the following three cases. 

(i) Suppose $w=0$. Then $B_{\partial\dom_2}(0)=4\cdot (2^{-1/n})^{2n-2}{({\rm Vol}(S^{n-1}))}^2\ne {({\rm Vol}(S^{n-1}))}^2=B_{\partial\dom_1}(0).$

(ii) Suppose $\I w\ne0$. We may assume without loss of generality that $b=\I w>0$. We have 
\begin{equation}\label{B_pO2_w}
B_{\partial\dom_2(\de)}(w)=2\cdot \big(2^{-1/n}\big)^{w+2n-2}B_{S^{n-1}}(w)
+2\int_{2^{-1/n}S^{n-1}}\int_{2^{-1/n}S^{n-1}+\delta e_1}{|x-y|}^w\,\dx \dy .
\end{equation}
Note that the first term of the right hand side is independent of $\de$. 
Since
\[
b\log\big(\delta-2\cdot 2^{-1/n}\big)\le b\log |x-y|=\arg \big({|x-y|}^w\big)
\le b\log\big(\delta+2\cdot 2^{-1/n}\big),
\]
for any positive number $\e$ there is $\delta_0>2\cdot 2^{-1/n}$ such that if $\delta>\delta_0$ then 
\[\arg \big({|x-y|}^w\big)\in (b\log\de-\e,b\log\de+\e) \quad ({\rm mod} \> 2\pi).\]
One can choose a small $\e$ and $\de_1,\de_2>\de_0$ such that 
\[
(b\log\de_1-\e,b\log\de_1+\e)\cap(b\log\de_2-\e,b\log\de_2+\e)=\emptyset,
\]
which implies $B_{\pO_2(\de_1)}(w)\ne B_{\pO_2(\de_2)}(w)$. 
Then at least one of $B_{\pO_2(\de_1)}(w)$ and $B_{\pO_2(\de_2)}(w)$ is different from $B_{\pO_1}(w)$.

(iii) Suppose $\I w=0$ and $w\ne0$. 
The second term of the right hand of \eqref{B_pO2_w} is a strictly increasing (or decreasing) function of $\de$ if $w>0$ (or respectively $w<0$). 
Therefore if $\de_1\ne\de_2$ at least one of $B_{\dom_2(\de_1)}(w)$ and $B_{\dom_2(\de_2)}(w)$ is different from $B_{\dom_1}(w)$. 

\smallskip
(1-2) Next assume $-2<\a<0$. Put 
\[
\begin{array}{rcl}
\dom_1&=&B^n \\
\dom_2=\dom_2(c,\de)&=&\displaystyle cB^n\sqcup (cB^n+\delta e_1) \qquad (\de>2c).
\end{array}
\]
We have 
\begin{equation}\label{B_pO2_c_w}
B_{\dom_2(c,\de)}(\a)=2\cdot {c}^{\a+2n}B_{B^{n}}(\a)
+2\int_{cB^n}\int_{cB^n+\delta e_1}{|x-y|}^\a\,\dx \dy . \notag
\end{equation}
Suppose $\de>2$. 
Since $\a+2n>0$ and $B_{B^{n}}(\a)>0$,  
$B_{\dom_2(c,\de)}(\a)$ is a strictly increasing function of $c$ $(0<c<\de/2)$ for fixed $\de$. 
Since $B_{\dom_2(c,\de)}(\a)$ tends to $+0$ as $c$ approaches $0$ and $B_{\dom_2(1,\de)}(\a)>B_{B^n}(\a)=B_{\dom_1}(\a)$ when $c=1$, there is a unique $c=c(\de)$ that satisfies $B_{\dom_2(c,\de)}(\a)=B_{\dom_1}(\a)$. 
We remark that $c(\de)$ is an increasing function of $\de$ and that $c(\de)$ tends to $2^{-1/(\a+2n)}$ as $\de$ goes to $+\infty$.  

We have 
\begin{equation}\label{B_pO2(c_delta)_w}
B_{\partial\dom_2(c(\de),\de)}(w)=2\cdot {c(\de)}^{w+2n-2}B_{S^{n-1}}(w)
+2\int_{c(\de)S^{n-1}}\int_{c(\de)S^{n-1}+\delta e_1}{|x-y|}^w\,\dx \dy .
\end{equation}
Let the second term of the right hand side be denoted by $I_2(c,\de,w)$. 
Let $w=a+b\sqrt{-1}$ $(a,b\in\RR)$. It can be divided into the following four cases. 

(i) Suppose $w=0$. As $\de$ goes to $+\infty$, $B_{\partial\dom_2(c(\de),\de)}(0)=4\,{c(\de)}^{2n-2}B_{S^{n-1}}(0)$ tends to $4\cdot{2}^{-(2n-2)/(\a+2n)}B_{S^{n-1}}(0)$, which is different from $B_{S^{n-1}}(0)$ since $B_{S^{n-1}}(0)\ne0$ and $\a\ne-n-1$. 

(ii) Suppose $a=\R w<0$. First assume $B_{S^{n-1}}(w)\ne0$. As $\de$ goes to $+\infty$, $B_{\partial\dom_2(c(\de),\de)}(w)$ tends to $2\cdot 2^{-(w+2n-2)/(\a+2n)}B_{S^{n-1}}(w)$, which is equal to $B_{\partial\dom_1}(w)=B_{S^{n-1}}(w)$ if and only if $w=\a+2$, which does not occur since $\a>-2$ and $\R w<0$. 

Next assume $B_{S^{n-1}}(w)=0$. 
If $\I w=0$ then $I_2(c(\de),\de,w)\ne0$. 
If $\I w\ne0$ then by a similar argument on the argument of ${|x-y|}^w$ as in (ii) of (1-1), for a sufficiently large $\delta$, $I_2(c(\de),\de,w)\ne0$. 
In both cases, $B_{\partial\dom_2(c(\de),\de)}(w)=I_2(c(\de),\de,w)\ne0= B_{\partial\dom_1}(w)$. 

(iii) Suppose $a=\R w>0$. 
Since the absolute value of the first term of \eqref{B_pO2(c_delta)_w} is uniformly bounded for $\de\ge3$, when $\de$ is sufficiently large, $B_{\partial\dom_2(c(\de),\de)}(w)$ is dominated by $I_2(c(\de),\de,w)$, 
which can be approximated by 
\[
\big(\cos(b\log\de)+i\sin(b\log\de)\big)\,\de^a \, {2}^{-\frac{2n-2}{\a+2n}}\big({\rm Vol}(S^{n-1})\big)^2
\]
by a similar argument as in (ii) of (1-1). 
Since its absolute value is an unbounded increasing function of $\de$, one can find a $\de$ such that $|B_{\partial\dom_2(c(\de),\de)}(w)|>|B_{\pO_1}(w)|$.

(iv) Suppose $\R w=0$ and $b=\I w\ne0$. 
When $\de$ is sufficiently large, $B_{\partial\dom_2(c(\de),\de)}(b\,i)$ can be approximated by 
\[
2\cdot 2^{-\frac{2n-2}{\a+2n}}\left(2^{-\frac{b\,i}{\a+2n}}B_{S^{n-1}}(b\,i)+\de^{\,b\,i}
\right),
\]
one can find $\de_1$ and $\de_2$ such that $B_{\partial\dom_2(c(\de_1),\de_1)}(b\,i)\ne B_{\partial\dom_2(c(\de_2),\de_2)}(b\,i)$. 

\medskip
(2) can be proved by the same argument as in (1-1) above. 

\end{proof}

\begin{corollary}\label{independence_B_dom_B_pO}
$B_\dom(z)$ and $B_{\pO}(z)$ are affinely independent. 
\end{corollary}

As for the residues, both $\Rd{-n-1}\dom$ and $\Rd{-n-1}\pO$ are equal to $\mbox{\rm Vol}(\pO)$ up to multiplication by constants, whereas \eqref{residue_closed_-m-2} and \eqref{residueOmega-n-3_BH} imply that $\Rd{-n-3}\dom$ and $\Rd{-n-3}\pO$, which have the same order of homogeneity, are independent.

\subsection{Relative energies and relative residues of compact bodies}\label{subsect_relative}
Let $\Omega$ be a smooth compact body, i.e. a compact $n$-dimensional submanifold of $\RR^n$ with a smooth boundary $\pO$.  
Put 
\begin{equation}\label{Irelative}
I_{\Omega,\partial\Omega}(z)=\int_{\pO} \dx  \int_\O |x-y|^z\,\dy .
\end{equation}
It is well-defined and holomorphic on $\R z>-n$. 
Just like \eqref{IOmegaboundary}, \eqref{div_y} implies that 
$I_{\Omega,\partial\Omega}(z)$ can be expressed by the boundary integral;  
\begin{equation}\label{mixed_energy}
I_{\Omega,\pO}(z)=\frac1{z+n}\iint_{\pO\times \pO}|x-y|^z\langle y-x, \uon{y}\rangle\,\dx \dy 
\end{equation}
on $\R z>-n$. 
By extending the domain by means of analytic continuation, we obtain a meromorphic function defined on the whole complex plane $\CC$ with only simple poles. 
We call it the {\em  meromorphic relative energy function of $\Omega$} and denote it by $B_{\Omega,\pO}(z)$. 

\medskip
The expression \eqref{mixed_energy} of $I_{\Omega,\pO}(z)$ has two kinds of local version since the weight $\langle y-x, \uon{y}\rangle$ is not symmetric in $x$ and $y$; 
\begin{equation}\label{relativeBboundarylocal}
B_{\Omega,\pO,x}^{\partial\>{\rm loc}}(z)=\frac1{z+n}\int_{\pO}|x-y|^z\langle y-x, \uon{y}\rangle\,\dy  \quad x\in\pO
\notag
\end{equation}
and\footnote{We put the symbol $\partial$ in $B_{\Omega,\pO}^{\partial\>{\rm loc}}(x;z)$ since $x$ was originally a point in $\pO$ in \eqref{Irelative}. }
\begin{equation}\label{relativeBlocal}
B_{\Omega,\pO,y}^{{\rm loc}}(z)=\frac1{z+n}\int_{\pO}|x-y|^z\langle y-x, \uon{y}\rangle\,\dx , \quad y\in\pO,
\notag
\end{equation}
where the right hand sides are understood to be meromorphic functions obtained by means of analytic continuation as in the previous cases. 

\begin{definition} \rm 
We denote the residues of $B_{\Omega,\pO}(z)$, $B_{\Omega,\pO,x}^{\partial\>{\rm loc}}(z)$, and $B_{\Omega,\pO,y}^{{\rm loc}}(z)$ at $z=k$ by $\Rsub{\Omega, \pO}(k)$, $\Rsub{\Omega,\pO,x}^{\partial\>{\rm loc}}(k)$, and $\Rsub{\Omega,\pO,y}^{{\rm loc}}(k)$ respectively, and call them the {\em ralative residue}, {\em boundary local relative residue} and {\em local relative residue} of $\Omega$ respectively. 
\end{definition}
We have 
\begin{equation}\label{relative_residu_in_two_ways}
\Rsub{\Omega, \pO}(k)=\int_{\pO} \Rsub{\Omega,\pO,x}^{\partial\>{\rm loc}}(k)\, \dx =\int_{\pO} \Rsub{\Omega,\pO,y}^{{\rm loc}}(k)\, \dy . 
\notag
\end{equation}

Define $\lambda_{{\rm rel}}^\partial$ and $\lambda_{{\rm rel}}$ from $\pO\times\pO$ to $\RR$ by \[\lambda_{{\rm rel}}^\partial(x,y)=\langle y-x, \uon{y}\rangle, \quad \lambda_{{\rm rel}}(x,y)=\langle x-y, \uon{x}\rangle.\] 
Both are $O(|x-y|^2)$, which implies $\psi_{\pO,\lambda_{{\rm rel}}^\partial,x}(t)$ and $\psi_{\pO,\lambda_{{\rm rel}},x}(t)$ given by \eqref{psiMrhox} satisfy 
\[
{\psi_{\pO,\lambda_{{\rm rel}}^\partial,x}}^{(n-1)}(0)=0, \quad {\psi_{\pO,\lambda_{{\rm rel}},x}}^{(n-1)}(0)=0.
\]
Then \eqref{residue_I_phi}  
implies that the residues of $B_{\Omega,\pO}(z)$, $B_{\Omega,\pO,x}^{\partial\>{\rm loc}}(z)$, and $B_{\Omega,\pO,y}^{{\rm loc}}(z)$ at $z=-(n-1)$ are equal to $0$. 
\begin{lemma} 
None of $B_{\Omega,\pO}(z)$, $B_{\Omega,\pO,x}^{\partial\>{\rm loc}}(z)$, and $B_{\Omega,\pO,y}^{{\rm loc}}(z)$ have poles at $z=-n+1$. 
\end{lemma}
\begin{corollary}
The meromorphic relative energy function $B_{\Omega,\pO}(z)$ and the local versions $B_{\Omega,\pO,x}^{\partial\>{\rm loc}}(z)$ and $B_{\Omega,\pO,y}^{{\rm loc}}(z)$ have possible simple poles at $z=-n$ and $z=-n-(2j+1)$ $(j=0,1,2,\dots)$. 
\end{corollary}

\subsection{Relative energy in terms of parallel bodies}\label{subsubsection_parallel_bodies}
%
We introduce parallel bodies since the relative energy and residues can be obtained from those of a compact body $\dom$ by regarding $\dom\times\pO$ as half the ``{\sl derivative}'' of $\dom\times\dom$ with respect to parallel body operation. 
Put $\displaystyle \Omega_\delta=\bigcup_{x\in \Omega}B^n _\delta(x)$ for $\delta\ge0$. It is called (outer) {\em $\delta$-parallel body} in convex geometry when $\Omega$ is convex. 
Since $\pO$ is compact, there is a positive number $\de_0$ such that 
\[
F\colon\pO\times[0,\de_0]\ni(x,t)\mapsto x+t\uon{x}\in\RR^n
\]
is an embedding. 
If $0\le\de\le\de_0$ then 
$\partial\Omega_\de$ is a regular hypersurface. 

We introduce a local expression of the Steiner formula, Weyl tube formula, and Gray's half-tube formula. 
Let $\ka_1,\dots,\ka_{n-1}$ be the principal curvatures of $\pO$, and \[
S_k(x)=\sum_{|I|=k}\,\prod_{i\in I}\ka_i(x) \quad (0\le k\le n-1)
\]
be the $k$-th elementary symmetric polynomials of $\ka_i(x)$. 
Then it is known that 
the volume element of $\pO_t=\{x+t\uon{x}\,:\,x\in\pO\}$ $(0\le t\le \de_0)$ is gievn by 
\begin{equation}\label{vol_elem_pO_t}
\left(\prod_{i=1}^{n-1}(1-t\ka_i)\right)dv
=\left(\sum_{k=0}^{n-1} (-1)^kS_{k}(x)t^k\right) \dx 
\end{equation}
where $dv $ is the volume element of $\pO$. 

\begin{remark}\label{remark_convention_orientation_normal}
\rm 
The sign of the principal curvatures depends on the choice of the unit normal vector.  
We use the {\sl outward} unit normal vector $\uon{x}$. 
In our convention the principal curvatures of the unit sphere are all $-1$ (cf. \cite{ON}).
\end{remark}

Integrating \eqref{vol_elem_pO_t} from $t=0$ to $r$ $(0\le r\le \de_0)$ we have 
\begin{equation}\label{Steiner_principal_curv}
\mbox{Vol}(\dom_r)=\mbox{Vol}(\dom)+\sum_{k=1}^n\left(\frac{(-1)^{k-1}}k\int_{\pO}S_{k-1}(x)\dx \right)r^k.
\end{equation}
Put
\begin{equation}\label{C_k_Omega}
C_k(\dom)=\frac{(-1)^{n-1-k}}{(n-k)\omega_{n-k}}\int_{\pO} S_{n-1-k}(x)\,\dx  \quad (0\le k\le n-1), \quad C_n(\dom)=\mbox{Vol}(\dom),
\end{equation}
where $\omega_j$ is the volume of the $j$-dimensional unit ball, then 
\begin{equation}\label{Steiner_smooth}
\mbox{Vol}(\dom_r)=
\sum_{k=0}^n\omega_kC_{n-k}(\dom)r^k \quad (0\le r\le\de_0).
\end{equation}
The $C_k(\dom)$ are called {\em Lipschitz-Killing curvatures}. 
Note that  
\begin{equation}\label{C_k_Omega_n-1_0}
C_{n-1}(\dom)=\frac12\mbox{Vol}(\pO), \> 
C_0(\dom)=\chi(\dom),
\notag
\end{equation}
where the last equality follows from Gauss-Bonnet-Chern theorem. 

If $\dom$ is convex, $C_k(\dom)$ are equal to the {\em intrinsic volumes} $V_k(\dom)$ which are defined by the {\em Steiner formula} in convex geometry\footnote{The $V_k(\dom)$ can also be interpreted as ``{\sl Quermassintegral''}, the integral of the $(n-k)$ dimensional volumes of the image of projection to $(n-k)$ dimensional hyperplanes, where the integration is done on the Grassmannian; 
\[
V_{k}(\dom)=\left(\begin{array}{c}n \\ k\end{array}\right)
\frac{\omega_n}{\omega_k\omega_{n-k}}\int_{G(n,n-k)}{{\rm Vol}_{n-k}}(\pi_L(\dom))d\mu^n_{n-k}(L) 
\]
cf. \cite{SW}, page 201. We do not use the above interpretation in this article. }; 
\begin{equation}\label{Steiner_formula}
\mbox{Vol}(\dom_r)=\sum_{k=0}^n\omega_kV_{n-k}(\dom)r^k. 
\end{equation}
We remark that we do not need smoothness in the Steiner formula. 
See Subsection \ref{subsubsection_Lipschitz-Killing_curv_intrinsic_volum} for the properties of intrinsic volumes. 
The reader is referred to, for example, \cite{Sch,M,RZ} for the roles of Lipschitz-Killing curvatures (intrinsic volumes) in differential geometry or convex geometry.

\begin{theorem}\label{thm_relative_diff_body}
The meromorphic relative energy function is half the derivative of the meromorphic energy function of the $\e$-parallel body by $\e$ at $\e=0$;  
\begin{equation}\label{relative_energy}
\displaystyle B_{\Omega,\pO}(z)=\frac12\left.\pd{}{\e}\,B_{\Omega_\e}(z)\right|_{\e=0} 
\end{equation}
when $z$ is not a pole of $B_{\Omega_\e}$, i.e. if $z\ne-n,-n-1-2j$ $(j=0,1,2,\dots)$. 
\end{theorem}

\begin{proof}
(1) Assume $0\le\e\le\de_0$ and $\R z>0$. 
Put $N_\e=F(\pO\times(0,\e])=\dom_\e\setminus\dom$. 
Then 
\begin{equation}\label{f_der_B_Omega_epsilon}
\begin{array}{l}
\displaystyle \iint_{\dom_\e\times\dom_\e}{|x-y|}^zdxdy-\iint_{\dom\times\dom}{|x-y|}^zdxdy 
\displaystyle =2\int_\dom \left(\int_{N_\e}{|x-y|}^zdx\right)dy
+\iint_{N_\e\times N_\e}{|x-y|}^zdxdy
\end{array}
\end{equation}
Let $\phi_1(\e)$ and $\phi_2(\e)$ be the first and second terms of the right hand side of the above respectively. 

First note 
\[
2\int_\dom dy \int_{N_\e}{|x-y|}^zdx
=2\int_\dom dy\int_0^\e dr \int_{\pO}{|x+r\uon{x}-y|}^z\left(\prod_{i=1}^{n-1}(1-r\ka_i(x))\right)\dx .
\]
The mean value theorem implies that for each $y$ in $\dom$ there is a positive number $h$ $(0<h<1)$ such that 
\[
\begin{array}{l}
\displaystyle \frac1\e \int_0^\e dr \int_{\pO}{|x+r\uon{x}-y|}^z\left(\prod_{i=1}^{n-1}(1-r\ka_i(x))\right)\dx 
=\int_{\pO}{|x+h\e\uon{x}-y|}^z\left(\prod_{i=1}^{n-1}(1-h\e\ka_i(x))\right)\dx , 
\end{array}
\]
which implies
\[
\lim_{\e\to0}\frac1\e \int_0^\e dr \int_{\pO}{|x+r\uon{x}-y|}^z\left(\prod_{i=1}^{n-1}(1-r\ka_i(x))\right)\dx 
=\int_{\pO}{|x-y|}^z \dx.
\]
Suppose $\R z>0$. The right hand side above is uniformly bounded and continuous in $y$. 
As $\dom$ is compact, 
\begin{equation}\label{f_der_B_Omega_epsilon_1}
\lim_{\e\to0}\frac{\phi_1(\e)}\e=2\int_\dom dy\int_{\pO}{|x-y|}^z \dx .
\end{equation}
On the other hand, since $\phi_2(\e)$ can be expressed by 
\[
\phi_2(\e)=\int_0^\e dr_1\int_0^\e dr_2\iint_{\pO\times \pO}
{|x-y+r_1\uon{x}-r_2\uon{y}|}^z
\left(\prod_{i=1}^{n-1}(1-r_1\ka_i(x))\right)
\left(\prod_{j=1}^{n-1}(1-r_2\ka_j(y))\right)\dx \dy 
\]
and the integrand is uniformly bounded, we have $\phi_2(\e)=O(\e^2)$ and hence 
\begin{equation}\label{f_der_B_Omega_epsilon_2}
\lim_{\e\to0}\frac{\phi_2(\e)}\e=0.
\end{equation}

\eqref{f_der_B_Omega_epsilon}, \eqref{f_der_B_Omega_epsilon_1}, and \eqref{f_der_B_Omega_epsilon_2} imply
\[
\left.\pd{}{\e}\,B_{\Omega_\e}(z)\right|_{\e=0}=2\int_\dom dy\int_{\pO}{|x-y|}^z \dx 
=2B_{\Omega,\pO}(z)
\]
on $\R z>0$. 
By analytic continuation, the above equality can be extended to the whole complex plane except for the poles of $B_{\Omega_\e}$, which completes the proof. 
\end{proof}

Since $B_{B^n_\e}(z)=(1+\e)^{z+2n}B_{B^n}(z)$, \eqref{beta_ball} implies 
\begin{corollary}\label{cor_rel_ball}
The meromorphic relative energy function of a unit ball is given by 
\[
B_{B^n,S^{n-1}}(z)=\frac{2^{z+n-1}(z+2n)\,o_{n-1}o_{n-2}}{(n-1)(z+n)}\,B\biggl(\frac{z+n+1}2,\,\frac{n+1}2\biggr).
\]
\end{corollary}

\subsection{Relative residues and mean curvature}\label{subsect_relative_mean_curv}
%
Theorem \ref{thm_relative_diff_body} implies 
\begin{corollary}\label{cor_residue_ralative}
The relative residues are half the derivatives of those of the $\e$-parallel body by $\e$ at $\e=0$;
\begin{equation}\label{relative_residue_der}
\displaystyle \Res_{\Omega,\pO}(k)=\frac12\left.\pd{}{\e}\,\Res_{\Omega_\e}(k)\right|_{\e=0} \qquad (k=-n, -n-(2j+1), \>\> j=0,1,2,\dots). 
\end{equation}
\end{corollary}
\begin{proof}
\eqref{relative_energy} implies 
\begin{eqnarray}
\Res_{\dom,\pO}(k)&=&\displaystyle \frac1{2\pi i}\oint_{|w-k|=\rho}B_{\dom,\pO}(w)dw \nonumber 
=\displaystyle \frac1{2\pi i}\oint_{|w-k|=\rho}\frac12\left.\pd{}{\e}\,B_{\Omega_\e}(w)\right|_{\e=0}dw
\label{relative_residue_1}
\end{eqnarray}
for $\rho$ $(0<\rho<1)$. 
Since both $B_{\dom_\e}(w)$ and $(\partial B_{\dom_\e}/\partial \e)(w)$ are continuous and integrable on $S^1_\rho(k)=\{w\in\CC\,:\,|w-k|=\rho\}$ for each $\e$ with $0\le\e\le\de_0$, 
\[
\Res_{\dom,\pO}(k)= \left. \frac12\pd{}{\e}\left(\frac1{2\pi i}\oint_{|w-k|=\rho} B_{\Omega_\e}(w)dw\right)\right|_{\e=0}
= \left. \frac12\pd{}{\e}\Res_{\dom_\e}(k)\right|_{\e=0},
\]
which shows \eqref{relative_residue_der}. 
\end{proof}

As a corollary, one obtains $\Rd{-n}{\O,\pO}$ and $\Rd{-n-1}{\O,\pO}$ using \eqref{Steiner_principal_curv}, \eqref{residue_Omega_-n} and \eqref{residue_Omega_-n-1}, 
and $\Rd{-n-3}{\O,\pO}$ using \eqref{vol_elem_pO_t} and \eqref{residueOmega-n-3_BH}. 
These results can be refined by finding (boundary) local residues as follows.

\begin{theorem}\label{prop_relative_residues}
\begin{enumerate}
\item The boundary local relative residue at $z=-n$ is given by 
\[
\Rsub{\Omega,\pO,x}^{\partial\>{\rm loc}}(-n)=\frac{o_{n-1}}2.
\]
Therefore  
\begin{equation}\label{relative_residue_-n}
\Rd{-n}{\O,\pO}=\frac{o_{n-1}}2\,\mbox{\rm Vol}(\pO).
\end{equation}
\item The boundary local and local relative residues at $z=-n-1$ are given by 
\begin{equation}\label{local_relative_residue_mean_curv}
\Rsub{\Omega,\pO,x}^{\partial\>{\rm loc}}(-n-1)
=\Rsub{\Omega,\pO,x}^{{\rm loc}}(-n-1)
=\frac{o_{n-2}}{2(n-1)}\,H.
\end{equation}
Therefore  
\begin{equation}\label{relative_residue_-n-1}
\Rd{-n-1}{\O,\pO}= \frac{o_{n-2}}{2(n-1)} \int_{\pO}H(x)\,\dx .
\end{equation}
\item The boundary local and local relative residues at $z=-n-3$ satisfy 
\[
\begin{array}{rcl}
\displaystyle \frac12\left(-\Rsub{\Omega,\pO,x}^{\partial\>{\rm loc}}(-n-3)+3\Rsub{\Omega,\pO,x}^{{\rm loc}}(-n-3)\right)
&=&\displaystyle \frac{o_{n-2}}{48(n^2-1)}\biggl( 4\sum_i{\kappa_i}^3-{\Big(\sum_j\kappa_j\Big)}^3\,\biggr),
\end{array}
\]
where $\kappa_j$ are the principal curvatures of $\pO$. 
Therefore
\begin{equation}\label{relative_residue_-n-3}
\Rd{-n-3}{\O,\pO}=\frac{o_{n-2}}{48(n^2-1)}\int_{\pO}\biggl(4\sum_i{\kappa_i}^3-{\Big(\sum_j\kappa_j\Big)}^3\biggr)\,\dx .
\end{equation}
\item The difference between boundary local and local relative residues  at $z=-n-3$ is given by
\[
\Rsub{\Omega,\pO,x}^{\partial\>{\rm loc}}(-n-3)-\Rsub{\Omega,\pO,x}^{{\rm loc}}(-n-3)
=\frac{o_{n-2}}{12(n^2-1)}\,\Delta H. 
\]
\end{enumerate}
\end{theorem}

We remark that the local relative residue at $z=-n$ is not necessarily constant unlike the boundary local relative residue, as can be verified by numerical computation when $\pO$ is an ellipse. 

\begin{proof}
(1) First remark that $\Rsub{\Omega,\pO,x}^{\partial\>{\rm loc}}(-n)$ is scale invariant. By Corollary \ref{cor_residue_ralative}, the first equality of \eqref{residue_Omega_-n} and \eqref{Steiner_principal_curv}, we have $\Rsub{\Omega,\pO}(-n)=(o_{n-1}/2){\rm Vol}(\pO)$. 
When $\dom$ is an $n$-ball, there holds 
\[
\frac{o_{n-1}^{\,2}}2
=\Rsub{B^n,S^{n-1}}(-n)=\int_{S^{n-1}}\Rsub{B^n,S^{n-1},x}^{\partial\>{\rm loc}}(-n)\,dv(x)
=o_{n-1}\,\Rsub{B^n,S^{n-1},x}^{\partial\>{\rm loc}}(-n),
\]
where the first equality follows from Corollary \ref{cor_rel_ball} 
and the last from the symmetry, 
which implies $\Rsub{B^n,S^{n-1},x}^{\partial\>{\rm loc}}(-n)=o_{n-1}/2$. 

Let $\kappa_i$ be principal curvatures of $\pO$ at $x$. 
We can take an $n$-ball $B^n_\rho$ with radius $\rho$ with $\rho<1/\max_i|\kappa_i|$ inside $\dom$ that is tangent to $\pO$ at point $x$. 
Put $\dom'=\dom\setminus B^n_\rho$. Then $\pO'=\pO\cup -\partial B^n_\rho$. 
Although $\partial \Omega'$ is not smooth but piecewise smooth, one can still apply Green's theorem, and hence \eqref{div_y} implies that 
\[
\int_{\pO'}{|x-y|}^{-n}\langle y-x,\uon{y}\rangle\,\dy 
=(-n+n)\int_{\dom'}{|x-y|}^{-n}\,\dy =0,
\]
which implies 
\[
\Rsub{\Omega,\pO,x}^{\partial\>{\rm loc}}(-n)=\Rsub{B^n,S^{n-1},x}^{\partial\>{\rm loc}}(-n)=\frac{o_{n-1}}2.
\]

(2), (3) and (4) can be proved by the same stragegy as in Lemma \ref{lem_nu_Res_-m-2}. 
Assume $x=0$, $T_0\pO=\RR^{n-1}$, the coordinate axes are the principal directions of $\pO$ at $0$, and $\pO$ is locally expressed as a graph of a function $f$ defined on a neighbourhood of $0$ of $T_0\pO=\RR^{n-1}$. 
Let $A_f(u)=\sqrt{1+|\nabla f(u)|^2}$ 
be the area density of the graph of $f$. 
The unit normal vectors of $\pO$ at $y=(u,f(u))$ and $x=0$ are given by 
\[
\uon{u}=\frac{(-\nabla f(u),1)}{A_f(u)} 
\quad\mbox{ and }\quad e_n=(0,\dots,0,1)
\]
respectively, which implies 
\begin{eqnarray}
B_{\Omega,\pO,x}^{\partial\>{\rm loc}}(z)&=&\displaystyle \frac1{z+n}\int_{\pO}|y-x|^z\langle y-x, \uon{y}\rangle\,\dy  \notag \\[1mm]
&=&\displaystyle (\mbox{holo.})+\frac1{z+n}\int_{B^{n-1}_\rho}\Big(|u|^2+f(u)^2\Big)^{z/2}\big(f(u)-\langle u, \nabla f(u)\rangle \big)\,du, \notag \\ 
&=&\displaystyle (\mbox{holo.})+\frac1{z+n}\int_0^\rho dr \int_{S^{n-2}} 
r^z{\biggl(1+\frac{f(rw)^2}{r^2}\biggr)}^{z/2}
r^2\bigg(\frac{f(rw)}{r^2}-\frac{\langle w, \nabla f(rw)\rangle}{r}\bigg)\,r^{n-2}\,dv, \qquad\qquad{} \label{rel_B_boundary_loc} \notag \\ 
B_{\Omega,\pO,x}^{{\rm loc}}(z)&=&\displaystyle \frac1{z+n}\int_{\pO}|y-x|^z\langle x-y, \uon{x}\rangle\,\dy \notag \\[1mm]
&=&\displaystyle (\mbox{holo.})+\frac1{z+n}\int_{B^{n-1}_\rho}
\Big(|u|^2+f(u)^2\Big)^{z/2}\bigl(-f(u)\bigr)A_f(u)\,du  \notag \\[1mm]
&=&\displaystyle (\mbox{holo.})-\frac1{z+n}\int_0^\rho dr \int_{S^{n-2}} 
r^z{\biggl(1+\frac{f(rw)^2}{r^2}\biggr)}^{z/2}
r^2\,\frac{f(rw)}{r^2}\,\sqrt{1+{|\nabla f(rw)|}^2}\,r^{n-2}\,dv \qquad{}\label{rel_B_loc} \notag
\end{eqnarray}
for some $\rho>0$, where (holo.) stands for holomorphic part. 
Divide the integrands above by $r^{z+n}$ and expand them in a series in $r$ and let $a_i^\partial(z;w)$ and $a_i(z;w)$ be the coefficients of $r^i$: 
\begin{eqnarray}
\displaystyle {\biggl(1+\frac{f(rw)^2}{r^2}\biggr)}^{z/2}
\bigg(\frac{f(rw)}{r^2}-\frac{\langle w, \nabla f(rw)\rangle}{r}\bigg)
&=&\displaystyle \sum_{i=0}^2a_i^\partial(z;w)\,r^i+O(r^3), \label{a_i_partial_zw} \\[1mm]
\displaystyle -{\biggl(1+\frac{f(rw)^2}{r^2}\biggr)}^{z/2}
\frac{f(rw)}{r^2}\,\sqrt{1+{|\nabla f(rw)|}^2}
&=& \displaystyle \sum_{i=0}^2a_i(z;w)\,r^i+O(r^3). \label{a_i_zw}
\end{eqnarray}
Then by Subsection \ref{method_graph} 
\begin{equation}\label{local_partial_relative_residue_-n-1-j}
\begin{array}{rcl}
\Res_{\Omega,\pO,x}^{\partial\>{\rm loc}}(-n-1-2j)&=&\displaystyle -\frac1{2j+1}
\int_{S^{n-2}} a_{2j}^{\partial}(-n-1-2j;w)\,dv(w), \\[4mm]
\Res_{\Omega,\pO,x}^{{\rm loc}}(-n-1-2j)&=&\displaystyle -\frac1{2j+1}
\int_{S^{n-2}} a_{2j}^{}(-n-1-2j;w)\,dv(w). 
\end{array}
\end{equation}

(2) It is easy to check 
\[
a_0^\partial(-n-1;w)=a_0(-n-1;w)=\sum_i -\frac{\kappa_i}2\,w_i^2.
\]
Now the assertion follows from \eqref{int_S_n-2_w^2}. 

\medskip
(3) First remark that 
the integral on $S^{n-2}$ of a term containing odd powers of $w_i$ vanishes. 
Let $\bar a_2^{\partial}(z;w)$ (or $\bar a_2(z;w)$) be the sum of the terms in $a_2^{\partial}(z;w)$ (or in $a_2(z;w)$ respectively) that consist of terms without odd powers of $w_i$. 

Substitute $z=-n-3$ to \eqref{a_i_partial_zw} and take out the terms with only even powers of $w_i$ from the coefficients of $1,r$, and $r^2$. 
To find $\bar a_2^{\partial}(-n-3;w)$, we do not need to take into account the $r^1$-terms of the second factor of the left hand side of \eqref{a_i_partial_zw}; we have only to consider 
\[
\biggl(1-\frac{n+3}2\Bigl(\sum_i\frac{\kappa_i}2\,w_i^2\Bigr)^2r^2\biggr)
\biggl(
-\sum_j\frac{\kappa_j}2\,w_j^2
 -3\Bigl(
  \sum_i\frac{f_{iiii}}{24}\,w_i^4+\sum_{i<j}\frac{f_{iijj}}{4}\,w_i^2w_j^2
 \Bigr)r^2
\biggr)
+O(r^3) 
\]
to have 
\begin{equation}\label{bar_a_partial_-n-3}
\bar a_2^{\partial}(-n-3;w)=\frac{n+3}2\Bigl(\sum_i\frac{\kappa_i}2\,w_i^2\Bigr)^2\Bigl(\sum_j\frac{\kappa_j}2\,w_j^2\Bigr)
-3\Bigl(
\sum_i\frac{f_{iiii}}{24}\,w_i^4+\sum_{i<j}\frac{f_{iijj}}{4}\,w_i^2w_j^2\Bigr).
\end{equation}

Similarly, substitute $z=-n-3$ to \eqref{a_i_zw} and take out the terms with only even powers of $w_i$ from the coefficients of $1,r$, and $r^2$. 
As $w_iw_j$ $(i\ne j)$ terms of 
\[
\sqrt{1+{|\nabla f(rw)|}^2}=1+2\Bigl(\sum_k\frac{\kappa_k}2\,w_k\Bigr)^2r^2+O(r^3),
\]
are not used in $\bar a_2(-n-3;w)$, we have only to consider 
\[
-\biggl(1-\frac{n+3}2\Bigl(\sum_i\frac{\kappa_i}2\,w_i^2\Bigr)^2r^2\biggr)
\biggl(
 \sum_j\frac{\kappa_j}2\,w_j^2
  +\Bigl(
   \sum_i\frac{f_{iiii}}{24}\,w_i^4+\sum_{i<j}\frac{f_{iijj}}{4}\,w_i^2w_j^2
  \Bigr)r^2
\biggr)
\biggl(1+2\sum_k{\Bigl(\frac{\kappa_k}2\Bigr)}^2\,w_k^2\,r^2\biggr)
+O(r^3)
\]
to have 
\begin{equation}\label{bar_a_-n-3}
\begin{array}{rcl}
\bar a_2(-n-3;w)&=&\displaystyle \frac{n+3}2\Bigl(\sum_i\frac{\kappa_i}2\,w_i^2\Bigr)^2\Bigl(\sum_j\frac{\kappa_j}2\,w_j^2\Bigr)
-2\Bigl(\sum_j\frac{\kappa_j}2\,w_j^2\Bigr)\Bigl(\sum_k \Bigl(\frac{\kappa_k}2\Bigr)^2\,w_k^2\Bigr) \\[4mm]
&&\displaystyle 
-\Bigl(\sum_i\frac{f_{iiii}}{24}\,w_i^4+\sum_{i<j}\frac{f_{iijj}}{4}\,w_i^2w_j^2\Bigr).
\end{array}
\end{equation}

By taking a linear combination of \eqref{bar_a_partial_-n-3} and \eqref{bar_a_-n-3} we can cancel $f_{****}$ terms. Then by 
\eqref{local_partial_relative_residue_-n-1-j} we obtain 
\begin{align}
&\displaystyle \frac12\left(-\Rsub{\Omega,\pO,x}^{\partial\>{\rm loc}}(-n-3)+3\Rsub{\Omega,\pO,x}^{{\rm loc}}(-n-3)\right) \notag \\
&\displaystyle =\frac16\int_{S^{n-2}}\biggl[-(n+3)\Bigl(\sum_i\frac{\kappa_i}2\,w_i^2\Bigr)^3
+6\Bigl(\sum_j\frac{\kappa_j}2\,w_j^2\Bigr)\Bigl(\sum_k{\Bigl(\frac{\kappa_k}2\Bigr)}^2\,w_k^2\Bigr)\biggr]\,dv(w). 
\label{local_relative_residue_-n-3-pf} 
\notag 
\end{align}
Now the assertion follows from \eqref{moment_integral_4_on_S3} and \eqref{moment_integral_6_on_S3}. 

\medskip
(4) 
By \eqref{local_partial_relative_residue_-n-1-j}, \eqref{bar_a_partial_-n-3} and \eqref{bar_a_-n-3}, we have 
\begin{align}
&\displaystyle \Rsub{\Omega,\pO,x}^{\partial\>{\rm loc}}(-n-3)-\Rsub{\Omega,\pO,x}^{{\rm loc}}(-n-3) \notag \\
&\displaystyle =-\frac23\int_{S^{n-2}}\biggl[
\frac18\sum_i\kappa_k^3\,w_k^4+\frac18\sum_{i\ne j}\kappa_i\kappa_j^2\,w_i^2w_j^2
-\sum_i\frac{f_{iiii}}{24}\,w_i^4-\sum_{i<j}\frac{f_{iijj}}{4}\,w_i^2w_j^2
\biggr]\,dv(w). 
\notag
\end{align}
Applying \eqref{moment_integral_4_on_S3} to the right hand side above, 
the assertion follows from \eqref{Delta_H_m=4}. 
\end{proof}

By \eqref{mixed_energy} and \eqref{relative_residue_-n} we have 
\[
\mbox{\rm Vol}(\pO)
=\frac2{o_{n-1}}
\iint_{\pO\times \pO}|x-y|^{-n}\langle y-x, \uon{y}\rangle\,\dx \dy .
\]
Thus, together with \eqref{residue_Omega_-n}, both the volumes of $\dom$ and $\pO$ can be expressed by the double integral on $\pO$.

\subsection{\M non-invariance of relative residues}
When $z$ is not a pole, $B_{\O,\pO}(z)$ is a homogeneous function of $\dom$ with degree $z+2n-1$. 
By the same reason as in the cases of $M$ and $\dom$, when $z$ is a pole the relative residue $\Res_{\O,\pO}$ inherits the homogeneity property, whereas the regularized energy $E_{\O,\pO}(z)$ does not unless the residue vanishes identically.

To see if the relative residues are \M invariant, we have only to study the case when $n$ is even and $z=-2n+1$. 
Notice that when $n=2$, 
\[\Rd{-3}{\O,\pO}=-\int_{\pO}\ka \dx =-2\pi\chi(\dom)\]
is a topological invariant\footnote{The sign of the curvature in the above formula is opposite to that in \eqref{relative_residue_-n-1}. This is because if the normal vector of a plane curve is on the left side of the tangent vector by convention and the boundary of a convex body is oriented counterclockwise, the normal vector becomes inward, which is opposite to our convention (cf. Remark \ref{remark_convention_orientation_normal}).}. 
We show that $\Rd{-2n+1}{\O,\pO}$ is not \M invariant when $n=4$. 

Lemma \ref{inversion_codim>1}, restricted to the codimension $1$ case, implies the following. 
Let $M$ be a hypersurface that does not pass through the origin, $\nu$ be the unit normal vector and $\kappa_i$ be the principal curvatures. 
Let $I$ be an inversion in a unit sphere with center the origin. 
Then the principal curvatures of the image $I(M)$ at a point $I(p)$ are given by 
\begin{equation}\label{lem323}
\tilde \kappa_i=\mp\bigl({|p|}^2\kappa_i+2\langle p,\nu\rangle\bigr),
\end{equation}
where the sign depends on the choice of the unit normal vector $\widetilde {\nu}$ of $I(M)$. 

\medskip
Let $S_a$ be an $a$-hyper-spheroid given by Definition \ref{def_hyper-spheroid}. 
Since 
\[
{|p(a;\theta)|}^2=a^2\cos^2\theta_1+\sin^2\theta_1, 
\quad 
\langle p,\nu\rangle=\frac1{|p(1/a;\theta)|}=\frac{a}{\sqrt{a^2\sin^2\theta_1+\cos^2\theta_1}},
\]
\eqref{principal_curvatures_S_a} and \eqref{lem323} imply that the principal curvatures of the image $\widetilde S_a=I(S_a)$ of an inversion $I$ in a unit sphere with center the origin are given by 
\begin{equation}\label{principal_curvatures_widetilde_S_a}
\tilde \kappa_1=-\frac{(2a^2-1)\sin^2\theta_1+(2-a^2)\cos^2\theta_1}{{\left(a^2\sin^2\theta_1+\cos^2\theta_1\right)}^{3/2}}\,a, \>\>
\tilde \ka_2=\dots=\tilde \ka_m=-\frac{-(a^2-1)\cos^2\theta_1+1}{\sqrt{a^2\sin^2\theta_1+\cos^2\theta_1}}\,a.
\end{equation}
By \eqref{vol_ele_spheroid}, the volume element $d\tilde v$ of $\widetilde S_a$ is given by 
\begin{eqnarray}
d\tilde v
&=&\displaystyle {|p|}^{-2m}dv 
%
=\displaystyle \frac{\sqrt{a^2\sin^2\theta_1+\cos^2\theta_1}}{{\left(a^2\cos^2\theta_1+\sin^2\theta_1\right)}^{m}}\,
\sin^{m-1}\theta_1\sin^{m-2}\theta_2\dots\sin\theta_{m-1} \,d\theta_1\dots d\theta_m.
 \label{vol_element_widetilde_S_a}
\end{eqnarray}

\begin{proposition}
When $n=4$, $\Rd{-7}{\O,\pO}$ is not \M invariant. 
\end{proposition}

\begin{proof}
Let $S^3_r$ denote the $3$-sphere with radius $r$ and center the origin. 
Let $\dom$ be a region bounded by an $a$-hyper-spheroid ($a\ge1$) and $S^3_{1/2}$. 
Then the image $\widetilde \dom$ of an inversion $I$ in the unit sphere $S^3_1$ is a region bounded by $S^3_2$ and $\widetilde S_a$. 

For a hypersurface $M$ that bounds a compact region in $\RR^4$ put 
\[
R(M)=\int_{M}\Big(\sum_i{\ka_i}^3-\sum_{i\ne j}{\ka_i}^2\ka_j-2\ka_1\ka_2\ka_3
\Big)dv,
\]
where the sign of the principal curvatures are taken with respect to the outer unit normal. Then \eqref{relative_residue_-n-3} implies that 
\begin{equation}\label{relative_residue_-7}
\begin{array}{rcl}
\Rd{-7}{\O,\pO}&=&\displaystyle \frac\pi{60}\Bigl(R(S_a)-R\big(S^3_{1/2}\big)\Bigr), \quad 
\Rd{-7}{\widetilde\O,\partial\widetilde\O}=\frac\pi{60}\Bigl(R\big(S^3_{2}\big)-R\big(\widetilde S_a\big)\Bigr)
\end{array}
\notag
\end{equation}
As $R$ is scale invariant  
we have only to show $R(S_a)+R\big(\widetilde S_a\big)\ne2R(S_1^3). $

Since the principal curvatures of $S_a$ and $\widetilde S_a$ depend only on $\theta_1$, putting 
\begin{equation}\label{two_local_relative_residues_-7}
\begin{array}{rcl}
R_a&=&\displaystyle \int_0^\pi\Big(\sum_i{\ka_i}^3-\sum_{i\ne j}{\ka_i}^2\ka_j-2\ka_1\ka_2\ka_3
\Big)\sqrt{a^2\sin^2\theta_1+\cos^2\theta_1}\,\sin^{2}\theta_1\,d\theta_1, \\
\widetilde R_a&=&\displaystyle \int_0^\pi\Big(\sum_i{\tilde\ka_i}^3-\sum_{i\ne j}{\tilde\ka_i}^2\tilde\ka_j-2\tilde\ka_1\tilde\ka_2\tilde\ka_3
\Big)\frac{\sqrt{a^2\sin^2\theta_1+\cos^2\theta_1}}{{\left(a^2\cos^2\theta_1+\sin^2\theta_1\right)}^{3}}\,\sin^{2}\theta_1\,d\theta_1, \\
\end{array}
\notag
\end{equation}
where $\ka_i$ and $\tilde\ka_i$ are given by \eqref{principal_curvatures_S_a} and \eqref{principal_curvatures_widetilde_S_a}, it is enough to show $R_a+\widetilde R_a\ne 2R_1$. 

Direct computation\footnote{The author used Maple in these calculations. } 
shows
\[
\begin{array}{rcl}
R_a&=&\displaystyle \int_0^\pi \frac{a^3\sin^2\theta_1\bigl(4(a^2-1)^2\sin^4\theta_1+10(a^2-1)\sin^2\theta_1+5\bigr)}{{\big(a^2\sin^2\theta_1+\cos^2\theta_1\big)}^4}\,d\theta_1 \\[4mm]
&=&\displaystyle \bigg[\frac{5(7a^4+2a^2-1)}{16a^4}\arctan(a\tan(\theta_1)) \\[4mm]
&&\displaystyle +\frac{(-87a^8-66a^6+33a^4)\tan^5(\theta_1)-(200a^6+80a^4-40a^2)\tan^3(\theta_1)-(105a^4+30a^2-15)\tan(\theta_1)}{48a^3{\big(a^2\tan^2(\theta_1)+1\big)}^3}
\biggr]_0^\pi  \\[4mm]
&=& \displaystyle \frac{5(7a^4+2a^2-1)}{16a^4}\,\pi
\end{array}
\]
\[
\begin{array}{rcl}
\widetilde R_a&=&\displaystyle \int_0^\pi \frac{a^3\sin^2\theta_1\bigl(3(a^2-1)\sin^2\theta_1-a^2+2\bigr)}{{\big(a^2\cos^2\theta_1+\sin^2\theta_1\big)}^3{\big(a^2\sin^2\theta_1+\cos^2\theta_1\big)}^4}
\Bigl(4{(a^2-1)}^4\sin^8\theta_1-2{(a^2-1)}^3(4a^2-15)\sin^6\theta_1
\\[6mm]
&& \displaystyle \qquad  +{(a^2-1)}^2(4a^4-40a^2+65)\sin^4\theta_1
+10(a^2-1)(a^2-2)(a^2-3)\sin^2\theta_1
+5{(a^2-2)}^2
\Bigr)\,d\theta_1 \\[3mm]
&=&\displaystyle \frac{13a^{10}+153a^8+138a^6+18a^4-7a^2+5}{16a^4{(a^2+1)}^3}\,\pi
\end{array}
\]
The antiderivative of the integrand of $\widetilde R_a$ can also be expressed explicitly using $\tan$ and $\arctan$, but we omit it as it is a little bit lengthy. 
Now we have $R_a+\widetilde R_a<2R_1$ if $a>1$. 
\end{proof}

\subsection{Inclusion-exclusion principle}\label{subsect_inclusion_exclusion}

The residues and the relative residues can be obtained by the integral on $\partial\dom$ of functions that are given locally except for $\Res_\dom(-n)={\mbox{\rm Vol}}(\dom)$. 
This fact implies that the residues and relative residues satisfy the {\em inclusion-exclusion principle}. 
Recall we say that a compact body $\dom$ is of class $C^r$ if the boundary $\partial\dom$ is a closed submanifold of $\RR^n$ of class $C^r$. 

\begin{proposition}\label{inclusion-exclusion_axiom}
Let $k\ge1$. 
Let $\dom_1$ and $\dom_2$ be compact bodies in $\RR^n$ such that $\dom_1, \dom_2$, $\dom_1\cap\dom_2$ and $\dom_1\cup\dom_2$ are all of class $C^{k+2}$. 
Suppose $\Res(\dom_i)$ $(i=1,2)$ is one of the residues $\Res_{\dom_i}(z)$, the relative residues $\Res_{\dom_i, \partial\dom_i}(z)$ $(z=-n$ or $-n-1-2j)$, or the residues of the boundary $\Res_{\pO_i}(z)$ $(z=-n-1-2j)$, where $0\le 2j\le k-1$. 
Then 
\begin{equation}\label{i-e_axiom_compact_bodies}
\Res(\dom_1)+\Res(\dom_2)=\Res(\dom_1\cap\dom_2)+\Res(\dom_1\cup\dom_2).
\end{equation}
\end{proposition}

\begin{proof} 
Recall that when $z=-n$ 
we have  
\[
\Res_\dom(-n)=o_{n-1}\textrm{Vol}(\dom), \>\Res_{\dom,\partial\dom}(-n)=\frac{o_{n-1}}2\textrm{Vol}(\partial\dom),
\]
where $o_p$ is the volume of the unit $p$-dimensional unit sphere. 
Since it is known that the right hand sides above satisfy the inclusion-exclusion principle (see Subsection \ref{subsubsection_Lipschitz-Killing_curv_intrinsic_volum}), 
we have only to prove the case when $z=-n-1-2j$, $j\ge0$. 

\smallskip
Let $\dom$ be a compact body and $M=\pO$. 
Let $\nu_x$ be a unit outer normal vector to $\dom$ at $x\in M$. 
We give a proof for the relative residues. 
Put 
\[
\lambda_{{\rm rel}}^\partial(x,y)=\langle y-x, \nu_y\rangle, \> 
\psi_{M,\lambda_{{\rm rel}}^\partial,x}(t)=\int_{M\cap B^n_t(x)}\lambda_{{\rm rel}}^\partial(x,y)\,dv(y). 
\]
Since the meromorphic potential functions at a boundary point $x$ are given by  
\[\begin{array}{rcl}
%
\displaystyle B_{\Omega,\pO,x}^{\partial\>{\rm loc}}(z)&=&\displaystyle \frac1{z+n}\int_{M}|x-y|^z\lambda_{{\rm rel}}^\partial(x,y)\,\dy, 
\end{array}
\]
by Subsection \ref{method_graph}, \eqref{RsubMrhoddtpsix} implies that the local residues are given by 
\[\begin{array}{rcl}
\displaystyle \Res_{\Omega,\pO,x}^{\partial\>{\rm loc}}(-n-(2j+1))&=&\displaystyle \frac{-1}{(2j+1)}\cdot
\left.\frac1{(n+2j)!}\,\frac{d^{\,n+2j+1}}{dt^{\,n+2j+1}} \,\psi_{M,\lambda_{{\rm rel}}^\partial,x}(t)\right|_{t=0} \\[4mm]
&=&\displaystyle -\frac{\psi_{M,\lambda_{{\rm rel}}^\partial,x}^{(n+2j+1)}(0)}{(2j+1)(n+2j)!}\,,
\end{array}
\]
which implies the residues are given by the integral of locally given functions as 
\begin{equation}\label{formula_residues}
\begin{array}{rcl}
R_{\dom,\pO}(-n-(2j+1))&=&\displaystyle -\frac{1}{(2j+1)(n+2j)!}\int_{M}
\psi_{M,\lambda_{{\rm rel}}^\partial,x}^{(n+2j+1)}(0)\,dv(x). 
\end{array}
\end{equation}
\indent
Now let us come back to the situation of the proposition. 
First remark that the regularity of $\dom$ guarantees the well-definedness of the residues. 
This follows from Proposition 2.3 and Corollary 2.4 of \cite{Onew} since the argumet works even if we replace $\lambda_\nu$ by $\lambda_{\rm rel}^\partial$ there. \\
\indent
Put 
\[
M_k=\partial\dom_k \,(k=1,2), \, M_0=M_1\cap M_2, \, M_\cap=\partial(\dom_1\cap\dom_2), \, M_\cup=\partial(\dom_1\cup\dom_2). 
\] 
Then we have $M_0\subset M_\cap, M_\cup$. 
Put 
\[
M_{1,i}=(M_1\cap\dom_2)\setminus M_0, \> M_{1,o}=M_1\setminus (M_1\cap \dom_2), \> 
M_{2,i}=(M_2\cap\dom_1)\setminus M_0, \> M_{2,o}=M_2\setminus (M_2\cap \dom_1).
\]
\begin{figure}[htbp]
\begin{center}
\includegraphics[width=.6\linewidth]{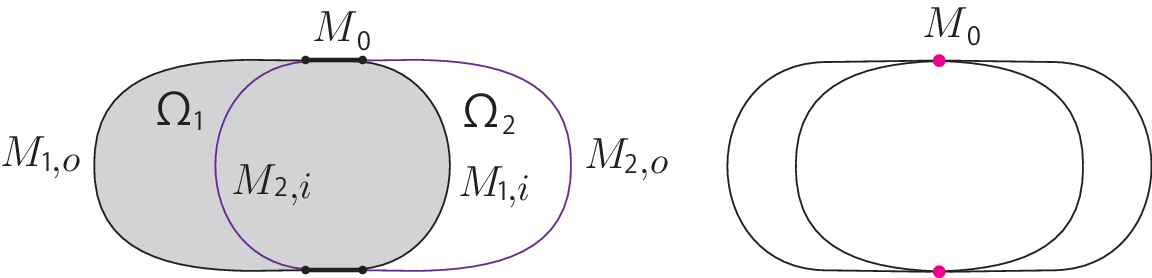}
\caption{}
\label{}
\end{center}
\end{figure}
Then 
\begin{equation}\label{M_decomposition}
M_k=M_0\sqcup M_{k, i} \sqcup M_{k,o} \>\> (k=1,2), \> 
M_\cap=M_0\sqcup M_{1, i}\sqcup M_{2, i}, \>
M_\cup=M_0\sqcup M_{1, o}\sqcup M_{2, o}, 
\notag
\end{equation}
which implies that for sufficiently small $t>0$ 
we have 
\begin{equation}\label{psi_M_iea}
\psi_{M_1,\lambda,x}(t)+\psi_{M_2,\lambda,x}(t)
=\psi_{M_\cap,\lambda,x}(t)+\psi_{M_\cup,\lambda,x}(t)
\end{equation}
for any $x\in M_1\cup M_2$. 
Now \eqref{formula_residues} and \eqref{psi_M_iea} imply that the relative residues satisfy the inclusion-exclusion principle. 

If we substitute $\lambda_{{\rm rel}}^\partial$ in the above argument by $\lambda_1(x,y)\equiv1$ or $\lambda_\nu(x,y)=\langle \nu_x, \nu_y\rangle$ we obtain the proof for $\Res_{\dom_i}$ and $\Res_{\pO_i}$. 
\end{proof}

\section{Residues and similar quantities}\label{sect_similar_q}
We obtain a sequence of geometric quantities of manifolds as residues. 
Let us compare the residues with known quantities which appear to be similar to the residues. 
Most of the quantities studied here are derived from asymptotic expansions.\footnote{It is an open problem whether one can find a weight $\la=\la(x,y)$ such that $\la$-weighted residues can produce the quantities studied in this section. } 
To begin with, we generalize residues to Riemannian manifolds.

\subsection{Energies and residues of Riemannian manifolds}
%
%
Let $M=(M,g)$ be an $m$-dimensional closed connected Riemannian manifold and $d(x,y)$ be 
the geodesic distance between $x$ and $y$. 
Let $\psi_{x}(t)$ be the volume of the geodesic ball $B_t(x)$ with center $x$ and radius $t$. 
Gray (\cite{G}) gave a series expansion of $\psi_{x}(t)$ in $t$ in terms of curvature tensors. 
\begin{equation}\label{Gray_volume_expansion_intrinsic}
\begin{array}{l}
\displaystyle \psi_{x}^{\,\prime}(t) 
=\displaystyle o_{m-1}t^{m-1}\left[1-\frac{\Sc}{6m}t^2 
%
+\frac1{360m(m+2)}
\left(-3{|\Rm|}^2+8{|\mbox{\rm Ric}|}^2+5{\Sc}^2-18\Delta \Sc
\right)t^4+O(t^6)
\right]. \\
\end{array} \qquad{}
\end{equation}
The next coefficient was given in \cite{GV} Theorem 3.3. 

By the argument in Subsection \ref{subsubsect_extrinsic_balls}, 
the first three residues are given by 
\begin{eqnarray}
\Res_M(-m)&=&\displaystyle {o_{m-1}}\, \mbox{Vol}(M), \label{intrinsic_residue_-m} \\
\Res_M(-m-2)&=&\displaystyle -\frac{o_{m-1}}{6m}\int_M \Sc \,dv, \label{intrinsic_residue_-m-2} \\
\Res_M(-m-4)&=&\displaystyle \frac{o_{m-1}}{360m(m+2)}\int_M \left(-3{|\Rm|}^2+8{|\mbox{\rm Ric}|}^2+5{\Sc}^2
\right) dv. \label{intrinsic_residue_-m-4}
\end{eqnarray}
We remark that connectedness assumption is needed to define energy, but not for residues. 

Let us see if the residues are conformally invariant. 
We have only to study $\Res_M(-2m)$ when $m$ is even since otherwise the residues are not even scale invariant.  
Notice that when $m=2$  
\[
\Res_M(-4)=-\frac{2\pi}{12}\int_M \Sc \,dv=-\frac{2\pi^2}3\chi(M)
\]
is a topological invariant. 

\begin{proposition}\label{prop_conf_inv_RMminus8}
When $m=4$, $\Res_M(-8)$ is not a conformal invariant. 
\end{proposition}

\begin{proof}
Let $C_{ijkl}$ be the Weyl tensor. Put 
\begin{equation}\label{Weyl_tensor}
{|W|}^2=\sum_{i,j,k,l}C_{ijkl}^{\,2}={|\Rm|}^2-2{|\mbox{\rm Ric}|}^2+\frac13{\Sc}^2, 
\end{equation}
then ${|W|}^2dv$ is a pointwise conformal invariant. 
Put 
\begin{eqnarray}
X&=&\displaystyle \frac14\left({|\Rm|}^2-4{|\mbox{\rm Ric}|}^2+{\Sc}^2 \right), \label{def_X} 
\end{eqnarray}
then by the Gauss-Bonnet-Chern theorem there holds 
\[
\int_MXdv=8\pi^2\chi(M), 
\]
which implies that $X$ is an integral topological invariant. 

Since the integrand of \eqref{intrinsic_residue_-m-4} satisfies 
\[
-3{|\Rm|}^2+8{|\mbox{\rm Ric}|}^2+5{\Sc}^2
=-2{|W|}^2-4X+\frac{20}3{\Sc}^2,
\]
and ${\Sc}^2$ is not an integral conformal invariant (this fact can also be shown by Theorem \ref{thm_PCE} when $M$ is a hypersurface in $\RR^5$), 
$\Res_M(-8)$ is not a conformal invariant. 
\end{proof}

\subsection{Lipschitz-Killing curvatures (intrinsic volumes)}\label{subsubsection_Lipschitz-Killing_curv_intrinsic_volum}
%
Recall that the Lipschitz-Killing curvatures are given by \eqref{C_k_Omega}, and when $\dom$ is convex they are equal to the intrinsic volumes given by \eqref{Steiner_formula}. 

\begin{proposition}
Let $\dom$ be a compact body in $\RR^n$. 
The first four Lipschitz-Killing curvatures 
can be expressed in terms of the residues, relative residues and the residues of the boundary as follows. 
\begin{eqnarray}
\displaystyle C_n(\dom)&=&
\displaystyle \frac1{o_{n-1}}\Res_\dom(-n),  \label{LKvsR1}\\
\displaystyle C_{n-1}(\dom)&=&
\displaystyle -\frac{n-1}{2o_{n-2}}\Res_\dom(-n-1)=\frac1{2o_{n-2}}\Res_{\pO}(-n+1)
=\frac1{o_{n-1}}\Rd{-n}{\O,\pO}, \qquad {} \label{LKvsR2} \\
%
\displaystyle C_{n-2}(\dom)&=&
\displaystyle -\frac{n-1}{\pi o_{n-2}}\Rd{-n-1}{\O,\pO}, \label{LKvsR3} \\
\displaystyle C_{n-3}(\dom)&=&
\displaystyle \frac{3(n-1)}{4\pi o_{n-2}}\big((n+1)\Rd{-n-3}{\O}-\Rd{-n-1}{\pO}\big). \label{LKvsR4}
\end{eqnarray}
\end{proposition}
\begin{proof}
Recall that \eqref{C_k_Omega} implies that $C_n(\dom), C_{n-1}(\dom), C_{n-2}(\dom)$, and $C_{n-3}(\dom)$ are given by 
\[
{\rm Vol}(\dom), \quad
\frac12{\rm Vol}(\pO), \quad
-\frac1{2\pi}\int_{\pO}H\,dv, \quad 
\frac1{4\pi}\int_{\pO}\sum_{i<j}\ka_i\ka_j\,dv
\]
respectively. 
\eqref{LKvsR1} follows from \eqref{residue_Omega_-n}, 
\eqref{LKvsR2} from \eqref{residue_Omega_-n-1}, Propositions \ref{prop_residues_closed_12} (1), and \eqref{relative_residue_-n}, 
\eqref{LKvsR3} from \eqref{relative_residue_-n-1}, 
and \eqref{LKvsR4} from \eqref{residue_closed_-m-2} and \eqref{residueOmega-n-3_BH}. 
We note that \eqref{LKvsR4} also follows from \eqref{BR_nu_BR_dom} and \eqref{scalar_curv_residue}. 
\end{proof}
\begin{corollary}\label{cor_LK-curv-residues}
If $n=2$ or $3$, Lipschitz-Killing curvatures (or equivalently, intrinsic volumes when $\dom$ is convex) can be obtained as linear combinations of the residues, relative residues and the residues of the boundary. 
\end{corollary}

The third elementary symmetric function of the principal curvatures $S_3(x)$ cannot be expressed by a linear combination of the local residues of $\dom, \pO$ and the local relative residues. 
Therefore we conjecture 
that Corollary \ref{cor_LK-curv-residues} does not hold if $n\ge4$. 

\medskip
The intrinsic volumes are of basic importance in convex geometry because of the following characterization theorem by Hadwiger, which can be considered as a starting point of the valuation theory in integral geometry (cf. \cite{AF}).

Let $\mathcal{K}^n$ be the set of convex bodies in $\RR^n$. 
A {\it valuation} is a function $\mu\colon\mathcal{K}^n\to\RR$ that satisfies $\mu(\emptyset)=0$ and the inclusion-exclusion principle 
\[
\mu(K)+\mu(L)=\mu(K\cap L)+\mu(K\cup L)
\]
for any $K,L\in\mathcal{K}^n$ as far as $K\cup L\in\mathcal{K}^n$. 
\begin{theorem}\label{Hadwiger_thm} {\rm (Hadwiger, cf. \cite{SW} page 628)}
Suppose $\mu$ is a valuation which is continuous with respect to the Hausdorff metric and is invariant under rigid motions. 
Then there are constants $c_0,\dots,c_n$ such that 
\[
\mu(K)=\sum_{j=0}^nc_jV_j(K)
\]
for all $K\in\mathcal{K}^n$, where $V_j$ are the intrinsic volumes. 
\end{theorem}

Proposition \ref{inclusion-exclusion_axiom} implies that, except for some residues which are essentially equal to intrinsic volumes, the residues of convex bodies are not continuous with respect to the Hausdorff metric.

\subsection{Coefficients of Weyl's tube formula}\label{subsec_Weyl}
%
Let $M$ be an $m$ dimensional closed oriented submanifold of $\RR^n$. 
Weyl (\cite{W}) gave a formula for the volume of a tube of a small radius $r$ around $M$: 
\[
{\rm Vol}(M_r)=\omega_{n-m}r^{n-m}\sum_{i=0}^{\lceil m/2 \rceil}\frac{\mathfrak{k}_{2i}(M)\,r^{2i}}{(n-m+2)(n-m+4)\dots(n-m+2i)}, 
\]
where $\omega_{n-m}$ is the volume of the unit $(n-m)$-ball. 

We remark that when $M$ is a boundary of a compact body $\dom$, $\overline{\dom_r\setminus\dom}$ can be considered as a one-sided tube or a half-tube, and hence 
$
{\rm Vol}((\pO)_r)={\rm Vol}(\dom_r)-{\rm Vol}(\dom_{-r}),
$
where $\dom_{-r}$ is an inner parallel body and ${\rm Vol}(\dom_{-r})$ is given by substituting $r$ by $-r$ in \eqref{Steiner_principal_curv} when $|r|$ is sufficiently small. 

The first coefficient $\mathfrak{k}_0$ is just the volume ${\rm Vol}(M)$, which is essentially equal to the residue $\Res_M(-m)$. 
The second coefficient 
\[
\mathfrak{k}_2=\frac12\int_M\Sc\,dv
\]
can be expressed in terms of the residues as 
\begin{equation}\label{Weyl_tube_formula_second_coeff_residues}
\mathfrak{k}_2=-\frac{m}{o_{m-1}}\bigl(\Res_{M,\nu}(-m-2)+3\Res_M(-m-2)\bigr)  \notag
\end{equation}
by Theorem \ref{thm_Sc_residue}.

\subsection{Spectrum of the Laplacian}\label{subsect_spectrum}
%
Let $M=(M,g)$ be an $m$ dimensional closed Riemannian manifold, $\Delta$ be the Laplace-Beltrami operator, and 
$
{\rm Spec}(M,g)=\{0=\la_0<\la_1\le\la_2 \dots\}
$ 
be the spectrum of $M$, where $\la_i$ are the eigenvalues of $\Delta$. 
Consider the asymptotics of the trace of the heat kernel 
\[
\sum_{i=0}^\infty e^{-\la_i t}\sim\frac1{{(4\pi t)}^{m/2}}\sum_{i=0}^{\lceil m/2 \rceil}
a_it^i+O\Bigl(t^{1/2}\Bigr) \qquad (t\downarrow 0). 
\]
The first three coefficients $a_i$ were given by 
\begin{eqnarray}
a_0&=&\mbox{Vol}(M), \label{heatkernel0} \notag \\
a_1&=&\displaystyle \frac16\int_M\Sc\,dv, \label{heatkernel1} \notag \\
a_2&=&\displaystyle \frac1{360}\int_M\left(
2{|\Rm|}^2-2{|\mbox{\rm Ric}|}^2+5{\Sc}^2
\right)dv \label{heatkernel2} \notag
\end{eqnarray}
(cf. \cite{BGM}).

By \eqref{intrinsic_residue_-m} and \eqref{intrinsic_residue_-m-2}, $a_0$ equals $\Res_M(-m)$, and $a_1$ equals $\Res_M(-m-2)$ up to multiplication by constants. 
On the other hand, comparing $a_2$ and \eqref{intrinsic_residue_-m-4}, 
since the integration of ${|\Rm|}^2$, ${|\mbox{\rm Ric}|}^2$, and ${\Sc}^2$ are independent (\cite{G}), 
$a_2$ and $\Res_M(-m-4)$ are independent; they are sum of the same monomials with different coefficients\footnote{The same phenomenon has been observed in the relation between the residues of Brylinski's beta function of knots and the invariants of the vortex filament equation as far as are calculated (\cite{Fenici})}.

\subsection{Graham-Witten's conformal invariant for $4$-dimensional manifolds}\label{subsec_Graham-Witten}

\subsubsection{Definition}
%
%
In 1999 Graham and Witten introduced conformal invariants of even dimensional compact submanifolds $M$ via volume renormalization. 
This invariant, called the Graham-Witten energy, appears as the coefficient of the log term in the series expansion of the renormalized volume of a minimal submanifold in a Poincar\'e-Einstein space which is asymptotic to $M$ as the minimal submanifold approaches the conformal infinity. 
To be precise, let $X$ be an $(n+1)$ dimensional manifold with a smooth bounary $W$, $Y$ be an $(m+1)$ dimensional ``asymptotically minimal'' submanifold in $X$ that satisfies $Y\cap W=M^m$, and $r$ be the defining function of the boundary $W$ so that the Poincar\'e-Einstein metric near the boundary can be expressed by 
\[
g_+=\frac{dr^2+g_r}{r^2}, 
\]
where $g_r$ is a $1$-parameter family of metrics on $W$. 
Then the volume of $Y\cap\{r>\varepsilon\}$ with respect to $g_+$ has the following asymptotic expansion as $\e$ tends to $+0$,
\begin{equation}\label{eq_vol_renormalization_GW}
\mbox{\rm Vol}\big(Y\cap\{r>\varepsilon\}\big)=
\sum_{l=0}^{\lceil m/2 \rceil-1}b_l\e^{-m+2l} 
+\left\{
\begin{array}{ll}
A^o+o(1) & \hspace{0.4cm}\mbox{$m$ : odd}, \\[1mm]
K\log(1/\e)+A^e+o(1) & \hspace{0.4cm}\mbox{$m$ : even}.
\end{array}
\right.
\end{equation}
The constant term $A^o$ or $A^e$ is called the {\em renormalized volume}. 
It was proved in Proposition 2.1 of \cite{GW} that $A^o$ and $K$ are independent of the choice of the representative metric $g$ in the conformal class. 
This $K$ is called the {\em Graham-Witten energy} of $M$, denoted by $\mathcal{E}(M)$. 

Suppose $M$ is a closed submanifold of $\RR^n$. 
When $M$ is a $2$-dimensional surface in $\RR^3$ the Graham-Witten energy coincides with the Willmore energy. 
In the rest of this Subsection we assume $m=4$. 
Then $\mathcal{E}(M)$ is given by 
\begin{equation}\label{GW}
\mathcal{E}(M)=\frac1{128}\int_M\Bigl(
\big|\nabla^\perp H\big|^2-g^{ik}g^{jl}\langle h_{ij},H\rangle\langle h_{kl},H\rangle +\frac7{16}\,\vert H\vert^4
\Bigr)dv 
\end{equation}
(Graham and Reichert \cite{GR}, Zhang \cite{Zh}).\footnote{The coefficients are different from those in Guven \cite{Gu}. 
Guven's energy is not \M invariant. }

%
\subsubsection{Independence}
Consider Graham-Witten's energy $\mathcal{E}(M)$, the residue $\Rd{-8}{M}$ and $\nu$-weighted residue $\Rd{-8}{M,\nu}$ at $z=-8$ for $4$-dimensional closed submanifolds of $\RR^n$. 
They are all scale invariant, the first two are both \M invariant, and so is the third if $M$ is a boundary of a compact body. All are in a sense generalization of the Willmore energy of surfaces. 
We show that the three quantities are independent by computing their values of a family of $a$-hyper-spheroids (Definition \ref{def_hyper-spheroid}). 

Note that when $M$ is a hypersurface the Graham-Witten energy is given by 
\begin{eqnarray}
\mathcal{E}(M)&=&\displaystyle \frac1{128}\int_{M}\Bigl(
{|\nabla H|}^2-{\Vert h\Vert}^2H^2 +\frac7{16}\,H^4
\Bigr)dv \notag \\
&=&\displaystyle \frac1{128}\int_{M}\biggl({|\nabla H|}^2
-\Bigl(\sum_j\ka_j^{\,2}\Bigr){\Bigl(\sum_l\ka_l\Bigr)}^2
+\frac7{16}\,{\Bigl(\sum_i\ka_i\Bigr)}^4
\biggr)dv. \label{GW-hypersurface}
\end{eqnarray}

\begin{lemma}\label{a-hyper-spheroid-GW}
Let $S_a$ be a $4$-dimensional $a$-hyper-spheroid. 
Then Graham-Witten's energy $\mathcal{E}(S_a)$ is given by 
\begin{equation}\label{formulae_GW-spheroid}
\mathcal{E}(S_a)=\displaystyle \frac{\pi^2}{17920}\cdot \frac{10613a^8-7778a^6-2376a^4-16a^2-128+4725a^8\left(a^2-\displaystyle \frac{16}{15}\right)\omega(a)}{a^6(a^2-1)\phantom{\tilde l}\!\!} 
\end{equation}
when $a\ne1$, where $\omega(a)$ is given by \eqref{def_omega}, and $\mathcal{E}(S_1)=\pi^2$ when $a=1$. 
\end{lemma}

\begin{proof} 
The proof is parallel to that of Lemma \ref{a-hyper-spheroid-residues}. 
Substituting \eqref{principal_curvatures_S_a} and \eqref{coeff_f_ijk} to \eqref{grad_H_sq_m=4} and \eqref{GW-hypersurface}, we obtain  
\begin{equation}\label{f_GW_S_a}
\begin{array}{rcl}
\mathcal{E}(S_a)&=&\displaystyle \frac{2\pi^2}{128} \int_0^\pi 
\frac{a^2}{{\big(a^2\sin^2\theta_1+\cos^2\theta_1\big)}^6} 
\biggl[
9(a^2-1)^2\sin^2\theta_1\cos^2\theta_1{\big(a^2\sin^2\theta_1+\cos^2\theta_1+1\big)}^2 \\[4mm]
&&\displaystyle \quad +\frac{135}{16}a^2 \big(a^2\sin^2\theta_1+\cos^2\theta_1+3\big)
\biggl(a^2\sin^2\theta_1+\cos^2\theta_1+\frac13\biggr)^2
\biggl(a^2\sin^2\theta_1+\cos^2\theta_1-\frac15\biggr)
\biggr]\cdot \\[4mm]
&&\displaystyle 
\quad 
\sqrt{a^2\sin^2\theta_1+\cos^2\theta_1}\,\sin^3\theta_1\,d\theta_1.
\end{array}
\end{equation}
The conclusion follows from direct computation of \eqref{f_GW_S_a} using Maple. 
We remark that the case when $a=1$, i.e., $\mathcal{E}(S_1)=\pi^2$ was given in \cite{Zh}, page 21. 
\end{proof}

Put $\mathcal{E}(a)=\mathcal{E}(S_a)$, $\Res_\nu(a)=\Rd{-8}{S_a,\nu}$ and $\Res(a)=\Rd{-8}{S_a}$ (cf. Lemma \ref{a-hyper-spheroid-residues}) in what follows. 
$\mathcal{E}(a), \,\Res_\nu(a)$ and $\Res(a)$ are all continuous with respect to $a$. 

\smallskip
We put a graph created by Maple. The horizontal axis stands for $a$. 
Graham and Reichert pointed out (\cite{GR} page 3) that a numerical calculation using Mathematica suggests that $\mathcal{E}(S_a)$ is a convex function of $a$ which goes to $+\infty$ as $a$ approaches $+0$ and $+\infty$, and that $\mathcal{E}(S_a)$ has unique minimum at $a=1$ which corresponds to a round sphere. 
Our calculation using Maple also suggests that the same statement holds for $\Res_{S_a, \nu}(-8)$ and $2\Res_{S_a, \nu}(-8)+\Res_{S_a}(-8)$ (cf. \eqref{GW-residues}). 

\begin{figure}[htbp]
\begin{center}
\begin{minipage}{.3\linewidth}
\begin{center}
\includegraphics[width=\linewidth]{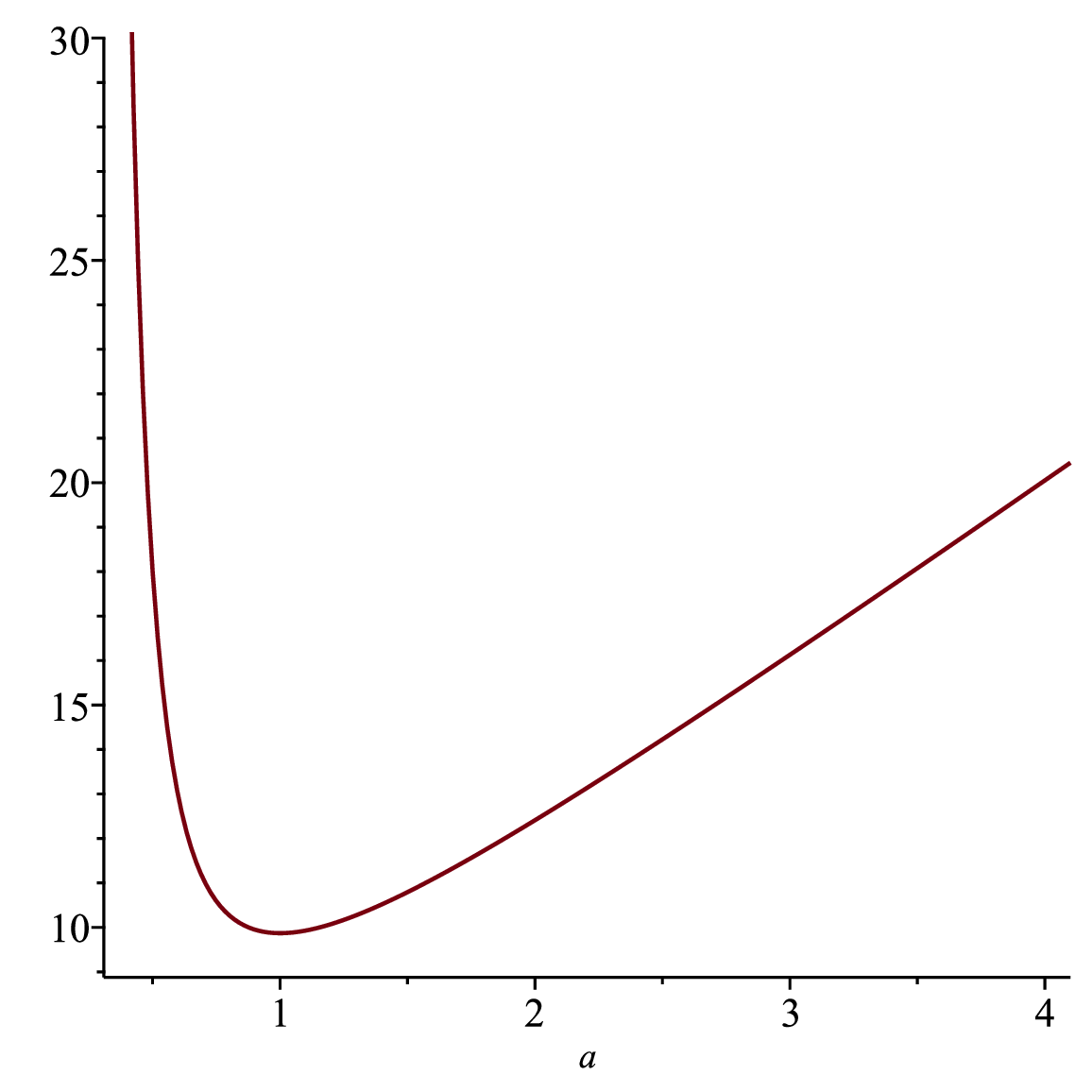}
\vspace{0.5cm}
\caption{Graham-Witten energy $\mathcal{E}(S_a)$}
\label{Q}
\end{center}
\end{minipage}
\hskip 0.1cm
\begin{minipage}{.3\linewidth}
\begin{center}
\includegraphics[width=\linewidth]{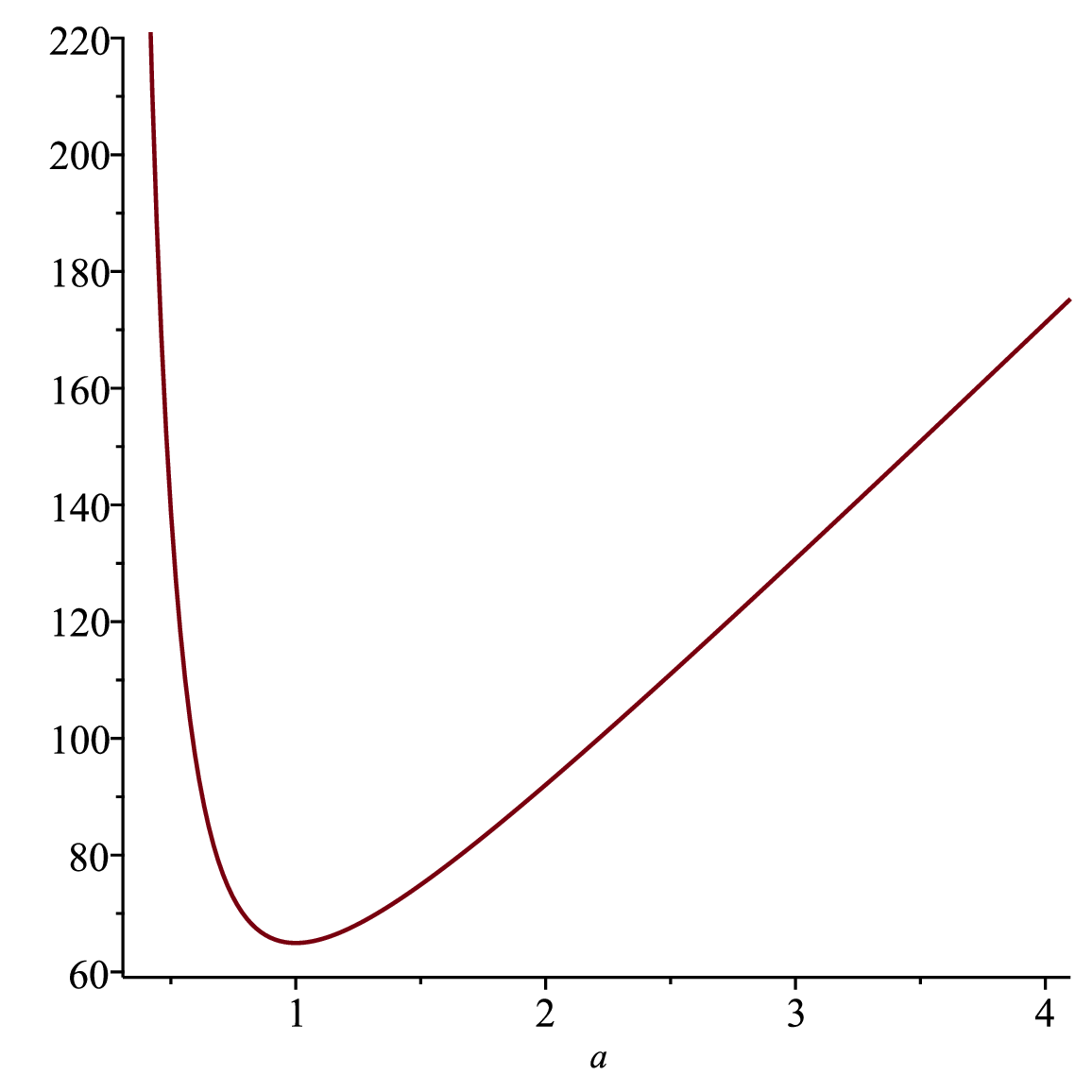}
\vspace{0.5cm}
\caption{$\nu$-weighted residue $\Rd{-8}{S_a,\nu}$}
\label{ROmega_m8}
\end{center}
\end{minipage}
\hskip 0.1cm
\begin{minipage}{.3\linewidth}
\begin{center}
\includegraphics[width=\linewidth]{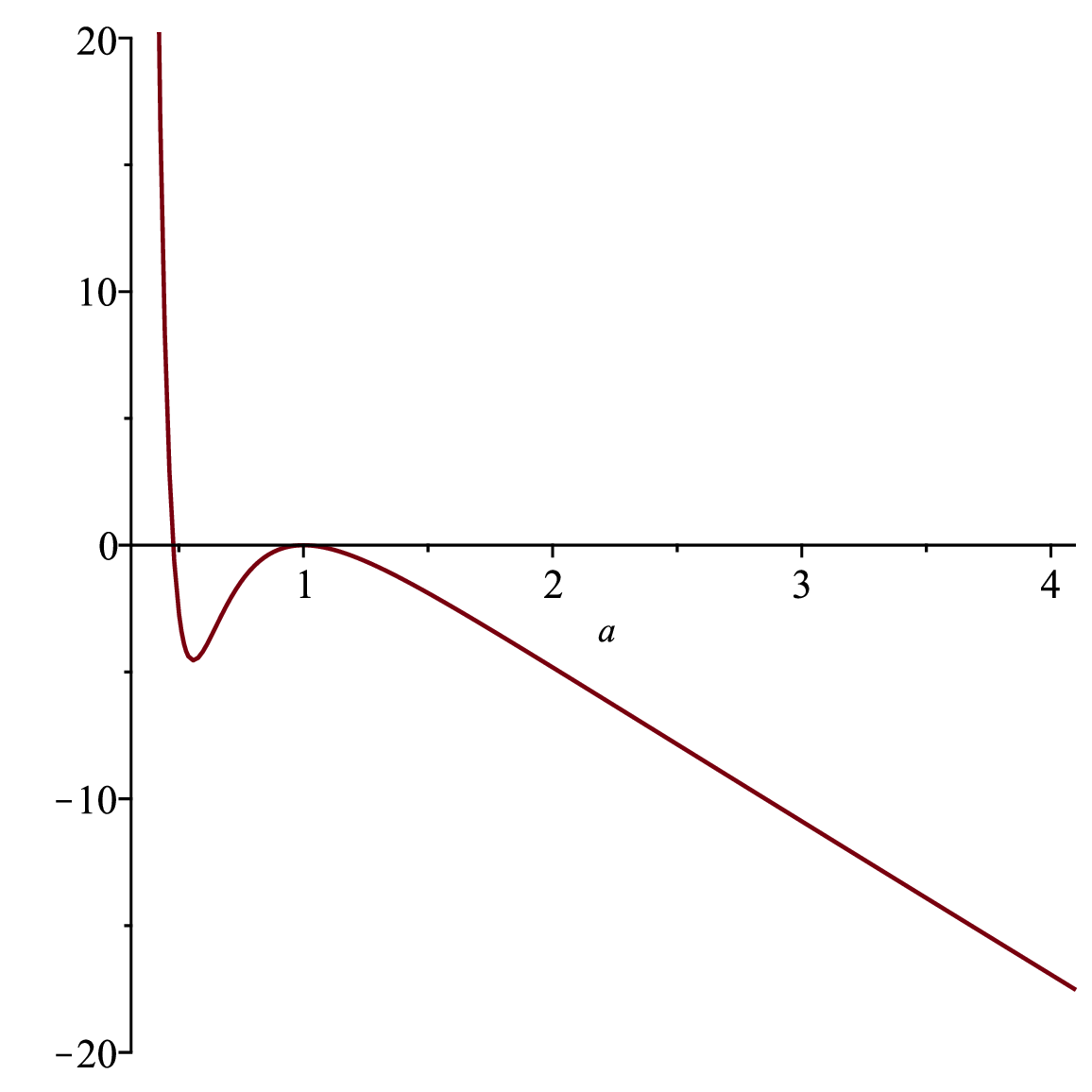}
\vspace{0.5cm}
\caption{Residue $\Rd{-8}{S_a}$}
\label{RM_m8}
\end{center}
\end{minipage}
\end{center}
\end{figure}

\begin{lemma}\label{lemma_GW_res}
\begin{enumerate}
\item There are constants $a^{{\rm{gw}}}_{l}$ and $a^{{\rm{gw}}}_{r}$ $(0<a^{{\rm{gw}}}_{l}\le 1\le a^{{\rm{gw}}}_{r})$ such that $\mathcal{E}(a)$ is positive and strictly decreasing on $(0,a^{{\rm{gw}}}_{l})$ and positive and strictly increasing on $(a^{{\rm{gw}}}_{r}, +\infty)$. 
$\mathcal{E}(a)$ tends to $+\infty$ as $a$ approaches $+0$ or $+\infty$. 
\item Similar statements hold for $\Res_\nu(a)$. Namely, 
there are constants $a^\nu_{l}$ and $a^\nu_{r}$ $(0<a^\nu_{l}\le 1\le a^\nu_{r})$ such that $\Res_\nu(a)$ is positive and strictly decreasing on $(0,a^\nu_{l})$ and positive and strictly increasing on $(a^\nu_{r}, +\infty)$. 
$\Res_\nu(a)$ tends to $+\infty$ as $a$ approaches $+0$ or $+\infty$. 
\item There are constants $a_{l}$ and $a_{r}$ $(0<a_{l}\le 1\le a_{r})$ such that $\Res(a)$ is positive and strictly decreasing on $(0,a_{l})$ and negative and strictly decreasing on $(a_{r}, +\infty)$. 
$\Res(a)$ tends to $+\infty$ as $a$ approaches $+0$ and to $-\infty$ as $a$ tends to $+\infty$. 
\end{enumerate}
\end{lemma}

The numerical experiments by Maple indicate that we can take $a^{{\rm{gw}}}_{l}=a^{{\rm{gw}}}_{r}=a^\nu_{l}=a^\nu_{r}=a_{r}=1$ and $a_{l}\approx 0.4771$. 

\begin{proof}
By the formulae \eqref{formulae_a-spheroid} and \eqref{formulae_GW-spheroid}, the following asymptotic estimates 
\begin{equation}\label{asymptotics}
\begin{array}{rlcrlc}
\mathcal{E}(a)&\sim \displaystyle \frac{\pi^2}{140\,a^6} & \hspace{0.5cm} (a\to+0), & \hspace{0.7cm} \mathcal{E}(a)&\sim \displaystyle \frac{135\,\pi^3}{1024}\,a & \hspace{0.5cm} (a\to+\infty), \\[4mm]
\Res_\nu(a)&\sim \displaystyle \frac{\pi^4}{24a^4} & \hspace{0.5cm} (a\to+0), & \hspace{0.7cm} \Res_\nu(a)&\sim \displaystyle \frac{35\,\pi^5}{256}\,a & \hspace{0.5cm} (a\to+\infty), \\[4mm]
\Res(a)&\sim \displaystyle \frac{2\pi^4}{315a^6} & \hspace{0.5cm} (a\to+0), & \hspace{0.7cm} \Res(a)&\sim \displaystyle -\frac{5\,\pi^5}{256}\,a  &\hspace{0.5cm} (a\to+\infty)
\end{array}
\end{equation}
hold, which imply the statements (1), (2) and (3). 
\end{proof}

\begin{corollary}\label{indep_GW_residues}
As functionals for closed $4$-dimensional submanifolds of $\RR^n$, 
Graham-Witten's energy $\mathcal{E}$, the residue $\Rd{-8}{M}$ and $\nu$-weighted residue $\Rd{-8}{M,\nu}$ at $z=-8$ are independent. 

To be precise, there are closed submanifolds of $\RR^5$, $M_1$ and $M_2$, that satisfy $\mathcal{E}(M_1)=\mathcal{E}(M_2)$ and $\Rd{-8}{M_1}\ne \Rd{-8}{M_2}$, and $M_1'$ and $M_2'$ that satisfy $\Rd{-8}{M_1'}= \Rd{-8}{M_2'}$ and $\mathcal{E}(M_1')\ne\mathcal{E}(M_2')$. 
Therefore $\Rd{-8}{M}$ cannot be expressed by $\mathcal{E}(M)$ and vice versa. 
\\ \indent
The same statements hold for the relations between $\Rd{-8}{M}$ and $\Rd{-8}{M,\nu}$ and between $\mathcal{E}(M)$ and $\Rd{-8}{M,\nu}$. 
\end{corollary}

\begin{proof}
Put $\widetilde a_l=\min\{a^{{\rm{gw}}}_{l}, a^\nu_{l}, a_{l}\}$ and  $\widetilde a_r=\max\{a^{{\rm{gw}}}_{r}, a^\nu_{r}, a_{r}\}$, where $a^{{\rm{gw}}}_{l}, \dots, a_{r}$ are constants given in the previous lemma. 
Let $\mathcal{E}(a)=\mathcal{E}(S_a)$, $\Res_\nu(a)=\Rd{-8}{S_a,\nu}$ and $\Res(a)=\Rd{-8}{S_a}$ for $a>0$ as before. 
\\ \smallskip \indent
(i) By Lemma \ref{lemma_GW_res} (1), there are constants $a_1$ and $a_2$ with $0<a_1<\widetilde a_l$ and $\widetilde a_r<a_2$ such that $\mathcal{E}(a_1)=\mathcal{E}(a_2)$. Lemma \ref{lemma_GW_res} (3) implies $\Res(a_1)\ne\Res(a_2)$. 
\\ \indent
On the other hand, if we put $M_1=S^4$ and $M_2=S^4 \sqcup S^4$ then \eqref{unit_4-sphere} 
implies that $\Rd{-8}{M_1}=\Rd{-8}{M_2}=0$ whereas $\mathcal{E}(M_1)=\pi^2\ne 2\pi^2=\mathcal{E}(M_2)$. 
\\ \smallskip \indent
(ii) The same argument holds for $\Rd{-8}{M,\nu}$ and $\Rd{-8}{M}$. 
\\ \smallskip \indent
(iii) The asymptotic estimates \eqref{asymptotics} imply that there are constants $a_1'$ and $a_2'$ with $0<a_1'<\widetilde a_l$ and $\widetilde a_r<a_2'$ such that $\mathcal{E}(a_1')=\mathcal{E}(a_2')$ and that $\Res_\nu(a_1')\ne\Res_\nu(a_2')$. 
The rest can be shown in the same way. 
\end{proof}

\subsubsection{Approximation by residues}
%
We give an approximation of the Graham-Witten energy by a linear combination of the residue and $\nu$-weighted residue that is optimal from a viewpoint of necessary order of derivatives to express the integrands of the difference (cf. Subsection \ref{subsub_order_diff}). 

\begin{proposition}
Let $f$ be a function from $T_xM$ to $(T_xM)^\perp$ so that a neighbourhood of $M$ of $x$ are given by a graph of $f$. 
Then 
\begin{equation}\label{GW-residues}
\mathcal{E}(M)-\frac{3}{2\pi^2}\big(\Res_{M,\nu}(-8)+2\Res_M(-8)\big)
\end{equation}
can be expressed by the integral over $M$ of a function such that only the derivatives of $f$ of order up to $2$ are involved. 
\end{proposition}

\begin{proof}
Assume $x=0$, $T_xM=\RR^4$, and $f\colon\RR^4\to\RR^{n-4}$ as before. 

Let us first calculate the Graham-Witten energy. 
We use the same notation as in (2) of the proof of Lemma \ref{lem_la_A_f_Delta_Sc_H}. We use indices $i,j,k,\ell,p,\dots\in\{1,2,3,4\}$ for $M$ and $\sigma, \tau, \upsilon,\dots\in\{1,\dots,n\}$ for $\RR^n$.
Then 
\[
g_{ij}=\delta_{ij}+\ip{f_i}{f_j} \quad 1\le i,j\le 4, \quad
g_{\sigma\tau}=\delta_{\sigma\tau} \quad 4<\sigma, \tau, \quad
g_{i\sigma}=\ip{f_i}{\vect e_\sigma}=\partial_if^\sigma \quad i\le 4<\sigma,
\]
\[
\Ga_{ij}^k\eao 0 \quad i,j,k\le 4, \quad
\Ga_{\sigma\tau}^\upsilon=0 \quad \sigma,\tau,\upsilon>4, \quad
\Ga_{ij}^\sigma\eao \partial_i\partial_j f^\sigma \quad i,j\le 4<\sigma, \quad
\Ga_{j\tau}^\sigma\eao 0 \quad j\le 4<\sigma,\tau. 
\]
The second fundamental form is given by 
\[
h(u,v)=\sum_{\sigma>4}\Ga_{ij}^\sigma u_iv_j\partial_\sigma 
\eao u_iv_j\partial_i\partial_j f \quad u,v\in T_xM\cong\RR^4, 
\]
and hence $h_{ij}\eao f_{ij}$, 
which, together with \eqref{mean_curv_vect_H}, implies that the second term of the integrand of the  Graham-Witten energy \eqref{GW} is given by 
\begin{equation}\label{GW_second_term}
g^{ik}g^{jl}\langle h_{ij},H\rangle\langle h_{kl},H\rangle
\eao \sum_{i,j}{\ip{h_{ij}}{H}}^2
=\sum_{i,j}{\big\langle{f_{ij}},{\sum_k f_{kk}}\big\rangle}^2.
\notag
\end{equation}

Since $\partial_jg^{k\ell}\eao 0$, \eqref{mean_curv_vect_H} implies \[
\partial_jH \eao \sum_k\partial_j\partial_k^{\,2}f,
\]
hence
\[
\nabla_{e_j}^\perp H=\sum_{\sigma>4}\left(\partial_jH^\sigma+\Ga_{ij}^\sigma H^i
\right)\partial_\sigma
\eao \sum_k\partial_j\partial_k^{\,2}f \qquad (\partial_\sigma\in T_x\RR^{n-4}),
\]
therefore the first term of the integrand of the Graham-Witten energy \eqref{GW} is given by 
\begin{equation}\label{GW_first_term}
{\big|\nabla^\perp H\big|}^2
\eao \sum_{j,k,\ell}\ip{f_{jkk}}{f_{j\ell\ell}}.
\end{equation}

Since the third term $|H|^4$ of \eqref{GW} can be expressed in terms of derivatives of $f$ up to order $2$, the highest order of derivatives involved in the integrand of \eqref{GW} is three, which appears in \eqref{GW_first_term}. 
Therefore the $f_{\ast\ast\ast}$-terms of the integrand of the Graham-Witten energy \eqref{GW} are given by 
\begin{equation}\label{f***-term_GW}
\frac1{128}\biggl(
\sum_i\ip{f_{iii}}{f_{iii}}+ 2\sum_{j\ne k}\ip{f_{jjk}}{f_{kkk}}
+\sum_{j\ne k}\ip{f_{jkk}}{f_{jkk}}+ 2\sum_{k<\ell, j\ne k,\ell}
\ip{f_{jkk}}{f_{j\ell\ell}}
\biggr).
\end{equation}

\smallskip
Next we consider the residues. 
Theorem \ref{prop_residue_H_Sc_graph} implies that 
\begin{equation}\label{Res_Res_Delta_Hsq_Sc}
\Res_{M,\nu,0}^{\rm loc}(-8)+2\Res_{M,0}^{\rm loc}(-8)+\frac{\pi^2}{384}\left(-\Delta {|H|}^2+4\Delta\Sc\right)
\end{equation}
can be expressed in terms of the derivatives of $f$ of order up to $3$. 
On the other hand, \eqref{f***-terms_R_M_nu_x_-8}, \eqref{f***-terms_R_M_x_-8}, \eqref{Delta_Sc_at_0-3} and \eqref{Delta_H_at_0_third} imply that the $f_{\ast\ast\ast}$-terms of \eqref{Res_Res_Delta_Hsq_Sc} are given by 
\begin{equation}\label{app_GW_modified_residues}
\frac{\pi^2}{192}\bigg(
\sum_i\ip{f_{iii}}{f_{iii}}
+2\sum_{j\ne k}\ip{f_{jjk}}{f_{kkk}}
+\sum_{j\ne k}\ip{f_{jkk}}{f_{jkk}}
+2\sum_{k<\ell, j\ne k,\ell}
\ip{f_{jkk}}{f_{j\ell\ell}}
\bigg). 
%
\end{equation}

Now the conclusion follows from \eqref{f***-term_GW} and \eqref{app_GW_modified_residues}. 
\end{proof}

\subsubsection{Expression by residues and curvature energies}\label{subsub_GW_residues_4-hypersurface}
%
%
We give an explicit expression of \eqref{GW-residues} 
when $M$ is a hypersurface. 

\begin{theorem}\label{prop_GW-residues=W+PCE}
When $M$ is a $4$-dimensional closed hypersurface \eqref{GW-residues} is given by 
\begin{equation}\label{GW-residues=W+PCE}
\mathcal{E}(M)-\frac{3}{2\pi^2}\big(\Res_{M,\nu}(-8)+2\Res_M(-8)\big)
=-\frac1{2048}\biggl(12\int_M{|W|}^2\,dv+5Z(M)\biggr),
\notag
\end{equation}
where $W$ is the Weyl tensor and $Z(M)$ is the \M invariant principal curvature energy given by \eqref{Z(M)}.  
\end{theorem}

\begin{proof}
This can be shown by direct computation. 
By \eqref{grad_H_sq_m=4}, \eqref{GW-hypersurface}, and \eqref{app_GW_modified_residues}, we have only to extract $\ka_i$ terms from 
\[
\frac1{128}\biggl(-\Bigl(\sum_j\ka_j^{\,2}\Bigr){\Bigl(\sum_l\ka_l\Bigr)}^2
+\frac7{16}\,{\Bigl(\sum_i\ka_i\Bigr)}^4\biggr)
-\frac{3}{2\pi^2}\left(\Res_{M,\nu,0}^{\rm loc}(-8)+2\Res_{M,0}^{\rm loc}(-8)+\frac{\pi^2}{384}\left(-\Delta {|H|}^2+4\Delta\Sc\right)\right)
\]
using \eqref{R_kappa}, \eqref{R_nu_kappa}, \eqref{Delta_H_m=4} - \eqref{Delta_Sc_3}, and compare the result with \eqref{Weyl_hypersurface} and \eqref{PCE}. 
We remark that $\Delta H^2$ can be computed either from \eqref{Delta_H_sq_gen} or from $\Delta H^2=2(H\Delta H+{|\nabla H|}^2)$, \eqref{Delta_H_m=4} and \eqref{grad_H_sq_m=4}. 
\end{proof}

\section{Applications and miscellaneous topics}

\subsection{Variational problems and minimizers}
It is in general difficult to study variations of the energy and residues. 
In this subsection we agree that $M$ is an $m$ dimensional closed submanifold of $\RR^n$. 
Let us deal with the scale invariant (possibly \M invariant) energy or residues, which seem to be particularly interesting. 

Suppose $m=1$. The energy $E_K(-2)$ is the integral of a quantity that is determined by the global data of knots, which is the cause of the difficulty. See \cite{B11,BRS,BV,Gv,He,IN14,IN15,R}. 
As for the energy minimizers, round circles give the minimum value of $E_M(-2)$. This was shown when codimension is $2$ by \cite{FHW}. The general case follows from the ``wasted length'' argument by Doyle and Schramm (see for example \cite{AS}), which holds regardless of the codimension. On the other hand, $E_{M,\nu}(-2)$ when codimension is $2$ is essentially equal to the ``energy'' studied in \cite{OS1}, which is not bounded above nor below. 

Suppose $m=2$ and the codimension is $1$. 
The residue is essentially Willmore energy (modulo the Euler characteristic). 
Round spheres give the minimum value of $\Res_{M}(-4)=\frac\pi8\int_M(\ka_1-\ka_2)^2\,dv$ (cf. \eqref{residue_surface_2}). 
On the other hand, Willmore's results claiming that round spheres minimize $\int_M (\ka_1+\ka_2)^2\,dv$ implies that 
\[
\Res_{M,\nu}(-4)=-\frac\pi8\int_M 3\ka_1^2+3\ka_2^2+2\ka_1\ka_2\,dv
=-\frac\pi8\biggl(
3\int_M (\ka_1+\ka_2)^2\,dv+16\pi g-16\pi
\biggr),
\]
where $g$ is the genus of $M$, attains the maximum value at round spheres. 

Suppose $m=4$ and the codimension is $1$. 
Let us consider $2\Res_M(-8)+\Res_{M,\nu}(-8)$ for $4$-dimensional closed hypersurfaces in $\RR^5$. 
Theorem \ref{prop_GW-residues=W+PCE} implies 
\[
\Res_{M,\nu}(-8)+2\Res_M(-8)
=\frac{2\pi^2}{3}\mathcal{E}(M)
+\frac{\pi^2}{3072}\biggl(12\int_M{|W|}^2\,dv+5Z(M)\biggr). 
\]
Proposition 1.2 of \cite{GR} claims that the second variation of the Graham-Witten energy $\mathcal{E}$ at a round sphere $S^4$ is nonnegative, and is positive in directions transverse to the orbit of the conformal group. See also Theorem 1.2 of \cite{T}. 
On the other hand, Corollary \ref{cor_pos_W_PCE} implies that the second variation of $12\int_M{|W|}^2\,dv+5Z(M)$ at a round sphere is also nonnegative, and is positive in directions transverse to the orbit of the \M group. 
Therefore we have 
\begin{theorem}\label{prop_second_var_residue}
The second variation of $\Res_{M,\nu}(-8)+2\Res_M(-8)$ at a round sphere is nonnegative, and is positive in directions transverse to the orbit of the \M group. 
\end{theorem}

The author does not know whether standard spheres give the minimum or maximum values of the ($\nu$-weighted) energy or residue at $z=-2m$ among all the $m$-dimensional closed connected submanifolds of Euclidean spaces. 

Numerical experiment by Maple (Figure \ref{RM_m8}) suggests that round $4$-spheres are not neither minimum nor maximum point of $\Res_{M}(-8)$ among $4$-dimensional closed hypersurfaces. 

\subsection{Characterization (identification) of manifolds}

Since residues are similar to spectra of Laplacian as we saw in Subsection \ref{subsect_spectrum}, it is tempting to ask Kac's question about spectra:``Can one hear the shape of a drum?" for the residues, and more generally, for the meromorphic energy function as well, namely,``Can one identify a space by the residues or by the meromorphic energy function?". 
If two manifolds are isometric then the meromorphic energy functions are equal, which implies that the residues are equal, although the inverse implications do not hold in general as is explained in what follows. Let us deal with only submanifolds of Euclidean space here. 

The fact that the inverse of the first implication does not hold can be reduced to classical results by Mallows and Clark \cite{MC} and Caelli \cite{C} as was explained in \cite{Onew}. 
As for the inverse of the second implication, we give an example of a pair of manifolds with the same residues that have different meromorphic energy functions. 
\begin{lemma}\label{lem_diam}
There is a pair of hexagons with all interior angles at $2\pi/3$ with the same area and the same circumference and different diameters. 
\end{lemma}
\begin{proof}
Let $H(a,b,c)$ be a hexagon with all interior angles at $2\pi/3$ with edge length $a,b,c,a,b,c$ in this cyclic order. 
It has area $(\sqrt3/2)(ab+bc+ca)$ and circumference $2(a+b+c)$. 
Suppose $H(a,b,c)$ and $H(1,2,3)$ have the same area and the same circumference. Then $a,b$ and $c$ can be expressed by 
\[
a(\theta)=2+\frac{\sqrt3}3\cos\theta+\sin\theta, \>
b(\theta)=2+\frac{\sqrt3}3\cos\theta-\sin\theta, \>
c(\theta)=2-2\frac{\sqrt3}3\cos\theta
\]
for some $\theta$. The diameter of $H(a,b,c)$ is equal to $\sqrt{(4\sqrt3+55)/3}\approx 4.54$ when $\theta=0$ whereas $\sqrt{21}\approx 4.58$ when $\theta=\pi/6$. 
\end{proof}

We remark that the hexagons $H(a,b,c)$ are included in what Waksman calls the non-generic case. 
\begin{proposition}
\begin{enumerate}
\item There are smooth closed $1$-dimensional manifolds $M_1$ and $M_2$ with the same residues but different meromorphic energy functions. 
\item There are compact bodies $\dom_1$ and $\dom_2$ in $\RR^2$ with the same residues but different meromorphic energy functions. 
\end{enumerate}
\end{proposition}
\begin{proof}
Let $\dom(\theta)$ be a smooth convex body in $\RR^2$ which can be obtained from a hexagon $H(a(\theta), b(\theta), c(\theta))$ by smoothing the vertices in the same way in a small neighbourhood at each vertex. 
We assume that the smoothing is carried out in a symmetric way with respect to the two adjacent edges. 
Let $M(\theta)$ be the boundary of $\dom(\theta)$. 
Then $M(\theta)$ has the same residues for any $\theta$, and so does $\dom(\theta)$. 
Recall for $X=M$ or $\dom$, 
\[
\lim_{q\to+\infty}\left(\frac{B_X(q)}{{\textrm{Vol}(X)}^2}\right)^{1/q}
=\lim_{q\to+\infty}\left(\frac1{\textrm{Vol}(X\times X)}\iint_{X\times X}{|x-y|}^q\,dv(x)dv(y)\right)^{1/q}
=\sup_{x,y\in X}|x-y|
=\textrm{diam}(X).
\]
Now the proposition follows from Lemma \ref{lem_diam}. 
\end{proof}

On the other hand, in special cases the characterization problem has a positive solution. The author showed that 
if we restrict ourselves to the closed subsets and compact bodies in Euclidean spaces with suitable regularity, round balls can be characterized by the residues and round circles and $2$-spheres can be identified by the meromorphic energy function \cite{Onew}. 
The characterization of balls is a consequence of isoperimetric inequality since the first two residues are the volumes of $\dom$ and $\pO$. 
Here we show that round circles can be identified by the residues. 
\begin{proposition}
Let $X$ be either a compact body or a closed submanifold of a Euclidean space. 
Suppose $X$ is of class $C^5$ and the meromorphic energy function $B_X(z)$ has poles at $z=-1$ and $-3$ on $\R z\ge-3$ with $\Res_X(-1)=\Res_{S^1}(-1)$ and $\Res_X(-3)=\Res_{S^1}(-3)$. Then $X$ is isomorphic to a round unit circle $S^1$. 
\end{proposition}

We remark that we need one more regularity than in Theorem 3.4 of \cite{Onew} to guarantee the existence of the two residues. 
\begin{proof}
(i) From the location of the poles we know that $X$ is a union of closed curves $C_i$ $(1\le i\le \ell)$ for some $\ell\in\NN$. \\
\indent
(ii) Since $\Res_X(-1)=\Res_{S^1}(-1)$ 
\[
\sum_{i=1}^\ell L(C_i)=2\pi,
\]
where $L(C_i)$ is the length.  
\\
\indent
(iii) From $\Res_X(-3)=\Res_{S^1}(-3)$ we have 
\[\sum_{i=1}^\ell\int_{C_i}\kappa^2\,ds=2\pi. \]
Schwarz's inequality implies
\[
\left(\int_{C_i}\kappa\,ds\right)^2\le L(C_i)\cdot \int_{C_i}\kappa^2\,ds,
\]
where the equality holds if and only if $\kappa$ is constant, 
and Fenchel's inequality (\cite{Fe}, cf. \cite{Su}) implies 
\[
\int_{C_i}\kappa\,ds\ge2\pi,
\]
where the equality holds if and only if $C_i$ is a planar oval. 
From all the above equalities and inequalities we know that $\ell=1$ and $\kappa$ is constant, which completes the proof. 
\end{proof}
We remark that we do not have a proof that $S^2$ can be characterized by the residues for the moment since our proof in \cite{Onew} need information of the diameter.

\smallskip
The auther conjectures that odd dimensional round spheres can be identified by the energy $E_M(-2m)$ as well. 
Unfortunately we do not have any results for the characterization problem of compact Riemannian manifolds in terms of the meromorphic energy function or the residues (cf. \cite{GV}).

\subsection{Computation of volume of extrinsic balls by residues}
%
Suppose $M$ is an $m$ dimensional closed submanifold of $\RR^n$. 
Calculation of the volume of the extrinsic ball is generally very complicated, but it helps if we use two expressions of the residues 
as explained in Subsections \ref{subsubsect_extrinsic_balls} and \ref{method_graph}, 
one in terms of the coefficients of the series expansion of the volume of the extrinsic ball, and the other in terms of the derivatives of the function used to express the manifold as a graph. 
Karp and Pinsky \cite{KP} obtained the first two coefficients of the series expansion of the volume of a extrinsic ball \eqref{KP-expansion}. 
\begin{proposition}
Let $M$ be a closed surface in $\RR^3$. 
The third coefficient, which is the coefficient of $t^6$, of the series expansion of the volume of a extrinsic ball with radius $t$ can be expressed in terms of the second fundamental form $h$ as
\[
\frac{\pi}{9216}\left(
-9(h_{ii})^4
+36(h_{ij}h_{ij})^2
+64h_{ij;k}h_{ij;k}
-24(h_{ii})h_{jj;kk}
+72h_{ij}h_{ij;kk}
+24h_{ij}h_{kk;ij}
\right),
\] 
where $h_{ij;kl}=\partial_k\partial_lh_{ij}$. 
We remark that we used the Einstein summation convention.
\end{proposition}

\begin{proof}
\eqref{RsubMrhoddtpsix} implies that the local residue at $z=-6$ is equal to $6$ times the coefficient of $t^6$ of the series expansion of $\psi_x(t)$ which is the volume of a extrinsic ball with radius $t$ (cf. \eqref{psiMrhox}). 
Now the conclusion follows from Fuller and Vemuri's result (\cite{FV} Theorem 4.1). 
\end{proof}

\noindent
We remark that $h_{ij}$ is given by \eqref{preliminaries_gij_hij_H_m=4}.

\subsection{Appendix: \M invariant curvature energy of $4$-hypersurfaces}\label{appendix}

Let $M$ be a $4$-dimensional closed hypersurface in $\RR^5$. 
In this case ${|W|}^2$ and $X$ given by \eqref{Weyl_tensor} and \eqref{def_X} can be expressed by the principal curvatures $\ka_1,\dots,\ka_4$ as 
\begin{eqnarray}
{|W|}^2
&=&\displaystyle \frac1{6}\sum_{\{i,j,k,l\}=\{1,2,3,4\}}(\ka_{i}-\ka_{j})(\ka_{j}-\ka_{k})(\ka_{k}-\ka_{l})(\ka_{l}-\ka_{i}), \notag \\[1mm]
&=&\displaystyle \frac43\sum_{j<k}\ka_j^2\ka_k^2
-\frac43\sum_{j<k, i\ne j,k}
\ka_i^2\ka_j\ka_k+8\,\ka_1\ka_2\ka_3\ka_4 \label{Weyl_hypersurface}\\[1mm]
X&=& 6 \, \ka_1\ka_2\ka_3\ka_4. \notag
\end{eqnarray}

\begin{theorem}\label{thm_PCE}
Let $\sigma(\ka_1,\dots,\ka_4)$ be a symmetric polynomial of principal curvatures. 
Then 
$\displaystyle \int_M\sigma(\ka_1,\dots,\ka_4)\,dv$ is \M invariant if and only if $\sigma(\ka_1,\dots,\ka_4)$ is a linear combination of ${|W|}^2$, $X$ and 
\begin{equation}\label{PCE}
q(\ka_1,\dots,\ka_4)=\frac12\sum_{\{i,j,k,l\}=\{1,2,3,4\}} (\ka_i-\ka_j)^2(\ka_i-\ka_k)(\ka_i-\ka_l). 
\end{equation}
\end{theorem}

\begin{proof}
``If'' part. As was stated in the proof of Proposition \ref{prop_conf_inv_RMminus8}, ${|W|}^2dv$ is a pointwise conformal invariant and $X$ is an integral topological invariant. 
\eqref{lem323} implies that $q(\ka_1,\dots,\ka_4)dv$ is \M invariant pointwise.

\smallskip
``Only if'' part. 
First note that a symmetric polynomial is generated by 
$\ka_i^4, \ka_i^3\ka_j, \ka_j^2\ka_k^2, \ka_i^2\ka_j\ka_k$ and $\ka_1\ka_2\ka_3\ka_4$. Since 
\[
\sum_{i\ne j<k\ne i}\ka_i^2\ka_j\ka_k=\sum_{j<k}\ka_j^2\ka_k^2
-\frac34{|W|}^2+X, \qquad \ka_1\ka_2\ka_3\ka_4=\frac16X, 
\]
we have only to consider a symmetric polynomial $\sigma$ of the form 
\begin{equation}\label{specail_p}
\sigma(\ka_1,\dots,\ka_4)=c_1\sum_i\ka_i^4+c_2\sum_{i\ne j}\ka_i^3\ka_j+c_3\sum_{j<k}\ka_j^2\ka_k^2. 
\end{equation}
Since 
\[\begin{array}{rcl}
q(\ka_1,\dots,\ka_4)&=&\displaystyle 3\sum_i\ka_i^4-4\sum_{i\ne j}\ka_i^3\ka_j+2\sum_{j<k}\ka_j^2\ka_k^2+4\sum_{i\ne j<k\ne i}\ka_i^2\ka_j\ka_k-24\ka_1\ka_2\ka_3\ka_4 \\[6mm]
&=&\displaystyle 3\sum_i\ka_i^4-4\sum_{i\ne j}\ka_i^3\ka_j+6\sum_{j<k}\ka_j^2\ka_k^2-3{|W|}^2.
\end{array}
\]
it is enough to show that if $\displaystyle \int_M\sigma(\ka_1,\dots,\ka_4)\,dv$, where $\sigma$ is given by \eqref{specail_p}, is \M invariant, then $(c_1,c_2,c_3)=c(3,-4,6)$ for some $c\in\RR$. 

\smallskip
Let us use an $a$-hyper-spheroid (cf. Definition \ref{def_hyper-spheroid}). 
Since \eqref{principal_curvatures_S_a} and \eqref{principal_curvatures_widetilde_S_a} imply that $\ka_i$ and $\tilde\ka_i$ depend only on $\theta_1$, if $\displaystyle \int_M\sigma(\ka_1,\dots,\ka_4)\,dv$ is \M invariant then by \eqref{vol_ele_spheroid} and \eqref{vol_element_widetilde_S_a} we have 
\begin{equation}\label{Mob_inv_CPE_hyper-spheroid}
\int_0^\pi 
\biggl[\sigma(\ka_1,\dots,\ka_4)-\frac{\sigma(\tilde\ka_1,\dots,\tilde\ka_4)}{{(a^2\cos^2\theta_1+\sin^2\theta_1)}^4}
\biggr]\sqrt{a^2\sin^2\theta_1+\cos^2\theta_1}\,\sin^3\theta_1\, d\theta_1=0.
\end{equation}
The above integrand is a linear combination of 
\[
\frac{\sin^3\theta_1\cos^j\theta_1}{{(\cos^2\theta_1+a^2\sin^2\theta_1)}^{9/2}{(a^2\cos^2\theta_1+\sin^2\theta_1)}^4} \qquad (j=0,2,4,\dots,12).
\]
Computation using Maple implies that all the antiderivatives of the above terms can be given by elementary functions including $\arctan$. 
Substituting $a=\sqrt2$ and $a=\sqrt3$ to \eqref{Mob_inv_CPE_hyper-spheroid} we have 
\[
\begin{array}{rccl}
\displaystyle -\frac{1}{35}\,(422\,c_1+726\,c_2+273\,c_3)&\displaystyle +\frac{8\sqrt3\,\pi}{81}\,(22\,c_1+39\,c_2+15\,c_3)&=&0, \\[4mm]
\displaystyle -\frac{1}{315}\,(619618c_1+1108164c_2+428967c_3)&\displaystyle +\frac{729\sqrt2\,\arctan\big(2\sqrt2\,\big)}{4}(6c_1+12c_2+5c_3)&=&0,
\end{array}
\]
which implies $(c_1,c_2,c_3)=c(3,-4,6)$. 
\end{proof}

\begin{definition}\label{def_Z(M)} \rm 
We define the {\em \M invariant principal curvature energy} of a $4$-dimensional hypersurface in $\RR^5$ by 
\begin{equation}\label{Z(M)}
Z(M)=\int_Mq(\ka_1,\dots,\ka_4)\,dv, 
\end{equation}
where $q(\ka_1,\dots,\ka_4)$ is given by \eqref{PCE}. 
\end{definition}

\begin{lemma}\label{lem_W_PCE}
\begin{enumerate}
\item $q(\ka_1,\dots,\ka_4)$ is non-negative for any $\ka_1,\dots,\ka_4$ and is equal to $0$ if and only if $\{\ka_1,\dots,\ka_4\}=\{\ka,\ka,\ka',\ka'\}$ for some $\ka$ and $\ka'$. 
Hence $Z(M)$ is non-negative. 
\item ${|W|}^2$ is non-negative for any $\ka_1,\dots,\ka_4$ and is equal to $0$ if and only if $\{\ka_1,\dots,\ka_4\}=\{\ka,\ka,\ka,\ka'\}$ for some $\ka$ and $\ka'$. 
\end{enumerate}
\end{lemma}

\begin{proof}
(1) Assume $\ka_1,\dots,\ka_4$ are not identical. 
It is enough to consider $q(0,x,y,1)$ for $0\le x,y\le 1$. 
Put $F(x,y)=q(0,x,y,1)$. 
Calculating the values of $F$ on the boundary $\partial [0,1]^2$, and using 
\[
\pd{F}{x}-\pd{F}{y}=16(x-y)\left(\left(x-\frac12\right)^2+\left(y-\frac12\right)^2-\frac12\right), \quad \pd{F}{x}(x,x)=8x-4,
\]
we conclude that $F$ has critical points only at $(1/2,1/2)$ which is a saddle point where $F$ takes $1$ and at $(0,1)$ and$(1,0)$ which are local minimum points where $F$ takes $0$, which turns out to be the global minimum. 

\smallskip
(2) It can be shown similarly from \eqref{Weyl_hypersurface}. 
\end{proof}

\begin{corollary}\label{cor_pos_W_PCE}
Suppose $M$ is a $4$-dimensional connected closed hypersurface in $\RR^5$. 
If $\a, \beta>0$ then 
\[
\a\int_M{|W|}^2\,dv+\beta Z(M)\ge 0
\]
and the equality holds if and only if $M$ is a sphere. 
\end{corollary}

\begin{proof}
Suppose the equality holds. 
Lemma \ref{lem_W_PCE} shows that $M$ is totally umbilical. Since $M$ is closed and connected, $M$ is a sphere. 
\end{proof}

Jun O'Hara

Department of Mathematics and Informatics,Faculty of Science, 
Chiba University

1-33 Yayoi-cho, Inage, Chiba, 263-8522, JAPAN.  

E-mail: ohara@math.s.chiba-u.ac.jp

\end{document}